\documentclass[10pt]{amsart}
\usepackage{kantlipsum} % for mock text

\usepackage[margin=1.25 in]{geometry} %changed margin from 2in to 1in

\usepackage{graphicx}
\usepackage{url}
\usepackage[pagebackref=false,colorlinks]{hyperref}
\hypersetup{citecolor={rgb,256:red,0;green,80;blue,170},
linkcolor={cyan}}

\usepackage{mathptmx}     
\usepackage{amsmath,amsfonts,amsthm,amssymb,color}
\usepackage{mathtools}
\usepackage{tikz-cd}

\usepackage{cleveref}

\usepackage{xparse}
\usepackage{calrsfs}
\DeclareMathAlphabet{\pazocal}{OMS}{zplm}{m}{n}

\usepackage{spectralsequences} 

\NewDocumentCommand{\tens}{e{_^}}{%
  \mathbin{\mathop{\otimes}\displaylimits
    \IfValueT{#1}{_{#1}}
    \IfValueT{#2}{^{#2}}
  }%
}

\newcommand{\F}{\mathbb{F}_2}
\newcommand{\Fp}{\mathbb{F}_p}
\newcommand{\free}{\mathrm{Free}}
\newcommand{\id}{\mathrm{id}}
\newcommand{\E}{\mathbb{E}}
\newcommand{\Sp}{\mathbb{S}}
\newcommand{\p}{\mathbf{P}}
\newcommand{\B}{\mathrm{Bar}_{\bullet}}
\newcommand{\taq}{\overline{\mathrm{TAQ}}}

\newcommand{\R}{\pazocal{R}}

\newcommand{\Lie}{\mathrm{Lie}}
\newcommand{\spla}{\Lie^{\pi}_{\Fp, \E_{\infty}}}

\newcommand{\N}{\mathbb{N}}
\newcommand{\Mod}{\mathrm{Mod}}

\newtheorem{theorem}{Theorem}[section]
\newtheorem{corollary}[theorem]{Corollary}

\newtheorem{lemma}[theorem]{Lemma}

\newtheorem{proposition}[theorem]{Proposition}

\theoremstyle{definition}
\newtheorem{definition}[theorem]{Definition}

\newtheorem{remark}[theorem]{Remark}
\newtheorem{notation}[theorem]{Notation}
\newtheorem{construction}[theorem]{Construction}

\begin{document}
\title{Operations on spectral partition Lie algebras and TAQ cohomology}
\author{Adela YiYu Zhang}
\address{Copenhagen Centre for Geometry and Topology, Copenhagen University, Universitetsparken 5, DK-2100 Copenhagen}
\email{yz@math.ku.dk}

\begin{abstract}
  We determine all natural operations and their relations on the homotopy groups of spectral partition Lie algebras, which coincide with  $\mathbb{F}_p$-linear topological Andr\'{e}-Quillen cohomology operations at any prime. We construct unary operations and a shifted restricted Lie algebra structure on the homotopy groups of spectral partition Lie algebras.  Then we prove a composition law for the unary operations, as well as a compatibility condition between unary operations and the shifted Lie bracket with restriction up to a unit for the restriction. Comparing with Brantner-Mathew's result on  the ranks of the homotopy groups of free spectral partition Lie algebras, we deduce that these generate all natural operations, thereby also recovering unpublished results of Kriz and Basterra-Mandell on $\mathbb{F}_p$-linear TAQ cohomology operations.  As a corollary, we determine the structure of natural operations on mod $p$ $\mathbb{S}$-linear TAQ cohomology.
\end{abstract}
\maketitle
\setcounter{tocdepth}{1}
\tableofcontents
\section{Introduction} 
 Spectral partition Lie algebras are the key objects in the emerging field of spectral formal moduli problems in characteristic $p$. Work of Brantner and Mathew \cite{bm} showed that  there is an equivalence of $\infty$-categories between spectral formal moduli problems over $\Fp$ and spectral partition Lie algebras, generalizing the characteristic 0 phenomenon studied by Drinfeld \cite{drinfeld}, Pridham \cite{pridham}, Lurie \cite{lurie11}, and many others.  In particular, this equivalence establishes spectral partition Lie algebras over $\Fp$ as divided power algebras Koszul dual to non-unital $\E_\infty$-$\Fp$-algebras.
 
 Since spectral partition Lie algebras are algebras over a certain  monad $\Lie^{\pi}_{\Fp,\E_\infty}$, the homotopy groups of  free spectral partition Lie algebras $\Lie^{\pi}_{\Fp,\E_\infty}(\Sigma^{i_1}\Fp\oplus\cdots\oplus\Sigma^{i_k}\Fp)$ parametrize all natural $k$-ary operations on the homotopy groups of spectral partition Lie algebras as $(i_1,\ldots,i_k)$ varies. In \cite[Theorem 1.20]{bm}, Brantner and Mathew obtained  bases for homotopy groups of free spectral partition Lie algebras  on single generators via an isotropy spectral sequence studied by Arone, Dwyer, and Lesh \cite{ADL}. Then they propagated the result to multiple generators by means of an EHP sequence and a decomposition of the partition complex developed by Arone and Brantner in \cite{young}. Nonetheless, this method did not provide explicit descriptions of the nature of the operations, nor were the relations among the operations clarified. 
 
  On the other hand,  $\Fp$-linear TAQ cohomology $\mathrm{TAQ}^*_{\Fp}(R;\Fp)$ of $\E_\infty$-$\Fp$-algebras $R$ (\Cref{def: FpTAQ}), first constructed by Kriz and Basterra \cite{taq}, has representing objects trivial square-zero extensions \cite[1.8.2]{lawson}. Hence reduced $\Fp$-linear TAQ cohomology groups of trivial algebras $\Fp\oplus\Sigma^{i_1}\Fp\oplus\cdots\oplus\Sigma^{i_k}\Fp$ parametrize all natural $k$-ary operations. By \cite[Proposition 5.35]{bm}, there is an isomorphism $$\pi_*(\Lie^{\pi}_{\Fp,\E_\infty}(\Sigma^{i_1}\Fp\oplus\cdots\oplus\Sigma^{i_k}\Fp))\oplus\Fp\cong \mathrm{TAQ}^{-*}_{\Fp}(\Fp\oplus\Sigma^{-i_1}\Fp\oplus\cdots\oplus\Sigma^{-i_k}\Fp;\Fp). $$ Hence natural operations on  the homotopy groups of spectral partition Lie algebras agree with cohomology operations on the (reduced) $\Fp$-linear TAQ cohomology of $\E_\infty$-$\Fp$-algebras. In unpublished work, Kriz computed the $\F$-linear TAQ cohomology on a connective generator in \cite[Theorem 1.10]{kriz}. %Arone and Mahowald computed the dimensions of unary operations on  $\Fp$-linear TAQ cohomology of connective objects in \cite[Theorem 3.16 ]{am}.
  Around the same time, Basterra and Mandell announced a computation of unary operations and their relations as the Koszul dual to Dyer-Lashof operations on $\Fp$-linear TAQ cohomology of connective objects for $p>2$  and observed a shifted restricted Lie algebra structure, but a proof never appeared. 
 
\

In this paper, we use a dual bar spectral sequence and the machinery of classical Koszul duality developed by Priddy \cite{priddy} to identify the structure of weight powers of $p$ operations on the homotopy groups of spectral partition Lie algebras and $\Fp$-linear TAQ cohomology. Roughly speaking, this structure is given by a collection of unstable Ext groups over the Dyer-Lashof algebra, with composition product given by a sheared Yoneda product. The verification of the law of composition makes use of a general result of Brantner  \cite{brantner} that demonstrates the compatibility of the algebraic Koszul duality on the $E^2$-page of the (dual) bar spectral sequence with the $\infty$-categorical monadic Koszul duality that the $E^\infty$-page assembles to. These operations are stable under suspension, and in the colimit we obtain the universal algebra of unary operations, which is the Koszul dual of the  Dyer-Lashof algebra.

Then we construct a shifted Lie bracket on the homotopy groups of spectral partition Lie algebras and use a homotopy fixed points spectral sequence to detect a restriction map on this shifted Lie bracket. The restriction on an odd degree homotopy class when $p>2$ and any homotopy class when $p=2$ coincides with the bottom unary operation up to a unit. Furthermore, brackets of unary operations that are not iterations of the restriction always vanish. Comparing with Brantner and Mathew's additive computation of the homotopy groups of the free spectral partition Lie algebras on finitely many generators \cite[Theorem 1.20]{bm}, we conclude that these are all the natural operations and obtain the target category in the sense that it records all algebraic structure on the homotopy groups of spectral partition Lie algebras.
 
  As an immediate corollary, we determine the structure of operations on the mod $p$ TAQ cohomology $\mathrm{TAQ}^*_{\mathbb{S}}(-;\Fp)$ of $\E_\infty$-$\mathbb{S}$-algebras. The  $\mathbb{S}$-linear mod $p$ TAQ cohomology operations consist of $\Fp$-linear TAQ cohomology operations, as well as mod $p$ Steenrod operations. The relations among the unary operations are given by the Adem relations and the Nishida relations on mod $p$ cohomology \cite{nishidaodd}.

\subsection{Statement of results}
In section \ref{section2}, we summarize recent results in \cite{bm} on spectral partition Lie algebras over $\Fp$. These are algebras over a certain monad $\Lie^{\pi}_{\Fp,\E_\infty}$ on the category of $\Fp$-modules. Then we describe the relation between spectral partition Lie algebras and  the reduced $\Fp$-linear TAQ spectrum $$\taq(R)\simeq |\B(\id, \mathbb{E}_\infty^{\mathrm{nu}}\tens \Fp, R)|$$ of non-unital $\E_\infty$-$\Fp$-algebras $R$, which is a variant of the TAQ spectrum  constructed by Basterra \cite[\S 5]{taq}. Here $\mathbb{E}_\infty^{\mathrm{nu}}\tens \Fp$ stands for the monad associated with the free non-unital $\E_\infty$-$\Fp$-algebra functor.  The $n$th reduced $\Fp$-linear TAQ cohomology group of $R$ is given by $$\taq^n(R)=[\Sigma^{-n}\taq(R), \Fp]_{\Mod_{\Fp}}\cong \pi_{-n}(\taq(R)^{\vee}).$$  It follows from   \cite[Proposition 5.35]{bm} that there is an isomorphism $$\pi_m(\Lie^{\pi}_{\Fp,\E_\infty}(\Sigma^{i_1}\Fp\oplus\cdots\oplus\Sigma^{i_k}\Fp))\cong \taq^{-m}(\Sigma^{-i_1}\Fp\oplus\cdots\oplus\Sigma^{-i_k}\Fp) $$  for any $m$ and tuple $(i_1,\ldots,i_k)$ of integers. The left hand side is the group of natural transformations $$\prod^k_{l=1}\pi_{i_l}(-)\rightarrow\pi_{m}(-)$$ of functors from the category of spectral partition Lie algebras to Sets, whereas the right hand side is the group of cohomology operations (or natural transformations) $$\prod^k_{l=1}\taq^{-i_l}(-)\rightarrow\taq^{-m}(-).$$ Hence the structure of natural operations on the homotopy groups of spectral partition Lie algebras over $\Fp$ coincide with that on  $\taq^*(A)$ for $A$ a non-unital $\E_\infty$-$\Fp$-algebra.

 In section \ref{section3}, we examine the bar spectral sequence
 \begin{equation}
     \widetilde{E}_{s,t}^2=\pi_s(\B(\id, \p_{\R},\pi_*(A) )_t \Rightarrow \pi_{s+t}(|\B(\id, \E^{\mathrm{nu}}_\infty\tens \Fp, A)|)=\taq_{s+t}(A)
 \end{equation}
and the dual bar spectral sequence 
\begin{equation}\label{sseq: dual bar}
   E_{s,t}^2=\pi_s(\B(\id, \p_{\R},\pi_*(A))^{\vee} )_t \Rightarrow \pi_{s+t}(|\B(\id, \E^{\mathrm{nu}}_\infty\tens \Fp, A|^{\vee})\cong\pi_{s+t}(\Lie^{\pi}_{\Fp,\E_\infty}(A^{\vee})) 
\end{equation}
for trivial $\E^{\mathrm{nu}}_\infty$-$\Fp$-algebras $A$ that are bounded above and finite type, i.e. each homotopy group is a finite $\Fp$-module.  Here $\p_{\R}$ is the monad associated to the free functor that first takes the graded polynomial algebra on the free unstable module over the Dyer-Lashof algebra,  then identifies $x^{\tens 2}$ with the bottom operation $Q^{|x|}$ if $p=2$ and $x^{\tens p}$ with the bottom operation $Q^{|x|/2}(x)$ if $p>2$ and $|x|$ is even, and finally imposes the Cartan formula. (See \Cref{def: polyR}.) When $A$ is a direct sum of shifts of $\Fp$ considered as a trivial $\E^{\mathrm{nu}}_\infty$-$\Fp$-algebra, the $E^2$-page of the dual bar spectral sequence (\ref{sseq: dual bar}) parametrizes natural operations and their relations on the \textit{Andr\'{e}-Quillen cohomology} $$\mathrm{AQ}^*_{\p_{\R}}(M):=\pi_{-*}(\B(\id, \p_{\R},M)^{\vee} )$$ of algebras $M$ over the monad $\p_{\R}$, whereas the $E^\infty$-page is isomorphic to the homotopy groups of the free spectral partition Lie algebra on $A^{\vee}$. 

In \Cref{section4}, we study the (dual) bar spectral sequence when $A$ is a finite sum of shifts of $\Fp$. The strategy is the following: first we find a suitable factorization of the indecomposables functor $Q^{\p_{\R}}_{\Mod_{\Fp}}$, which allows us to replace the $E^1$-page with a much smaller bi-simplicial object.  This strategy is a modification of the method in \cite[Proposition 4.19]{bhk} and was employed in \cite{me} to compute the mod $p$ Quillen homology of spectral Lie algebras. Then we compute the $E^2$-page via the (dual) Grothendieck spectral sequence as in \cite{behrensrezk}, whose $E^2$-page is the bigraded homotopy of a double bar complex. We deduce by sparseness in \Cref{E2unary} and \Cref{E2unaryodd} that the dual bar spectral sequence for $A=\Sigma^j\Fp$ and the associated Grothendieck spectral sequence both collapse on the $E^2$-page by sparseness when $p=2$ or $p>2$ and $j$ is odd. If $p>2$ and $j$ is even, we show in \Cref{cor: E2unaryeven} that both spectral sequences collapse at weight $p^k$ for all $k\geq 0$ by a degree comparison with the $\Fp$-basis of $\pi_*(\spla(A))$ given by Brantner-Mathew in \cite{bm}. In particular, the weight $p^k$ unary operations on the Andr\'{e}-Quillen cohomology $\mathrm{AQ}^*_{\p_{\R}}(M)$ are parametrized by a ringoid (\Cref{ringoidmod2} and \Cref{ringoidodd}) that encodes certain unstable Ext groups over the Dyer-Lashof algebra $\R$, which follows from the algebraic Koszul duality machinery developed by Priddy \cite{priddy}. This allows us to construct all unary operations of weights powers of $p$ on the homotopy groups of spectral partition Lie algebras in \Cref{unary}.

In \Cref{sec: Lie}, we produce a shifted  Lie bracket $$[-,-]:\pi_m(A)\tens \pi_n(A)\rightarrow\pi_{m+n-1}(A)$$ with restriction $x\mapsto x^{\{p\}}$ for any spectral partition Lie algebra $A$ that comes from a construction in $\Mod_{\Fp}(\mathrm{Sp})$. Equivalently, for any $\E^{\mathrm{nu}}_{\infty}$-$\Fp$-algebra $A$,  there is a shifted  Lie bracket with restriction
$$[-,-]: \taq^m(A)\tens \taq^n(A)\rightarrow \taq^{m+n+1}(A).$$ We also determine the interaction of this bracket with the unary operations of \Cref{unary} in \Cref{compatible}. The main tool is the associated homotopy fixed points spectral sequence 
\begin{equation}
        E^2_{s,t}=\bigoplus_{n} H^s\Big(\Sigma_n, \pi_t\big(\Lie_{dg}^{s}(n)\tens A^{\tens n}\big)\Big)\Rightarrow \bigoplus_{n}\pi_{t-s}\Big(\big(\Lie_{dg}^{s}(n)\tens A^{\tens n}\big)^{h\Sigma_n}\Big),
\end{equation}
which consists of a copy of the free shifted restricted Lie algebra on $\pi_*(A)$ along the line $s=0$ and degenerates in all cases and weights of interest.

Section \ref{section5} is devoted to determining the entire algebraic structure on the homotopy groups of spectral partition Lie algebras and mod $p$ TAQ cohomology groups. First we determine the composition law and the Adem relations among the weight $p$ unary operations. Inspired by the work of Brantner regarding the structure of operations on the Lubin-Tate theory of spectral Lie algebras in \cite[section 4]{brantner}, we encode unary operations on the homotopy groups of free spectral partition Lie algebras with a power ring $\pazocal{P}$ (\Cref{powerring}).
 Roughly speaking, this power ring arises from equipping the ringoid parametrizing weight $p^k$ unary operations on the $E^2$-page of the dual bar spectral sequence  with a twisted composition product. Note that our notion of power ring differs from that of Brantner in that we do not require bilinearity. This is to accommodate the identification of the bottom operation with the restriction map (\Cref{restricted}), which is nonadditive. However, on the associated graded of the bar filtration (and the filtration associated with the Grothendieck spectral sequence when $p=2$), this power ring does act additively, see \Cref{rmk: powerring}.

 Now we can state the main result of this paper, which records all natural operations and their relations on the homotopy groups of a spectral partition Lie algebra over $\Fp$ and the reduced TAQ cohomology of  $\E_\infty$-$\Fp$-algebras.

  \begin{theorem}(\Cref{unarymod2}, \Cref{compatible}, and \Cref{generation})\label{maintheorem1}
  \begin{enumerate}
      \item On a class $x$ of a degree $j$ in the homotopy groups of a spectral partition Lie algebra over $\F$, or equivalently the reduced TAQ cohomology of an $\E_\infty$-$\F$-algebra, there are weight $2$ operations  $R^{i}$ of degree $-i$ satisfying $i>-j$. The Adem relations are given by  $$R^a R^b(x)=\sum_{a+b-c\geq 2c,\,\, c\geq a-b}\binom{b-c-1}{a-2c} R^{a+b-c}R^c(x)$$
 for all $a,b\in\mathbb{Z}$ satisfying $b-j\leq a< 2b$ and $b>-j$. These operations are additive unless $i=-j+1$.

\item There is also a shifted Lie algebra structure with restriction map $x^{\{2\}}=R^{-|x|+1}$ on the homotopy groups of a spectral partition Lie algebra over $\F$. The restriction on a sum of classes $x$ and $y$ in different degrees is given by $$(x+y)^{\{2\}}= x^{\{2\}}+ y^{\{2\}}+ [x,y].$$ The bracket is compatible with the unary operations in the sense that $[y, \alpha (x)]=0$ for any homotopy class $x,y$ and  unary operation $\alpha$ of weight greater than one that is not an iteration of the restriction. 
\item 
The operations $R^i$ and the shifted Lie bracket with restriction generate all natural operations under the above relations. A basis for unary operations on a degree $j$ class $x$ is given by the collection of all monomials $R^{i_1}R^{i_2}\cdots R^{i_l}x^{\{2\}^r}$ such that $i_l>-2^rj+2^r$ and $i_m\geq 2i_{m+1}$ for $1\leq m<l$.

  \end{enumerate}
  \end{theorem}

\begin{theorem}(\Cref{unaryodd},  \Cref{compatible}, \Cref{cor: betabottomadditive}, and \Cref{generation})\label{maintheorem2}
Suppose that $p$ is an odd prime.
\begin{enumerate}
    \item 
   On a class $x$ of a degree $j$ in the homotopy groups of a spectral partition Lie algebra over $\Fp$, or equivalently the reduced TAQ cohomology of an $\E_\infty$-$\Fp$-algebra, there are  weight $p$ unary operations are given $\beta^\epsilon R^{i}$ of degree $-2(p-1)i-\epsilon$  for $\epsilon=0,1$ and $i>-j/2$. These operations are additive unless $R^{(-j+1)/2}$ is acting on an odd class in degree $j$. The Adem relations are given by 
    \begin{equation*}
   \beta R^a \beta R^b(x)=\sum_{a+b-c> pc, 2c>-j}(-1)^{a-c+1}\binom{(p-1)(b-c)-1}{a-pc-1} \beta R^{a+b-c} \beta R^c(x) 
\end{equation*}
 for all $a,b\in\mathbb{Z}$ satisfying $a \leq pb$, $2b>-j$, $2a>2(p-1)b-j$ ,
\begin{align*}
   R^a \beta R^b(x)=&\sum_{a+b-c\geq pc,2c>-j}(-1)^{a-c}\binom{(p-1)(b-c)}{a-pc} \beta P^{a+b-c} R^c(x)\\&-\sum_{a+b-c> pc, 2c>-j}(-1)^{a-c}\binom{(p-1)(b-c)-1}{a-pc-1}  R^{a+b-c}\beta R^c(x)
\end{align*}
for all $a,b\in\mathbb{Z}$ satisfying  $a \leq pb$, $2b>-j$, $2a>2(p-1)b+1-j$, and
\begin{equation*}
   \beta^\epsilon R^a R^b(x)=\sum_{a+b-c\geq pc, 2c>-j}(-1)^{a-c}\binom{(p-1)(b-c)-1}{a-pc} \beta^\epsilon R^{a+b-c}R^c(x)
\end{equation*}
for all $a,b\in\mathbb{Z}$ satisfying   $ a < pb$, $2b>-j$, $2a>2(p-1)b-j$, and $\epsilon\in\{0,1\}$.

\item  For all odd $j$ and $x$ a homotopy class in degree $j$, the restriction $x^{\{p\}}$ is the bottom operation $R^{(-j+1)/2}(x)$ up to a unit $\lambda_j$, i.e., $[y,\lambda_j R^{(-j+1)/2}(x)]=[[\cdots[[y,x],x]\cdots],x]$ for any class $y$, where bracketing with $x$ is iterated $p$ times on the right hand side. The restriction map on a sum of classes $x$ and $y$ in odd degrees $j\neq k$ is given by $$(x+y)^{\{p\}}=\lambda_j R^{(-j+1)/2}(x)+\lambda_k R^{(-k+1)/2}(y)+\sum_{i=1}^{p-1} \frac{s_i}{i}(x,y),$$ where $s_i$ is the coefficient of $t^{i-1}$ in the formal expression $\mathrm {ad} (tx+y)^{p-1}(x)$. Furthermore, $[y, \alpha (x)]=0$ for any homotopy class $x,y$ and  $\alpha$ a unary operation of weight greater than one, unless $x$ is in odd degree and $\alpha$ an iteration of the restriction.
\item The operations $\beta^{\epsilon}R^i$ and the shifted  Lie bracket generate all natural operations under the above relations. A basis for unary operations on a degree $j$ class with $j$ odd is given by all monomials $\beta^{\epsilon_1}R^{i_1}\beta^{\epsilon_2}R^{i_2}\cdots\beta^{\epsilon_l}R^{i_l}$ such that $2i_l>-j$ and $i_m\geq pi_{m+1}+\epsilon_{m+1}$ for $1\leq m<l$. If $j$ is even, a basis is given by $\beta^{\epsilon_1}R^{i_1}\beta^{\epsilon_2}R^{i_2}\cdots\beta^{\epsilon_l}R^{i_l}B^{\epsilon}$ such that $2i_l>-(1+\epsilon)j-\epsilon$  and $i_m\geq pi_{m+1}+\epsilon_{m+1}$ for $1\leq m<l$.
 \end{enumerate} 
 
  \end{theorem}

In order to verify the composition product and relations among the unary operations on  $\pi_*(\Lie^{\pi}_{\Fp,\E_{\infty}}(A))$ that comes from the power ring $\pazocal{P}$ (\Cref{powerring}), we make use of a general result of Brantner \cite[Theorem 3.5.1 and 4.3.2]{brantner}, which explains the compatibility of the structure of additive operations on the $E^2$-page of the (dual) bar spectral sequence with that on the homotopy groups of the monadic bar construction the $E^\infty$-page assembles to when the spectral sequence degenerates on the $E^2$-page at weights powers of $p$. 

Taking the colimit along the suspension map on the power ring $\pazocal{P}$, which is induced by the canonical map $\Sigma\spla(\Sigma^k\Fp)\rightarrow\spla(\Sigma^{k+1}\Fp)$, we observe in \Cref{algebra} that the universal unary operations form an algebra Koszul dual to the Dyer-Lashof algebra, which is isomorphic to the extended Steenrod algebra but with $Sq^0=0$. Indeed, the Dyer-Lashof algebra is isomorphic to the extended Steenrod algebra for restricted Lie algebras, cf. \cite{dual}. For filtration reasons, the restriction on the shifted Lie bracket is not visible on the $E^2$-page of the dual bar spectral sequence. Instead, we use a homotopy fixed points spectral sequence to detect the restriction map in \Cref{restricted}. Thus we recover and clarify the unpublished computations of Kriz and Basterra-Mandell on connective objects via a different method. 

  Therefore the target category for the homotopy groups of a spectral partition Lie algebra or the $\Fp$-linear TAQ cohomology is the category of $\pazocal{P}$-$\Lie^{s,\rho}$-\textit{algebras} (Definition \ref{sLieP}), i.e., $\Fp$-modules $L$ over the power ring $\pazocal{P}$ together with a shifted Lie bracket and a restriction map satisfying the  conditions in part (3) of Theorem \ref{maintheorem1} and \ref{maintheorem2}. %Denote by $\free^{\Lie^{s,\rho}_{\pazocal{P}}}$ the free $\pazocal{P}$-$\Lie^{s,\rho}$-algebra functor.
\begin{remark}
The shifted Lie brackets on the mod $p$ homology of spectral Lie algebras always vanish on unary operations of weight at least $p$, as was shown by Antol\'{i}n-Camerena \cite{omar} and Kjaer \cite{kjaer}. The difference lies in that all unary operations on the mod $p$ homology of spectral Lie algebras are additive, whereas on the homotopy groups of spectral partition Lie algebras, there is a non-additive unary operation, i.e., the restriction on the shifted Lie bracket.  See Remark \ref{comparebrackets}.
\end{remark}

 \begin{remark}\label{rmk: unit}
  To identify the unit $\lambda_j$ by which the bottom operation and the restriction on an odd class $x$ in degree $j$ differ when $p>2$, we expect that a chain-level understanding of the operations involved is necessary. A natural guess is that $\lambda_j$ is the sign by which the bottom Dyer-Lashof operation $\beta Q^{(j+1)/2}(x)$ on a degree $j$ class $x$ and its p-fold Massey-product differ. We hope to construct these operations on the explicit chain model of spectral partition Lie algebra obtained by Brantner, Campos, and Nuiten in \cite[Definition 4.43]{pd}.
 \end{remark}

\begin{remark}
Recently, Konovalov \cite{konovalov} computed the relations among unary operations on the odd primary homology groups of spectral Lie algebras by studying differentials in an algebraic Goodwillie spectral sequence and thus determined the entire structure of operations in the odd primary case.  The method in this paper suggests an alternative approach to settle the question. Namely, one could construct a  monad $\pazocal{C}$ that parametrizes divided power $\E^{\mathrm{nu}}_{\infty}$-$\Fp$-algebras, determine the structure of natural operations on the homotopy groups of algebras over $\pazocal{C}$, and feed it into the dual bar spectral sequence converging to $\pi_*(|\B(\id, \pazocal{C}, A)|^\vee)$ with $A$ any trivial algebra over $\pazocal{C}$ of finite type.
\end{remark}

As an immediate application, we obtain a computation of natural operations and relations up to the unit in \Cref{rmk: unit} on the mod $p$ TAQ cohomology $$\mathrm{TAQ}^*_{\Sp}(R;\Fp)=\mathrm{Map}_{\mathrm{Sp}}(|\B(\id, \E_\infty, R)|, \Sigma^*\Fp)$$ of $\E_\infty$-$\mathbb{S}$-algebras $R$, which is based on conversations with Tyler Lawson.

Since the functor $\mathrm{TAQ}_{\Sp}^i(-;\Fp)$ has representing object the trivial square-zero extension $\mathbb{S}\oplus \Sigma^{i}\Fp$ for all $i$, operations and relations are again parametrized by the mod $p$ TAQ cohomology on the trivial square-zero extensions $\mathbb{S}\oplus \Sigma^{i_1}\Fp\oplus\cdots\oplus \Sigma^{i_k}\Fp$. Using a base change formula to the reduced $\Fp$-linear TAQ cohomology, we deduce immediately from Theorem \ref{maintheorem1} and \ref{maintheorem2} the structure of natural operations on the  mod $p$ TAQ cohomology $\E_{\infty}$-$\mathbb{S}$-algebras. Note that the mod $p$ TAQ cohomology is in particular the mod $p$ cohomology of a spectrum, and hence is acted on by the mod $p$ Steenrod operations in the usual sense.
\begin{theorem} (Corollary \ref{slinear}, Proposition \ref{slinearrelations})
 For any tuple $(i_1,\ldots i_k)$ of integers, the $k$-ary cohomology operations $$\prod^k_{i=1}\mathrm{TAQ}^{i_l}_{\Sp}(-;\Fp)\rightarrow\mathrm{TAQ}^{m}_{\Sp}(-;\Fp).$$ are parametrized by the homological degree $-m$ part of $\free^{\Lie^{s,\rho}_{\pazocal{P}}}(\Sigma^{-i_1}\pazocal{A}\oplus\cdots\oplus \Sigma^{-i_k}\pazocal{A})$, where $\pazocal{A}$ is the Steenrod algebra graded homologically. All operations evaluate to zero on the unit except for scalar multiplication. The Steenrod operations commute with the bracket via the Cartan formula and the $\Fp$-linear TAQ cohomology operations via the Nishida relations on cohomology of the second extended power: 
 \begin{enumerate}
     \item For $p=2$ we have
$$Sq^a[x,y]=\sum_i[Sq^i(x),Sq^{a-i}(y)],$$
$$Sq^a R^{-|x|+1} (x) = \sum \binom{|x|-c}{a-2c} R^{a+|x|+1-c} Sq^c  (x) + \sum_{l<k, l+k=a} [Sq^l (x) , Sq^k( x) ],$$
 $$ Sq^a R^b (x) = \sum \binom{b-1-c}{a-2c} R^{a+b-c} Sq^c  (x),\,\,\,\, b>-|x|+1.$$
 \item For $p>2$ we have 
$$P^a[x,y]=\sum_i[P^i(x),P^{a-i}(y)],\,\,\,\,\beta P^a[x,y]=\sum_i([\beta P^i(x),P^{a-i}(y)]+[ P^i(x),\beta P^{a-i}(y)]).$$ 
 For any class $x$ and all $2j>-|x|+1$, the Nishida relations are
 \begin{align*}
     P^n \beta R^j=&(-1)^{n-i}\sum_i\binom{(j-i)(p-1)}{n-pi}\beta R^{n+j-i}P^i+(-1)^{n-i}\sum_i\binom{(j-i)(p-1)-1}{n-pi-1} R^{n+j-i}\beta P^i,
 \end{align*}
$$P^n R^j=(-1)^{n-i}\sum_i\binom{(j-i)(p-1)-1}{n-pi}R^{n+j-i}P^i,$$
 as well as
\begin{align*}
    P^n R^{j}(x)=&(-1)^{n-i}\sum_i\binom{(j-i)(p-1)-1}{n-pi}R^{n+j-i}P^i(x)\\
    &+\frac{1}{\lambda_{|x|}}\sum_{I, \sigma\in\Sigma_p, \sigma(1)=1}[[\cdots[[P^{i_{\sigma(1)}}(x),P^{i_{\sigma(2)}}(x)],P^{i_{\sigma(3)}}(x)]\cdots], P^{i_{\sigma(p)}}(x)]
\end{align*}
when the degree of $x$ is odd and $2j=-|x|+1$, where the bracket term sums over all nondecreasing sequences $I=(0\leq i_1\leq i_2\leq\ldots\leq i_p)$ with $i_1+i_2+\cdots+i_p=n$, and $\lambda_{|x|}$ is a fixed unit given in \Cref{maintheorem2}.(3).
 \end{enumerate}
\end{theorem}
\begin{remark}
  In parallel to the theory of spectral partition Lie algebras, Brantner and Mathew developed a derived version called \textit{partition Lie algebras}, which are algebras over a certain monad $\Lie^{\pi}_{\Fp,\Delta}$ on the derived category of chain complexes over $\Fp$. They showed that coconnective, finite type partition Lie algebras over $\Fp$ serve as the Koszul dual to complete local Noetherian simplicial commutative $\Fp$-algebras \cite[Theorem 1.11 and 1.13]{bm}. They then computed the ranks of homotopy groups of free partition Lie algebras in \cite[Theorem 1.16]{bm}.  While partition Lie algebras are more delicate than their spectral counterparts studied in this paper, we are hopeful that a similar strategy can be employed to obtain relations among the operations on the homotopy groups of partition Lie algebras. This will require as input the structure of natural operations on the homotopy groups of simplicial and cosimplicial commutative $\Fp$-algebras. 
  
  Note that for $p=2$, Goerss identified all operations and their relations on the homotopy groups of coconnective partition Lie algebras in \cite[Theorem H]{goerss} via a Quillen spectral sequence and a Hilton-Milnor type argument, with input the structure of unary operations on the homotopy groups simplicial commutative $\F$-algebras by the works of Cartan \cite{cartan}, Bousfield \cite{bousfield}, Dwyer \cite{dwyer}. For $p>2$, a basis of unary operations on the homotopy groups simplicial commutative $\Fp$-algebras was computed by Nakaoka in the simplicial case \cite{nak57,nak58}, and the quadratic relations among the operations were determined by Bousfield in \cite[Theorem 8.9]{bousfield}. In the cosimplicial case, Priddy determined the unary operations and their relations in \cite{priddy2} at all primes.
\end{remark}

\subsection{Conventions}
We denote by $\mathrm{Sp}$ the $\infty$-category of spectra.
We use $\Fp$ for both the field $\Fp$ and the Eilenbarg-MacLane spectrum representing $\Fp$.
The grading convention is homological unless for TAQ cohomology groups or otherwise stated.

A weighted graded $\Fp$-module (resp. $\Fp$-module spectrum) $M$ is an $\N$-indexed collection of $\mathbb{Z}$-graded $\Fp$-modules (resp. $\Fp$-module spectra) $\{M(w)\}_{w\in\N}$. The weight grading of an element $x\in M(w)$ is $w$. Morphisms are weight preserving morphisms of graded $\Fp$-modules. 
We assume that every object is graded and weighted whenever it makes sense.  For instance, $\Mod_{\Fp}$  stands for the ordinary category of weighted graded $\Fp$-modules,  and $\Mod_{\Fp}(\mathrm{Sp})$ the $\infty$-category of weighted graded module-spectra over $\Fp$ (which is equivalent to the derived category of chain complexes over $\Fp$ equipped with a weight grading). If the context is clear, we will call objects in either category $\Fp$-modules. The Day convolution $\tens$ makes $\Mod_{\Fp}$ (resp. $\Mod_{\Fp}(\mathrm{Sp})$) a symmetric monoidal category. The Koszul sign rule $x\tens y=(-1)^{|x||y|}y\tens x$ for the symmetric monoidal product $\tens$ depends only on the internal grading and not the weight grading.

Similarly, a shifted Lie algebra $L$ over $\Fp$ is a weighted graded $\Fp$-module $L\in\Mod_{\Fp}$ equipped with a shifted Lie bracket $[- ,-]:L_m\tens L_n\rightarrow L_{m+n-1}$ that adds weights, as well as satisfying graded commutativity $[x,y]=(-1)^{|x||y|}[y,x]$ and the graded Jacobi identity $$(-1)^{|x||z|}[x,[y,z]]+(-1)^{|y||x|}[y,[z,x]]+(-1)^{|z||y|}[z,[x,y]]=0.$$ If $p=2$, then we further require that $[x,x]=0$ for all $x$. If $p=3$, then we further require that $[x,[x,x]]=0$ for all $x$. Note that in general an unweighted Lie algebra over $\Fp$ does not admit a weight decomposition. 
We use $\pi_n(-)$ to denote the functor taking the $n$th homotopy group of a spectrum or a simplicial $\Fp$-module, as well as the functor taking the $n$th homology group of a chain complex over $\Fp$.

We use $\E_\infty$ and $\E^{\mathrm{nu}}_\infty$ to denote respectively the unital and non-unital commutative operad in $\mathrm{Sp}$, and $\E^{\mathrm{nu}}_\infty\tens \Fp$ the nonunital commutative operad in $\Mod_{\Fp}(\mathrm{Sp})$. Often we abuse notations and denote by $O$ the monad associated to the free $O$-algebra functor when $O=\E_\infty,\E^{\mathrm{nu}}_\infty,\E^{\mathrm{nu}}_\infty\tens \Fp$.

\

\noindent\textbf{Acknowledgements.} The author would like to thank Lukas Brantner for suggesting the project and sharing his insights,  Jeremy Hahn, Tyler Lawson, and Haynes Miller for many discussions and encouragements, Robert Burklund and Martin Frankland for helpful conversations, as well as Basterra and Mandell for their generosity. Special thanks to Lukas Brantner and Tyler Lawson for pointing out several issues in earlier drafts, to the referee for their detailed suggestions, and to the editor for their great patience. During the revision of this paper, the author was supported by the Danish National Research Foundation through the Copenhagen Centre for Geometry and Topology (DNRF151) and the European Union's Horizon 2020 research and innovation programme
under the Marie Sklodowska-Curie Grant Agreement No.101150469.

\section{Preliminaries}\label{section2}

\subsection{Spectral Lie operad}

We begin with a brief review of the spectral Lie operad $\partial_*(\mathrm{Id})$. 
Ching \cite{ching} and Salvatore \cite{salvatore} showed that the Goowillie derivatives $\{\partial_n(\mathrm{Id})\}$ of the identity functor $\mathrm{Id}:\mathrm{Top}_*\rightarrow\mathrm{Top}_*$  form an operad $\partial_*(\mathrm{Id})$ in spectra. This operad is Koszul dual to the nonunital commutative operad $\E^{\mathrm{nu}}_\infty$ via the operadic bar construction  $$\partial_*(\mathrm{Id})\simeq \mathbb{D}\mathrm{Bar}(1,\E_\infty^{\mathrm{nu}}, 1).$$
For a description of the operadic bar construction, see \cite{ching} for a topological model using trees and \cite[Appendix D]{brantner} for an $\infty$-categorical construction along with a comparison with the topological model.

The $n$th-derivative $\partial_n(\mathrm{Id})$ admits an explicit description due to  Arone and Mahowald \cite{am}, following the work of Johnson \cite{johnson}. Let $P_n$ be the poset of partitions of the set $\underline{n}=\{1,2,\ldots,n\}$ ordered by refinements, equipped with a $\Sigma_n$-action induced from that on $\underline{n}$. Denote by $\hat{0}$ the discrete partition and $\hat{1}$ the partition $\{\underline{n}\}$. Set $\Pi_n=P_n-\{\hat{0},\hat{1}\}$. Regarding a poset $P$ as a category, we obtain via the nerve construction a simplicial set $N_\bullet(P)$. The \textit{partition complex} $\Sigma|\Pi_n|^{\diamond}$, the reduced-unreduced suspension of the geometric realization of  $N_\bullet(\Pi_n)$,  is modeled by the simplicial set
$$N_\bullet(P_n)/(N_\bullet(P_n-\hat{0})\cup N_\bullet(P_n-\hat{1}))$$ for $n\geq 2$ and the simplicial 0-circle $S^0$ for $n=1$.
Then there is an equivalence $$\partial_n(\mathrm{Id})\simeq\mathbb{D}(\Sigma|\Pi_n|^{\diamond})$$  of spectra with $\Sigma_n$-action, where $\mathbb{D}$ denotes the Spanier-Whitehead dual of a spectrum.

\subsection{Spectral partition Lie algebras}
Motivated by the theory of classical operadic Koszul duality \cite{gk}, the natural next step is to formulate a Koszul duality theorem between suitable categories of algebras over the Koszul dual pair $\E^{\mathrm{nu}}_\infty$ and $\partial_*(\mathrm{Id})$. Partial progress was achieved by Ching and Harper in \cite{chingharper}, following a general conjecture by Francis and Gaitsgory \cite{fg}. Recent work of Brantner and Mathew \cite{bm} on spectral partition Lie algebras completely resolved the question over $\Fp$, and we will give a very brief summary of their results.

\

Let $\Mod^{\mathrm{ft}}_{\Fp}\subset \Mod_{\Fp}(\mathrm{Sp})$ be the subcategory spanned by $\Fp$-modules of \textit{finite type}, i.e. $\Fp$-modules with degree-wise finite-dimensional homotopy groups. Denote by $\Mod^{\mathrm{ft}}_{\Fp,\leq 0}\subset \Mod^{\mathrm{ft}}_{\Fp}$ the subcategory spanned by coconnective objects.  Let $\E^{\mathrm{nu}}_\infty\tens \Fp$ be the nonunital commutative operad in $\Mod_{\Fp}(\mathrm{Sp})$. There is an adjunction 
\begin{center}
\begin{tikzpicture}[node distance=3.8cm, auto]
\pgfmathsetmacro{\shift}{0.3ex}
\node (P) {$\mathrm{Alg}_{\E^{\mathrm{nu}}_\infty\tens \Fp}(\Mod_{\Fp}(\mathrm{Sp}))$};
\node(Q)[right of=P] {$\Mod_{\Fp}(\mathrm{Sp})$\ ,};

\draw[transform canvas={yshift=0.5ex},->] (P) --(Q) node[above,midway] {\footnotesize $\mathrm{cot}$};
\draw[transform canvas={yshift=-0.5ex},->](Q) -- (P) node[below,midway] {\footnotesize $\mathrm{sqz}$}; 
\end{tikzpicture}
\end{center}
where the functor sqz sends an object $M$ to the trivial $\E^{\mathrm{nu}}_\infty\tens \Fp$-algebra $M$. The restriction of this adjunction to the subcategory $\Mod^{\mathrm{ft}}_{\Fp,\leq 0}$ defines a sifted-colimit-preserving monad $(M\mapsto \mathrm{cot}(\mathrm{sqz}(M)
 ^{\vee})^{\vee})$ on $\Mod^{\mathrm{ft}}_{\Fp,\leq 0}$.
\begin{definition}\cite[Definition 5.32]{bm}\label{def: spla}
 The \textit{spectral partition Lie monad} $\spla$ is the unique sifted-colimit-preserving monad $$\spla:\Mod_{\Fp}(\mathrm{Sp})\rightarrow\Mod_{\Fp}(\mathrm{Sp})$$ extending the monad $(M\mapsto \mathrm{cot}(\mathrm{sqz}(M)
 ^{\vee})^{\vee})$ on $ \Mod^{\mathrm{ft}}_{\Fp,\leq 0}$, 
\end{definition}
Algebras over the monad $\spla$ are called \textit{spectral partition Lie algebras}. The free spectral partition Lie algebras on bounded above objects admit an explicit description.

\begin{proposition}\cite[Proposition 5.35]{bm}\label{bm535}
For $V\in\Mod_{\Fp}(\mathrm{Sp})$ bounded above, we have
$$\spla(V)\simeq|\B(\id, \E^{\mathrm{nu}}_\infty\tens \Fp, V^\vee)|^\vee\simeq \bigoplus_{n\geq 1}\big((\partial_n(\id)\tens \Fp)\tens (V)^{\tens n}\big)^{h\Sigma_n}.$$
\end{proposition}
The above formula makes it clear that spectral partition Lie algebras are \textit{not} algebras over the spectral Lie operad, as the structural map of an algebra $L$ over the spectral Lie operad in $\Mod_{\Fp}(\mathrm{Sp})$ is given by 
$$\free^{\partial_*(\mathrm{Id})\tens \Fp}(L)\simeq \bigoplus_{n\geq 1}\big((\partial_n(\id)\tens \Fp)\tens (L)^{\tens n}\big)_{h\Sigma_n}\rightarrow L.$$ Heuristically, spectral partition Lie algebras are the dual of divided power coalgebras over the cooperad $\mathrm{Bar}(1, \E^{\mathrm{nu}}_\infty\tens \Fp, 1)$, and hence candidates for the Koszul dual of $\E^{\mathrm{nu}}_\infty$-$\Fp$-algebras. To formulate the precise Koszul duality statement, we need to introduce one more  condition.

\begin{definition}
An $\E_\infty$-$\Fp$-algebra $A$ is \textit{complete local Noetherian} if

(1). $\pi_0(A)$ is a complete local Noetherian ring; 

(2). $A$ is connective and $\pi_n(A)$ is a finitely-generated module over $\pi_0(A)$ for all $n\geq 0$.
\end{definition}

Now we can state a special case of the main results by Brantner and Mathew.
\begin{theorem}\cite[Theorem 1.19]{bm}\label{kosuzlduality}
 There is an equivalence of $\infty$-categories between complete local Noetherian $\E_\infty$-$\Fp$-algebras and the $\infty$-category of coconnective spectral partition Lie algebras of finite type.
\end{theorem}

\subsection{Relation to TAQ cohomology}
Spectral partition Lie algebras are closely related to the $\Fp$-linear TAQ spectrum. Inspired by the unpublished work of Kriz \cite{kriz}, Basterra constructed the \textit{topological Andr\'{e}-Quillen homology object}
 $\mathrm{TAQ}_A(R;B)$  for a fixed map of $\E_\infty$-algebras $A\rightarrow B$ and any object $R$ in the category of $\E_\infty$-algebras between $A$ and $B$ \cite{taq}.
  For any object $R$ in the category of $\E_\infty$-$\mathbb{S}$-algebras with a map to $\Fp$, we obtain the mod $p$ TAQ spectrum $$\mathrm{TAQ}_{\mathbb{S}}(R; \Fp)\simeq |\B(\id, \free_{\E_\infty}, R)|\tens \Fp,$$ where $\free_{\E_\infty}$ denotes the monad associated with the free $\E_\infty$-$\mathbb{S}$-algebra, cf. \cite[section 5]{taq} and \cite[Proposition 1.8.9]{lawson}. The $n$th \textit{mod p TAQ cohomology} is defined to be $$\mathrm{TAQ}^n_{\mathbb{S}}(R; \Fp)=[\Sigma^{-n} \mathrm{TAQ}^{\mathbb{S}}(R; \Fp), \Fp]_{\Mod_{\Fp}(\Sp)}$$ for $R$ any $\E_{\infty}$-$\mathbb{S}$-algebra.

  We can also consider the $\Fp$-linear TAQ version. 

\begin{definition}\label{def: FpTAQ}
   For $R$ an $\E_{\infty}$-$\Fp$-algebra, its $\Fp$-linear TAQ spectrum
$$\mathrm{TAQ}_{\Fp}(R;\Fp)\simeq |\B(\id, \free_{\E_\infty\tens \Fp}, R)|.$$ The $n$th $\Fp$-\textit{linear TAQ cohomology} is defined to be $$\mathrm{TAQ}_{\Fp}^n(R; \Fp)=[\Sigma^{-n} \mathrm{TAQ}_{\Fp}(R; \Fp), \Fp]_{\Mod_{\Fp}(\mathrm{Sp})}.$$
\end{definition}
Since we will be concerned solely with mod $p$ coefficients, we will suppress the coefficient $\Fp$ in the notations of TAQ spectra and their cohomology.
There is a base-change formula
$$\mathrm{TAQ}_{\mathbb{S}}(R)\simeq \mathrm{TAQ}_{\Fp}(R\tens_{\mathbb{S}} \Fp)$$ for $R$ any $\E_{\infty}$-$\mathbb{S}$-algebra.  

In this paper we work with the nonunital $\Fp$-linear version $$\taq(A)
:=|\B(\id, \E^{\mathrm{nu}}_\infty\tens \Fp, A)|,$$ where $\E^{\mathrm{nu}}_\infty\tens \Fp$ is the nonunital $\E_\infty$-operad in $\Mod_{\Fp}$ and $A$ a $\E^{\mathrm{nu}}_\infty\tens \Fp$-algebra. We call this the \textit{reduced mod} $p$ \textit{TAQ spectrum} of $A$, since $$\taq(A)\oplus \Fp\simeq \mathrm{TAQ}_{\Fp}(\Fp\oplus A).$$ 
Thus the reduced mod $p$ TAQ cohomology group $\taq^n(A):=[\Sigma^{-n} \taq(A), \Fp]_{\Mod_{\Fp}(\Sp)}$ differ from the $\Fp$-linear TAQ cohomology group $\mathrm{TAQ}^n_{\Fp}(A\oplus \Fp)$ only when $n=0$ by a copy of $\Fp$. 
By Proposition \ref{bm535},  when $A$ is a bounded above  $\Fp$-module of finite type considered as a trivial $\E^{\mathrm{nu}}_\infty\tens \Fp$-algebra, there is an equivalence  $$\taq^n(A^\vee)\cong \pi_{-n}(|\B(\id, \E^{\mathrm{nu}}_\infty\tens \Fp, A^\vee)|^{\vee})\cong \pi_{-n}( \spla(A)).$$
From here on, we will often omit the terms $\Fp$-linear and mod $p$ when there is no ambiguity regarding which version of TAQ cohomology is concerned.

\subsection{Operations}
The goal of this paper is to understand natural operations and their relations on the homotopy groups of spectral partition Lie algebras and mod $p$ TAQ cohomology. First we record a few general remarks about operations on algebras over a monad, adapted from Lawson's excellent survey \cite[section 1.4]{lawson} and \cite[section 3]{rezk} on the theory of operations for algebras over operads. 

Given a monad $\mathbf{T}$ on the $\infty$-category $\Mod_{\Fp}(\mathrm{Sp})$, we define an \textit{operation} on $\mathbf{T}$-algebras to be a natural transformation $\pi_m(-)\rightarrow\pi_n(-)$ of functors $\mathrm{Alg}_{\mathbf{T}}\rightarrow\mathrm{Sets}$ for some $m,n$. Here $\mathrm{Alg}_{\mathbf{T}}=\mathrm{Alg}_{\mathbf{T}}(\Mod_{\Fp}(\mathrm{Sp}))$ is the $\infty$-category of $\mathbf{T}$-algebras over $\Mod_{\Fp}(\mathrm{Sp})$. Let $\mathrm{Op}(m;n)$ be the set of operations for fixed $m,n$.
It follows from the universal property of free algebras that for any $\mathbf{T}$-algebra $A$, $$\pi_m(A)\cong\mathrm{Map}_{\mathrm{Alg}_{\mathbf{T}}}(\free^{\mathbf{T}}(\Sigma^m \Fp), A).$$ Hence  $\free^{\mathbf{T}}(\Sigma^m \Fp)$ is the representing object for the functor $\pi_m(-)$ on $\mathrm{Alg}_{\mathbf{T}}$. 

By the Yoneda Lemma, the set of operations $\mathrm{Op}(m;n)$, or equivalently natural transformations $\pi_m(-)\rightarrow\pi_n(-)$ in $\mathrm{Alg}_{\mathbf{T}}$, is isomorphic to $\pi_n(\free^{\mathbf{T}}(\Sigma^m \Fp))$. Explicitly, given an operation $\alpha\in\pi_n(\free^{\mathbf{T}}(\Sigma^m \Fp))$ and a class $x\in\pi_m(A)$ with $A$ a $\mathbf{T}$-algebra, we obtain a class $\alpha(x)$ in $\pi_n(A)$ via the pullback
$$\pi_m(A)\cong\mathrm{Map}_{\mathrm{Alg}_{\mathbf{T}}}(\free^{\mathbf{T}}(\Sigma^m \Fp), A)\xrightarrow{\alpha^*}\mathrm{Map}_{\mathrm{Alg}_{\mathbf{T}}}(\free^{\mathbf{T}}(\Sigma^n \Fp), A)\cong \pi_n(A).$$

Therefore, to understand the unary operations on $\mathbf{T}$-algebras and their relations, we need to first compute $\pi_*(\free^\mathbf{T}(\Sigma^m \Fp))$ as an algebra for all $m$. Then we need to understand the composition product on unary operations $$\pi_n(\free^\mathbf{T}(\Sigma^m \Fp))\times\pi_m(\free^\mathbf{T}(\Sigma^l \Fp))\rightarrow\pi_n(\free^\mathbf{T}(\Sigma^l \Fp))$$ for all $l,m,n$, which corresponds to composing two natural transformations $\pi_l(-)\rightarrow\pi_m(-)$ and $\pi_m(-)\rightarrow\pi_n(-)$ of functors on $h\mathrm{Alg}_{\mathbf{T}}$. In general, natural $k$-ary operations $\prod_{l=1}^k \pi_{i_l}(-)\rightarrow\pi_n(-)$ are parametrized by the homotopy groups $$\pi_n(\free^\mathbf{T}(\Sigma^{i_1} \Fp\oplus\cdots\oplus \Sigma^{i_k} \Fp))$$ for all $k$-tuples $(i_1,\ldots, i_k)$.
\begin{remark}\label{rmk: alg approx}
    Since taking homotopy groups yields an equivalence between the homotopy category associated with $\Mod_{\Fp}(\mathrm{Sp})$ and $\Mod_{\Fp}$, the monad $\mathbf{T}$ on $\Mod_{\Fp}(\mathrm{Sp})$ descends to a monad $\pazocal{T}$ on $\Mod_{\Fp}$, called its \textit{algebraic approximation}. Therefore the canonical map $\pazocal{T}(\pi_*(X))\rightarrow \pi_*(\mathbf{T}(X))$ is a natural isomorphism for any $X\in \Mod_{\Fp}(\mathrm{Sp})$, so $\pazocal{T}$ paramtrizes natural operations on the homotopy groups of $\mathbf{T}$-algebras.
\end{remark}

Here we specialize to $\mathbf{T}=\spla$. The decomposition of the free algebra over $\spla$ into homogeneous pieces in Proposition \ref{bm535} allows us to impose a weight grading on the operations on the homotopy groups of $\spla$-algebras in the usual sense.

\

On the other hand, the mod $p$ TAQ cohomology functor $\mathrm{TAQ}_{\Fp}(-)$ on $\mathrm{Alg}_{\E_\infty}(\Mod_{\Fp}(\mathrm{Sp}))$ has as representing objects the square-zero extensions $\mathrm{sqz}(M)\simeq\Fp\oplus M$ for $M\in \Mod_{\Fp}(\mathrm{Sp})$ \cite[section 1.8]{lawson}.
Therefore, for any $m$ and tuple $(i_1,\ldots, i_k)$, the group of cohomology operations $$\prod^k_{i=1}\mathrm{TAQ}^{i_l}_{\Fp}(-)\rightarrow\mathrm{TAQ}^m_{\Fp}(-)$$  is the given by $\mathrm{TAQ}^{m}_{\Fp}(\Fp\oplus\Sigma^{i_1}\Fp\oplus\cdots\Sigma^{i_k}\Fp).$

Note that all operations vanish on the unit except for scalar multiplication. Since $$\taq^n(A^\vee)\cong \pi_{-n}(|\B(\id, \E^{\mathrm{nu}}_\infty\tens \Fp, A^\vee)|^{\vee})\cong \pi_{-n}( \spla(A))$$ for all $n$ when $A$ is a bounded above $\Fp$-module of finite type considered as a trivial $\E^{\mathrm{nu}}_\infty\tens \Fp$-algebra, natural operations and their relations on the reduced mod $p$ TAQ cohomology, or equivalently, the mod $p$ TAQ cohomology away from the unit, agree with those on the homotopy group of spectral partition Lie algebras up to a change of grading conventions.   

Brantner and Mathew obtained bases of the homotopy groups of free spectral partition Lie algebras on $\Sigma^j \Fp$ via an isotropy spectral sequence as in \cite[Example 1.3]{ADL}. Then they propagated the result to any direct sum of shifts of $\Fp$ using the Takayasu cofibration sequence (\cite{takayasu}, cf. \cite{kuhncofib,arone}) and a Hilton-Milnor-type decomposition of the partition complex and an EHP sequence developed in \cite{young}.

\begin{definition}\label{def: Lyndon}
 We say a word $w$ in letters $\{x_1, \ldots, x_k\}$ is a \textit{Lyndon word} if it is smaller than
any of its cyclic rotations in the lexicographic order with $x_1<\cdots < x_k$. Write $B(n_1,\ldots, n_k)$ for
the set of Lyndon words in which the letter $x_i$ appears precisely $n_i$ times, and define the degree of $w\in B(n_1,\ldots, n_k)$ to be $\mathrm{deg}(w):=
\sum_i(l_i - 1)n_i + 1$.
\end{definition}
 Note that the collection of all Lyndon words in letters $x_1, \ldots, x_k$ produces a basis for the free \textit{totally-isotropic} Lie algebra over $\Fp$ on $k$ generators, where totally-isotropic means that self-brackets are also zero.
\begin{notation}\label{notn: L_p approx}
    Denote by $\pazocal{L}_p$ the monad on $\Mod_{\Fp}$ that serves as the  algebraic approximation of the monad $\spla$ in the sense of \Cref{rmk: alg approx}, , which has both an internal grading and a weight grading $[-]$ induced by the weight decomposition of $\spla$.
\end{notation}
 
\begin{theorem} \cite[Theorem 1.22]{bm}\label{dimension}
 The $\Fp$-vector space $$\pi_*(\spla(\Sigma^{l_1}\Fp\oplus\cdots\oplus \Sigma^{l_k}\Fp))\cong \pazocal{L}_p(\Sigma^{l_1}\Fp\oplus\cdots\oplus \Sigma^{l_k}\Fp))$$ has a basis indexed by
sequences $(i_1,\ldots, i_k, e,w)$. Here $w \in B(n_1, \ldots , n_k)$ is a Lyndon word. We have $e \in \{0, \iota\}$, where
$\iota= 1$ if $p$ is odd and $\mathrm{deg}(w)$ is even. Otherwise, $\iota=0$.
The integers $i_1,\ldots, i_k$ satisfy:
\begin{enumerate}
    \item Each $i_j$ is congruent to 0 or 1 modulo $2(p-1)$.
    \item For all $1 \leq j < k$, we have $i_j < pi_{j+1}$.
    \item We have $i_k \leq (p- 1)(1+ e) \mathrm{deg}(w)- \iota$.
\end{enumerate}
The homological degree of $(i_1,\ldots, i_k, e,w)$ is $((1 + e) \mathrm{deg}(w)- e) + i_1 + \ldots + i_k - k$.
\end{theorem}

In other words, Brantner and Mathew identified the underlying functor of $\pazocal{L}_p$, but not the monadic structure. In the rest of the paper, we use a dual bar spectral sequence to compute the relations among the unary operations. By comparing with the bases in the above  theorem, we show that the spectral sequences of interest degenerate and deduce the composition product among unary operations. Then we use the homotopy fixed points spectral sequence to deduce a shifted Lie algebra structure with restrictions on the homotopy group of spectral partition Lie algebras.  Finally we show that these are all the natural operations and construct the target category that captures all algebraic structures on the homotopy group of spectral partition Lie algebras.
\subsection{Shifted restricted Lie algebras}
We end this section by recalling the definition of shifted Lie algebras with a restriction map.

\begin{definition} (cf. \cite{jacobson}, \cite{fresse} for the unshifted version.)
A shifted \textit{restricted} Lie algebra over $\F$, denoted as a $\Lie^{s,\rho}_{\F}$-algebra, is a graded $\F$-module $L=L_\bullet$ with a shifted Lie bracket $L_m\tens L_n\rightarrow L_{m+n-1}$ and a restriction map $x\mapsto x^{\{2\}}$ with $x^{\{2\}}\in L_{2|x|-1}$ for all $x\in L$ satisfying the following identities:
 \begin{enumerate}
     \item $\mathrm{ad}(x^{\{2\}})=\mathrm{ad}(x)$ for all  $x\in L$, so in particular $[y, x^{\{2\}}]=[y,[x,x]]$ for all $x,y\in L$;
     \item For all $x,y\in L$, $(x+y)^{\{2\}}=x^{\{2\}}+y^{\{2\}}+[x,y]$, where $s_i$ is the coefficient of $t^{i-1}$ in the formal expression $\mathrm {ad} (tx+y)^{p-1}(x)$.
\end{enumerate}
Suppose that $p>2$. A shifted \textit{restricted} Lie algebra over $\Fp$, denoted as a $\Lie^{s,\rho}_{\Fp}$-algebra, is a graded $\Fp$-module $L=L_\bullet$ with a shifted Lie bracket $L_m\tens L_n\rightarrow L_{m+n-1}$ and a restriction map $x\mapsto x^{\{p\}}$ with $x^{\{p\}}\in L_{p|x|-p+1}$ whenever $|x|$ is odd, satisfying the following identities:
 \begin{enumerate}
    \item $(c x)^{\{p\}}=c^p x^{\{p\}}$ for all odd degree $x\in L$ and $c \in \Fp$;
     \item $\mathrm{ad}(x^{\{p\}})=\mathrm{ad}(x)$ for all odd degree $x\in L$;
     \item For all odd degree $x,y\in L$, $(x+y)^{\{p\}}=x^{\{p\}}+y^{\{p\}}+\sum_{i=1}^{p-1} \frac{s_i}{i}(x,y)$, where $s_i$ is the coefficient of $t^{i-1}$ in the formal expression $\mathrm {ad} (tx+y)^{p-1}(x)$.
 \end{enumerate}
 Here $\mathrm{ad}(x)$ stands for the self-map $y\mapsto [y,x]$ on $L$.
\end{definition}

Let  $\free^{\Lie^{s,\rho}_{\Fp}}$ be the associated free functor. 
\begin{proposition} \cite[p.149]{restrictedbasis}\label{prop: basis-restrictedlie}
    A basis for $\free^{\Lie^{s,\rho}_{\F}}(M)$ is given by $$\{u,u^{\{2\}},u^{\{2\}^2}=(u^{\{2\}})^{\{2\}},\ldots\},$$ where $u$ ranges over Lyndon words in letters  an $\F$-basis  of $M$. 
When $p>2$, a basis for $\free^{\Lie^{s,\rho}_{\Fp}}(M)$ is $$\{v\}\cup\{u,u^{\{p\}},u^{\{p\}^2}=(u^{\{p\}})^{\{p\}},\ldots\},$$ where $u$ ranges over Lyndon words of odd degrees in letters an $\Fp$-basis  of $M$ and $v$ ranges over Lyndon words of even degrees. 
\end{proposition}

\section{A bar spectral sequence}\label{section3}
This section serves as a preliminary examination of the bar spectral sequence for a $\E^{\mathrm{nu}}_\infty\tens \Fp$-algebra $A$, where $\E^{\mathrm{nu}}_\infty\tens \Fp$ is the  $\E^{\mathrm{nu}}_\infty$-operad in $\Mod_{\Fp}(\mathrm{Sp})$, obtained by the skeletal filtration of the geometric realization of the bar construction
$$\widetilde{E}^2_{s,t}=\pi_s(\pi_t(\B(\id, \E^{\mathrm{nu}}_\infty\tens \Fp,A)) \Rightarrow \pi_{s+t}(|\B(\id, \E^{\mathrm{nu}}_\infty\tens \Fp, A)|)\cong\taq_{s+t}(A).$$ 
Similarly, there is a dual bar spectral sequence
$$E_{s,t}^2=\pi_s(\pi_t(\B(\id, \E^{\mathrm{nu}}_\infty\tens \Fp,A)^{\vee} )) \Rightarrow \pi_{s+t}(|\B(\id, \E^{\mathrm{nu}}_\infty\tens \Fp, A)|^\vee).$$

When $A=\Sigma^{-j} \Fp$ is a trivial $\E^{\mathrm{nu}}_\infty\tens \Fp$-algebra, the $E^\infty$-page records unary operations on a degree $j$ class in the homotopy group of any spectral partition Lie algebra and those on a degree $-j$ cohomology class in the reduced mod p TAQ cohomology. Both spectral sequences converge strongly when $\pi_*(A)$ is of finite type by Boardman's criteria (Theorem 6.1 and 7.1 respectively in \cite{boardman}), since they are concentrated in the upper half plane and on the $E^2$-page each bidegree supports finitely many copies of $\Fp$. We will see that both spectral sequences degenerate on the second page in all cases of interest.
\subsection{The Dyer-Lashof algebra}
Dyer-Lashof operations are natural unary operations on the mod $p$ homology of infinite loop spaces and $\E_{\infty}$-algebras in Spectra. These operations and their relations were computed by Araki-Kudo \cite{kudoaraki}, Dyer-Lashof \cite{dyerlashof},  Cohen-Lada-May \cite{adem}, and Bruner-May-McCLure-Steinberger \cite{bmms}. Denote by $\R$ the non-unital mod $p$ Dyer-Lashof algebra. 
\begin{proposition}\cite[I.1]{adem}, \cite[III.1]{bmms}\label{adem}
At $p=2$, the Dyer-Lashof algebra $\R$ is generated by operations $Q^i$ in degree $i$ and weight $2$ subject to the Adem relations
$$Q^r Q^s =\sum_{r+s-i\leq 2i} \binom{i-s-1}{2i-r} Q^{r+s-i}Q^i$$ for $r>2s$. An $\F$-basis for $\R$ is given by monomials $Q^{i_1}\cdots Q^{i_n}$ with $i_l\leq 2i_{l+1}$ for $1\leq l<n$.

For $p$ an odd prime, the mod $p$  Dyer-Lashof algebra, also denoted by $\R$, is generated by operations $\beta^\epsilon Q^i$ in degree $2(p-1)i-\epsilon$ and weight $p$ for $\epsilon\in\{0,1\}$ and all $i$, subject to the Adem relations
$$\beta^{\epsilon}Q^r Q^s =\sum_{r+s-i\leq pi} (-1)^{r+i} \binom{(p-1)(i-s)-1}{pi-r} \beta^{\epsilon}Q^{r+s-i}Q^i$$
%$$\beta Q^r Q^s =\sum_{r+s-i\leq pi} (-1)^{r+i} \binom{(p-1)(i-s)-1}{pi-r} \beta Q^{r+s-i}Q^i$$
for $r>ps$, and
\begin{align*}
    \beta^\epsilon Q^r \beta Q^s =&(1-\epsilon)\sum_{r+s-i< pi} (-1)^{r+i}\binom{(p-1)(i-s)}{pi-r} \beta Q^{r+s-i}Q^i\\
    &-\sum_{r+s-i< pi} (-1)^{r+i}\binom{(p-1)(i-s)-1}{pi-r-1}  \beta^\epsilon Q^{r+s-i}\beta Q^i
\end{align*}
%$$\beta Q^r \beta Q^s =-\sum_{r+s-i<pi} (-1)^{r+i}\binom{(p-1)(i-s)-1}{pi-r-1}  \beta Q^{r+s-i}\beta Q^i$$ 
for $r\geq ps$. An $\Fp$-basis for $\R$ is given by monomials $\beta^{\epsilon_1}Q^{i_1}\cdots \beta^{\epsilon_n}Q^{i_n}$ with $i_l\leq pi_{l+1}-\epsilon_{l+1}$ for $1\leq l<n$.

\end{proposition}

We say that a modules $M$ over the Dyer-Lashof algebra $\R$ is \textit{unstable} (or \textit{allowable}) if the following conditions holds:

1. When $p=2$, for any nonempty sequence of operation $Q^I=Q^{i_1}\cdots Q^{i_k}$ and $x\in M$ of degree $j$, if  $i_l-i_{l+1}-\ldots- i_k<j$ for some $1\leq l\leq k$ then $Q^I(x)=0$.

2. When $p>2$, for any $x\in M$ of degree $j$ and any nonempty sequence of operation $\alpha=\beta^{\epsilon_1}Q^{i_1}\cdots \beta^{\epsilon_k}Q^{i_k}$, if  $2i_m-\epsilon_m< j+2(p-1)i_{m+1}+\ldots +2(p-1)i_k-\epsilon_1-\cdots \epsilon_k$ for some $1\leq m\leq k$  then $\alpha(x)=0$. 

Denote by $\Mod_{\R}$ the category of unstable $\R$-modules, and $\pazocal{A}_{\R}$ the monad associated with the free unstable $\R$-module functor. We will often omit the term unstable before $\R$-modules when the context is clear.

\begin{definition} \label{def: polyR}
    Let $\p_{\Fp}:\Mod_{\Fp}\rightarrow \Mod_{\Fp}$ be the monad associated with the free (graded weighted) commutative $\Fp$-algebra functor that sends the $\Fp$-module on a single generator $x$ to the polynomial algebra $\Fp[x]$. Define a monad $\p_{\R}$ on $\Mod_{\Fp}$ by setting $$\p_{\R}(M)=\p_{\Fp}(\pazocal{A}_{\R}(M))/J$$ for any $\Fp$-module $M$, where $J$ is the two-sided ideal generated under the commutative product $\tens$ by the relations $Q^{|x|/2}(x)=x^{\tens p}$ for any $x\in \pazocal{A}_{\R}(M)$ of even degree when $p>2$, and generated by $Q^{|x|}(x)=x^{\tens 2}$ for any $x\in \pazocal{A}_{\R}(M)$ when $p=2$. The monad composition map is given by the Adem relations and the Cartan formula $$Q^i(x\tens y)=\sum_j Q^j(x)\tens Q^{i-j}(y).$$
    
\end{definition}

 Let $\p_{\R}$ be the category of $\p_{\R}$-algebras, i.e. (graded weighted) commutative algebras over $\Fp$ with an unstable $\R$-module structure that is compatible with the commutative product $\tens$ in the sense that the Cartan formula is satisfied and  $Q^{|x|/2}(x)=x^{\tens p}$ for all even degree $x\in M$ when $p>2$, whereas $Q^{|x|}(x)=x^{\tens 2}$ when $p=2$.
 
 A classical result by May and McClure tells us that this is the target category for the mod $p$ homology of non-unital $\E_\infty$-$\Fp$-algebras.
 \begin{theorem}\cite[I.4.1]{adem}\cite[IX.2.1]{bmms}\label{may}
   For any $\E^{\mathrm{nu}}_\infty\tens \Fp$-algebra $X$, there is an isomorphism $$\pi_*(\free^{\E^{\mathrm{nu}}_\infty\tens \Fp} (X)) \cong \p_{\R}(\pi_*(X))$$ of free $\p_{\R}$-algebras.
 \end{theorem}
 
%Denote again by $\p_{\R}$ the monad coming from the free-forgetful adjunction
%\begin{center}
%\begin{tikzpicture}[node distance=2.8cm, auto]
%\pgfmathsetmacro{\shift}{0.3ex}
%\node (P) {$\Mod_{\Fp}$};
%\node(Q)[right of=P] {$\p_{\R}$\ .};

%\draw[transform canvas={yshift=0.5ex},->] (P) --(Q) node[above,midway] {\footnotesize $\free^{\p_{\R}}_{\Mod_{\Fp}}$};
%\draw[transform canvas={yshift=-0.5ex},->](Q) -- (P) node[below,midway] {\footnotesize $U^{\p_{\R}}_{\Mod_{\Fp}}$}; 
%\end{tikzpicture}
%\end{center}

By repeatedly applying Theorem \ref{may}, the bar spectral sequence for a $\E^{\mathrm{nu}}_\infty\tens \Fp$-algebra $A$ with $M=\pi_*(A)$ can be rewritten as
$$\widetilde{E}^2_{s,t}=\pi_s(\pi_t(\B(\id, \E^{\mathrm{nu}}_\infty\tens \Fp,A))\cong\pi_{s,t}(\B(\id, \p_{\R},M))\Rightarrow \pi_{s+t}(|\B(\id, \E^{\mathrm{nu}}_\infty\tens \Fp, A)|).$$ 
Similarly, the dual bar spectral sequence takes the form
$$E_{s,t}^2=\pi_{s,t}(\B(\id, \p_{\R},M)^{\vee} ) \Rightarrow \pi_{s+t}(|\B(\id, \E^{\mathrm{nu}}_\infty\tens \Fp, A)|^\vee).$$
The $E^\infty$-page is the reduced mod $p$ $\mathrm{TAQ}$ cohomology $\taq^{-*}(A)$, or the homotopy group of the spectral partition Lie algebra $|\B(\id, \E^{\mathrm{nu}}_\infty\tens \Fp, A)|^\vee\simeq \mathrm{Lie^{\pi}_{\E_\infty,\Fp}}(A^\vee)$ when $A$ is a bounded-below $\Fp$-module of   finite type considered as a trivial $\E^{\mathrm{nu}}_\infty\tens \Fp$-algebra. Since we will only be concerned with objects of finite type over $\Fp$, we can switch freely between the two version by taking linear dual when computing the second page. It is easier to work with the bar construction, so we will focus on the bar spectral sequence. 

Observe that the $E^2$-page of the bar spectral sequence  $$\widetilde{E}^2_{s,t}=\pi_{,t}s(\B(\id, \p_{\R},M)=\pi_{s,t}(\mathbb{L}Q^{\p_{\R}}_{\Mod_{\Fp}}(M))$$ is the \textit{Andr\'{e}-Quillen homology} of $M$ with respect to the monad $\p_{\R}$.
Our plan is to find a suitable factorization of the indecomposable functor $Q^{\p_{\R}}_{\Mod_{\Fp}}$ to separate the unary and binary structures of the monad $\p_{\R}$. This will allow us to replace the bar construction computing its total left derived functor by a smaller double complex that is amenable to Koszul-duality-type computations. We will follow a similar strategy used in \cite[section 3.3]{me}, which is inspired by \cite[section 4.4]{bhk}. The subtlety lies in the identification of bottom non-vanishing Dyer-Lashof operations with self products, which precludes the factorization of $\p_{\R}$ as a composite monad.

\subsection{The derived indecomposable functor}
Before diving into the computation, we briefly recall without proof the homotopy theory of monads on the category of weighted graded $\Fp$-modules and especially the two-sided bar construction for simplicial objects, following closely Sections 3.1, 4.2 and 4.3 in \cite{bhk}. For the general theory, see for instance Sections 3.1 and 3.2 of \cite{jn}.

 Let $\mathbf{T}$ be an augmented monad on the category $\Mod  _{k}$ of weighted graded $k$-modules, where $k$ is a field. Suppose that $\mathbf{T}$ preserves reflexive coequalizers or filtered colimits. By augmented monad we mean that there is a map of monads $\mathbf{T}\rightarrow \id$ in $\Mod_{k}$, with $\id$ the identity functor considered as the trivial monad. Denote by $\mathrm{Alg}_{\mathbf{T}}(\Mod  _{k})$ the category of $\mathbf{T}$-algebras. The forgetful functor $U:\mathrm{Alg}_{\mathbf{T}}(\Mod  _k)\rightarrow \Mod  _k$ admits a left adjoint, the free functor $\free^{\mathbf{T}}:\Mod  _k\rightarrow\mathrm{Alg}_{\mathbf{T}}(\Mod_k)$. 

Denote by $s\Mod_k$ the category of simplicial weighted graded $k$-modules. Levelwise application of the adjunction $\free^{\mathbf{T}}\dashv U$ gives rise to an adjunction  between the corresponding categories of simplicial objects $$\free^{\mathbf{T}}\dashv U:\mathrm{Alg}_{\mathbf{T}}(s\Mod_{k})\rightarrow s\Mod_{k},$$ as well as a monad $\mathbf{T}$ on $s\Mod_{k}$. We equip $s\Mod_{k}$ with the standard cofibrantly generated model structure. Suppose that the path objects of $s\Mod_{k}$ lifts to $s\mathrm{Alg}_\mathbf{T}$, the category of simplicial $\mathbf{T}$-algebras. Then this adjunction induces a right transferred model structure on the category of simplicial $\mathbf{T}$-algebras, with weak equivalences and fibrations defined on the underlying simplicial weighted graded $k$-modules by \cite[Theorem 3.2, Remark 3.3]{jn}.

Denote by $T^{\mathbf{T}}:\Mod  _{k}=\mathrm{Alg}_{\mathbf{\id}}(\Mod  _{k})\rightarrow\mathrm{Alg}_{\mathbf{T}}(\Mod_{k})$ the trivial $\mathbf{T}$-algebra functor, which is induced by the augmentation. It has a left adjoint $Q^{\mathbf{T}}:\mathrm{Alg}_{\mathbf{T}}(\Mod  _{k})\rightarrow \Mod_{k}$, the indecomposable functor with respect to the $\mathbf{T}$-algebra structure, which satisfies $Q^{\mathbf{T}}\circ \free^{\mathbf{T}}\simeq \id$. Applying this adjunction levelwise to the corresponding categories of simplicial objects, we obtain a Quillen adjunction  $$Q^{\mathbf{T}}\dashv T^{\mathbf{T}}:s\mathrm{Alg}_\mathbf{T}\rightarrow s\Mod_{k}.$$ The left derived functor $\mathbb{L}Q^{\mathbf{T}}$ of $Q^{\mathbf{T}}$ can be computed by the following standard recipe.

\begin{construction}\label{constr: simplicial bar}
Given a right module $R:\Mod_{k}\rightarrow\mathcal{D}$ over $\mathbf{T}$, and a simplicial object $A$ in $\mathrm{Alg}_{\mathbf{T}}(\Mod _{k})$, one can apply the two-sided bar construction $\B(R,\mathbf{T},-)$ levelwise to $A$. The diagonal of the resulting bisimplicial complex is a simplicial object in $\mathcal{D}$, denoted by $\B(R,\mathbf{T}, A)$.
\end{construction}
In particular, if we regard a $\mathbf{T}$-algebra $A$ as the constant simplicial object on $U(A)$ equipped with a simplicial $\mathbf{T}$-algebra structure, denoted also as $A$ by abuse of notation, then $\B(R,\mathbf{T}, A)$ agrees with the usual two-sided bar construction.

Since the free resolution $\B(\free^{\mathbf{T}},\mathbf{T},A)$ is a cofibrant replacement of $A$ in the category of simplicial $\mathbf{T}$-algebras, the left derived functor of a functor $F$ can be computed by applying $F$ levelwise to a cofibrant replacement, so $$\mathbb{L}Q^{\mathbf{T}}(A)\simeq Q^\mathbf{T}\B(\free^{\mathbf{T}},\mathbf{T},A)=\B(\mathrm{id},\mathbf{T},A). $$

\subsection{A smaller complex for the $E^1$-page}
Now we would like to replace the bar complex on the $E^1$-page of the bar spectral sequence with the total complex of iterated bar constructions that deals with the unary and binary operations separately. 

It follows from \Cref{def: polyR} that killing the commutative product has the effect of killing the bottom Dyer-Lashof operation on a class. Hence for $p>2$ we shall consider the intermediate category $\Mod_{\R'}$ of modules over the Dyer-Lashof algebra $\R$ with unstability conditions $Q^i(x)=0$ for $2i\leq |x|$.  
Then the indecomposables functor $Q^{\p_{\R}}_{\Mod_{\Fp}}$ factors as $Q^{\Mod_{\R'}}_{\Mod_{\Fp}}\circ Q^{\p_{\R}}_{\Mod_{\R'}}$, and we have a  composite of adjunctions
\begin{center}
\begin{tikzpicture}[node distance=2.8cm, auto]
\pgfmathsetmacro{\shift}{0.3ex}
\node(P) {$\Mod_{\Fp}$};
\node(Q)[right of =P] {$\Mod_{\R'}$};
\node(R)[right of =Q]{$\p_{\R}$.};
\draw[transform canvas={yshift=-0.5ex},->] (P) --(Q) node[below,midway] {\footnotesize $T^{\Mod_{\R'}}_{\Mod_{\Fp}}$};
\draw[transform canvas={yshift=0.5ex},->](Q) -- (P) node[above,midway] {\footnotesize $Q^{\Mod_{\R'}}_{\Mod_{\Fp}}$}; 
\draw[transform canvas={yshift=-0.5ex},->] (Q) --(R) node[below,midway] {\footnotesize $T^{\p_{\R}}_{\Mod_{\R'}}$};
\draw[transform canvas={yshift=0.5ex},->](R) -- (Q) node[above,midway] {\footnotesize $Q^{\p_{\R}}_{\Mod_{\R'}}$}; 
\end{tikzpicture}
\end{center}
Whereas for $p=2$, the indecomposables functor $Q^{\p_{\R}}_{\Mod_{\F}}$ factors as $Q^{\Mod_{\R}}_{\Mod_{\F}}\circ Q^{\p_{\R}}_{\Mod_{\R}}$ by first killing the exterior algebra structure and then the unary operations, and we have a composite of adjunctions
\begin{center}
\begin{tikzpicture}[node distance=2.8cm, auto]
\pgfmathsetmacro{\shift}{0.3ex}
\node(P) {$\Mod_{\F}$};
\node(Q)[right of =P] {$\Mod_{\R}$};
\node(R)[right of =Q]{$\p_{\R}$.};
\draw[transform canvas={yshift=-0.5ex},->] (P) --(Q) node[below,midway] {\footnotesize $T^{\Mod_{\R}}_{\Mod_{\F}}$};
\draw[transform canvas={yshift=0.5ex},->](Q) -- (P) node[above,midway] {\footnotesize $Q^{\Mod_{\R}}_{\Mod_{\F}}$}; 
\draw[transform canvas={yshift=-0.5ex},->] (Q) --(R) node[below,midway] {\footnotesize $T^{\p_{\R}}_{\Mod_{\R}}$};
\draw[transform canvas={yshift=0.5ex},->](R) -- (Q) node[above,midway] {\footnotesize $Q^{\p_{\R}}_{\Mod_{\R}}$}; 
\end{tikzpicture}
\end{center}
We would like to use the above factorizations of $Q^{\p_{\R}}_{\Mod_{\Fp}}$ to obtain a double complex that is more computable than the bar complex $\B(\id, \p_{\R},M)$ in the special case where $M$ is a square-zero extension. 
\begin{lemma}
There is a weak equivalence of simplicial $\R'$-modules 
$$\mathbb{L}Q^{\p_{\R}}_{\Mod_{\R'}}(M)\simeq \B(\id, \p_{\Fp}, M)$$ for any $\Fp$-module $M$ considered as a trivial $\p_{\R}$-algebra when $p>2$. The simplicial $\R'$-module structure on the right-hand side is levelwise trivial.
Similarly, for $p=2$ there is a weak equivalence of simplicial $\R$-modules 
$$\mathbb{L}Q^{\p_{\R}}_{\Mod_{\R}}(M)\simeq \B(\id, \Lambda_{\F}, M)$$ for any $\F$-module $M$ considered as a trivial $\p_{\R}$-algebra, where $\Lambda_{\F}$ is the monad associated with the free exterior algebra over $\F$. The simplicial $\R$-module structure on the right-hand side is levelwise trivial.
\end{lemma}

\begin{proof}
When $p>2$, there is a map of augmented monads $\p_{\R}\rightarrow\p_{\Fp}\rightarrow \id$, the first of which kills all Dyer-Lashof operations that are not the bottom operations $Q^{|x|/2}(x)=x^{\tens p}$ on even classes. This is compatible with the Cartan formula because under this map, $Q^k(x\tens y)=\sum_{i+j=k} Q^i(x)\tens Q^{j}(y)$ has nonzero summands if and only if $i=|x|/2$ and $j=|y|/2$, in which case $Q^{(|x|+|y|)/2}$ is nonzero and is identified with the $p$-fold product on $x\tens y$. Whereas all terms in any Adem relation are zero under this map. When $M$ is an $\Fp$-module considered as a trivial $\p_{\R}$-algebra, the map $\p_{\R}(M)\rightarrow\p_{\Fp}(M)\rightarrow M$ is a map of $\p_{\R}$-algebras if we regard $\p_{\Fp}(M)$ as a $\p_{\R}$-algebra where all Dyer-Lashof operations vanish except for the bottom operations on even classes. Therefore we obtain a map of free bar resolutions
$$\Psi:\B(\free^{\p_{\R}}_{\Mod_{\Fp}}, \p_{\R}, M)\rightarrow \B(\free^{\p_{\Fp}}_{\Mod_{\Fp}}, \p_{\Fp}, M).$$ This is a weak equivalence of simplicial $\p_{\R}$-algebras, since weak equivalences are detected by the underlying simplicial $\Fp$-modules.

Next we want to show that applying $Q^{\p_{\R}}_{\Mod_{\R'}}$ preserves this weak equivalence, i.e.
$$\mathbb{L}Q^{\p_{\R}}_{\Mod_{\R'}}(M)\simeq Q^{\p_{\R}}_{\Mod_{\R'}}\B(\free^{\p_{\R}}_{\Mod_{\Fp}}, \p_{\R}, M)\rightarrow Q^{\p_{\R}}_{\Mod_{\R'}}\B(\free^{\p_{\Fp}}_{\Mod_{\Fp}}, \p_{\Fp}, M)$$ is a weak equivalence of simplicial $\R'$-modules. This is equivalent to showing that the underlying map of simplicial $\Fp$-modules  $U^{\Mod_{\R'}}_{\Mod_{\Fp}}(\Psi)$ is a weak equivalence. 

It follows from the definition of $\mathbf{P}_{\R}$ (\Cref{def: polyR}) that there is an isomorphism $$U^{\Mod_{\R'}}_{\Mod_{\Fp}}\circ Q^{\p_{\R}}_{\Mod_{\R'}}\cong Q^{\p_{\Fp}}_{\Mod_{\Fp}}\circ U^{\p_{\R}}_{\p_{\Fp}}$$ of functors landing in the $1$-category of $\Fp$-modules. Therefore we can rewrite  $U^{\Mod_{\R'}}_{\Mod_{\Fp}}(\Psi)$ as
\begin{align*}
   &U^{\Mod_{\R'}}_{\Mod_{\Fp}}\circ Q^{\p_{\R}}_{\Mod_{\R'}}\B(\free^{\p_{\R}}_{\Mod_{\Fp}}, \p_{\R}, M)\simeq Q^{\p_{\Fp}}_{\Mod_{\Fp}}\circ U^{\p_{\R}}_{\p_{\Fp}}\B(\free^{\p_{\R}}_{\Mod_{\Fp}}, \p_{\R}, M)\\ \rightarrow &U^{\Mod_{\R'}}_{\Mod_{\Fp}}\circ Q^{\p_{\R}}_{\Mod_{\R'}}\B(\free^{\p_{\Fp}}_{\Mod_{\Fp}}, \p_{\Fp}, M) \\\simeq& Q^{\p_{\Fp}}_{\Mod_{\Fp}}\circ U^{\p_{\R}}_{\p_{\Fp}}\B(\free^{\p_{\Fp}}_{\Mod_{\Fp}}, \p_{\Fp}, M). 
\end{align*}
This is indeed a weak equivalence since in both the source and the target, we are applying the left Quillen functor $Q^{\p_{\Fp}}_{\Mod_{\Fp}}$  to a free resolution of $M$ in the category of simplicial commutative $\Fp$-modules. Here we are use the fact that the underlying $\Fp$-algebra of a free $\mathbf{P}_{\R}$-algebra is free, see for instance \cite[1.5.7]{lawson}. In particular, the target is equivalent to $\B(\id, \p_{\Fp}, M)$.
Therefore we obtain a weak equivalence of simplicial $\R'$-modules
$$\mathbb{L}Q^{\p_{\R}}_{\Mod_{\Fp}}(M)\cong\B(\free^{\Mod_{\R'}}_{\Mod_{\Fp}}, \p_{\R}, M)\simeq\B(\id, \p_{\Fp}, M)$$
as desired.

The case $p=2$ is analogous, noting that there is a map of augmented monads $\p_{\R}\rightarrow\Lambda_{\F}\rightarrow \id$, the first of which kills all Dyer-Lashof operations, which is trivially compatible with the Adem relations and the Cartan formula.
\end{proof}

Let $\pazocal{A}_{\R'}$ be the additive monad associated to the free $\R'$-module functor for $p>2$.
Therefore the Andr\'{e}-Quillen homology of a trivial $\p_{\R}$ algebra $A=\mathrm{sqz}(M)$, i.e. the $E^2$-page of the bar spectral sequence
\begin{align*}
    \widetilde{E}^2_{*,*}\cong\pi_*(\mathbb{L}Q^{\p_{\R}}_{\Mod_{\Fp}}(M))&\cong\pi_*\big(Q^{\Mod_{\R'}}_{\Mod_{\Fp}}\circ Q^{\p_{\R}}_{\Mod_{\R'}}\B(\free^{\p_{\R}}_{\Mod_{\Fp}}, \p_{\R}, M)\big)\\&\cong\pi_*\big(Q^{\Mod_{\R'}}_{\Mod_{\Fp}}\B(\free^{\Mod_{\R'}}_{\Mod_{\Fp}}, \p_{\R}, M)\big)\\
    &\cong \pi_*\big(\mathbb{L}Q^{\Mod_{\R'}}_{\Mod_{\Fp}}(\B(\id, \p_{\Fp}, M))\big),
\end{align*} can be computed as the homotopy group of the double complex
$\B(\id, \pazocal{A}_{\R'}, \B(\id, \p_{\Fp}, M)).$ Similarly, for $p=2$ the $E^2$-page of the bar spectral sequence against a trivial $\mathbf{P}_{\R}$-algebra  can be computed as the homotopy group of the double complex
$\B(\id, \pazocal{A}_{\R}, \B(\id, \Lambda_{\Fp}, M)).$  In particular, we can employ a spectral sequence that computes these homotopy groups, as in 
 \cite[Proposition 3.5]{behrensrezk}.
\begin{lemma}\label{lem: Grothendieck}
    Let $M$ be a trivial $\p_{\R}$-algebra. For $p>2$ there is a Grothendieck spectral sequence
    $$\widetilde{E}^2_{s,s',t}=\pi_{s,t}\B\big(\id, \pazocal{A}_{\R'}, \pi_{s'+t'}\B(\id, \p_{\Fp}, M)\big)\Rightarrow\pi_{s+s',t}(\mathbb{L}Q^{\p_{\R}}_{\Mod_{\Fp}}(M)).$$
    If the underlying $\Fp$-module of $M$ is of finite type, then we have a dual Grothendieck spectral sequence $$E^2_{s,s',t}=\pi_{s,t}\B\big(\id, \pazocal{A}_{\R'}, \pi_{s'+t'}\B(\id, \p_{\Fp}, M)^\vee\big)^{\vee}\Rightarrow\pi_{s+s',t}(\mathbb{L}Q^{\p_{\R}}_{\Mod_{\Fp}}(M)^{\vee}).$$

    Whereas for $p=2$, we have a dual Grothendieck spectral sequence $$E^2_{s,s',t}=\pi_{s,t}\B\big(\id, \pazocal{A}_{\R}, \pi_{s'+t'}\B(\id, \Lambda_{\Fp}, M)^\vee\big)^{\vee}\Rightarrow\pi_{s+s',t}(\mathbb{L}Q^{\p_{\R}}_{\Mod_{\F}}(M)^{\vee}).$$
\end{lemma}
Note that the dual Grothendieck spectral sequence is strongly convergent under the assumption on $M$ by Theorem 7.1 and the Remark afterwards in \cite{boardman}, using again the fact that each bidegree of the $E^2$-page is a finite $\Fp$-module.

\section{Unary operations} \label{section4}
In this section, we compute the $E^2$-page of the dual Grothendieck spectral sequence when $p=2$ and $A=\Sigma^i \F$ or $p>2$ and $A=\Sigma^{i}\Fp$ with $i$ odd is a trivial $\E^{\mathrm{nu}}_\infty\tens \Fp$-algebra. Both the dual Grothendieck spectral sequence and the dual bar spectral sequence collapse on the $E^2$ page for degree reasons in these cases. This is also the case for the weight $p^k$ parts of both spectral sequences when $p>2$, $k\geq 1$, and $A=\Sigma^{i}\Fp$ with $i$ even, which we prove using that the associated homotopy fixed points spectral sequence collapses on the $E^2$-page.

\subsection{The unstable Ext groups over $\R$}
To compute the $E^2$-page of the dual Grothendieck spectral sequence, we start by computing the bigraded homotopy groups $\pi_{*,*}\B(\id,\pazocal{A}_{\R}, M)$ for a trivial $\pazocal{A}_{\R}$-algebra $M$ when $p=2$ and $\pi_{*,*}\B(\id,\pazocal{A}_{\R'}, M)$ for a trivial $\pazocal{A}_{\R'}$-algebra $M$ when $p>2$.
The main tool is Priddy's machinery of algebraic Koszul duality in \cite[Theorem 2.5]{priddy}. Since the $\Fp$-basis of $\R$ in \Cref{adem} is a Poincar\'{e}-Birkhoff-Witt basis in the sense of Priddy \cite[Theorem 5.3]{priddy}, the quadratic algebra $\R$ is a homogeneous Koszul algebra. Furthermore, it is finite in each degree. Hence the Koszul dual algebra of $\R$ is given by he Ext group $$\mathrm{Ext}^{*,*}_{\R}(\Fp,\Fp)=\pi_*((\B(\Fp, \R, \Fp)^{\vee}).$$ 

\subsubsection{The case $p=2$}
The Koszul generators of $$\mathrm{Ext}^{*,*}_{\R}(\F,\F)=\pi_*((\B(\F, \R, \F)^{\vee}).$$   are given by the collection $$( Q^i)^*:=[ (Q^i)^\vee]\in\mathrm{Ext}^{-1,*}_{\R}(\F,\F)$$ with  bidegree $(s,t)=(-1,-i)$ and weight 2. Here $(Q^i)^\vee$ denotes the linear dual of the class $Q^i=1|Q^i|1\in \B(\F,\R,\F)$. The composition product is given by juxtaposition, which corresponds to the Yoneda product on Ext groups, cf. \cite[p.42]{priddy} and \cite[Theorem 9.8]{mccleary}. The quadratic relations among the generators are the Koszul dual of the Adem relations (Proposition \ref{adem}), i.e. 
\begin{equation}\label{relationmod2}
   (Q^a)^*(Q^b)^*=\sum_{a+b-c> 2c, c\geq a-b}\binom{b-c-1}{a-2c-1} (Q^{a+b-c})^*(Q^c)^*
\end{equation}
for $a\leq 2b$, obtained by plugging the Adem relations into Priddy's machinery \cite[Theorem 2.5]{priddy} and discarding the summands whose binomial coefficient vanish.

We are interested in the \textit{unstable} $\mathrm{Ext}$ group
$$\mathrm{UnExt}^{*,*}_{\R'}(\F,\Sigma^{j} \F)=\pi_*(\B(\id, \pazocal{A}_{\R'}, \Sigma^{-j} \F)^{\vee}),$$ which is a variant of the Ext group $$\mathrm{Ext}^{*,*}_{\R}(\F,\F)=\pi_*((\B(\F, \R, \F))^{\vee})$$ obtained by regarding $\Sigma^{-j}\F$ as an unstable trivial module over $\R$ and imposing the unstability conditions $[Q^j|\alpha]=0$ for $j\leq |\alpha|$ in the bar complex, cf. \cite[\S 3]{unstable}. Note that the Koszul dual of the instability condition $Q^i|x=0$ is that $(Q^i)^*$ is undefined on $[x^\vee]$. It is thus convenient to record the operations $(Q^a)^*$ by a ringoid, which takes into account the degree of the element an operation is acting on. Explicitly, we would like the ringoid to take into account of two gradings, one corresponding to the homological degree in Ext, or equivalently the filtration degree in the dual bar spectral sequence, and the other corresponding to the internal degree.
\begin{definition}\label{ringoidmod2}
 Let $\pazocal{F}$ be the ringoid with objects the $\mathbb{Z}_{\leq 0}\times \mathbb{Z}$ and morphisms  freely generated over $\F$ under juxtaposition by the following elements: for any $s\leq 0$ and all $i,j$ satisfying $i\geq-j$, there is an element $(Q^i)^*\in \pazocal{F}\big((s,j),(s-1,j-i)\big)$ of weight 2.
 Let $\R_u^!$ be the quotient of $\pazocal{F}$ by the ideal generated by the  relations 
 \begin{equation}
   (Q^a)^*(Q^b)^*=\sum_{a+b-c> 2c, c\geq a-b}\binom{b-c-1}{a-2c-1} (Q^{a+b-c})^*(Q^c)^*
\end{equation}
for all $a,b$ satisfying $a\leq 2b$, $b\geq -j$ and $a\geq b-j$ in $\pazocal{F}\big((s,j), (s-2,j-a-b)\big)$.
\end{definition}
The functor $\free^{\R_u^!}$  takes the free $\R_u^!$-module on a $\F$-module $M$ with a further homological grading, i.e. the class $ (Q^{i_1})^* (Q^{i_2})^*\cdots ( Q^{i_k})^*(x)\in\free^{\R_u^!}(M) $ has homological degree $-k$ for $x\in M$. 
%As a sanity check, we show that the relations are never vacuous on both sides: if $b\geq -j+1,$ then $a\geq b-j+1\geq -2j+2$, so $\lfloor \frac{a+b}{3}\rfloor\geq -j+1$ and  the right hand side is never empty.

\begin{remark}\label{rmk: desuspension2}
There is an evident isomorphism $\R_u^!((s,i),(s',j))\cong \R_u^!((s-r,i),(s'-r,j))$
for any $i,j, s, s'$ and $r$ such that $s-r<0$. For any $t>0,$ there is an injection $$\mathrm{susp^t}:\R_u^!((s,i),(s-1,j))\hookrightarrow \R_u^!((s,i+t),(s-1,j+t)).$$ This is because if $(Q^{i-j})^*$ is defined on an element of degree $i$, then $i-j\geq-i>-i-t$ for any $t>0$, then $(Q^{i-j})^*$ is defined on any element of degree $i+t$. 
\end{remark}

\begin{corollary}\label{cor: unext p=2}
The unstable $\mathrm{Ext}$ group
$\mathrm{UnExt}^{*,*}_{\R}(\F,\Sigma^{j} \F)$ is  the free $\R_u^!$-module  $\R_u^!((0,j),-)$. Here we grade the Ext groups homologically. 
\end{corollary}
\begin{proof}
This follows from the argument given in \cite[Theorem 3.3]{unstable}. Since differentials in the bar complex respect internal degrees and hence the instability condition, the bar complex $\B(\id, \pazocal{A}_{\R},\Sigma^j\F)$ is a quotient of $\B(\F, \R,\Sigma^j\F)$ and thus the unstable Tor group $\mathrm{UnTor}^{*,*}_{\R}(\F,\Sigma^{j} \F)=\pi_*(\B(\id, \pazocal{A}_{\R},\Sigma^j\F))$ is the  quotient of $\mathrm{Tor}^{*,*}_{\R}(\F,\Sigma^{j} \F)$  obtained by imposing the instability condition.  Taking linear dual makes $\mathrm{UnExt}^{*,*}_{\R}(\F,\Sigma^{j} \F)$ a subgroup of $\mathrm{Ext}^{*,*}_{\R}(\F,\Sigma^{j} \F)$ with cokernel generated by the undefined operations. The composition is given by juxtaposition from the left, and hence preserves the undefinedness condition.
\end{proof}

\begin{proposition}\label{E2unary}
Let $A=\Sigma^j \F$. Then $\pi_{*}\Lie^{\pi}_{\F,\E_\infty}(A^{\vee})\cong \free^{\R_u^!}(\Sigma^{-j}\F)$ with respect to the total degree $s+t$. 
\end{proposition}
\begin{proof}
  If $A=\Sigma^j \Fp$ with $j$ odd, then $\pi_{*,*}\B(\id, \Lambda_{\Fp}, \Sigma^j \Fp)$
 consists of a single $\Sigma^{j} \Fp$. Thus the  dual Grothendieck spectral sequence (\Cref{lem: Grothendieck}) collapses on the $E^2$-page, and the $E^2$-page of the dual bar spectral sequence is concentrated in a single line $s=-k$ at each weight $2^k$ for $k\in \mathbb{N}$. Hence the  dual bar spectral sequence also collapses on the second page and there are no extension problems.  
\end{proof}

 \begin{remark}\label{rmk: additive p=2}
     The unary operations on a degree $j$ class $x$ is given by  $\pazocal{L}_2(\Sigma^j\F)\cong\free^{\R_u^!}(\Sigma^{j}\F)$.  The canonical map $\Sigma\pazocal{L}_2(\Sigma^{j}\F)\rightarrow\pazocal{L}_2(\Sigma^{j+1}\F)$ induces a map of $E^1$-pages of the $s=1$ line of the dual bar spectral sequences $$\pi_*\Sigma\mathrm{Bar}_1(\id,\p_{\R}, \Sigma^{-j}\F)^\vee\rightarrow\pi_*\mathrm{Bar}_1(\id,\p_{\R}, \Sigma^{-j-1}\F)^\vee$$ sending $\sigma(Q^{-|x|})^\vee|x$ to $(Q^{-|x|})^\vee|\sigma x$ since the Dyer-Lashof operations are stable. Passing to the $E^2$-pages, the suspension map sends $\sigma x^{\{2\}}$ to $(Q^{-|x|})^*(\sigma x)$. Note that the image is well-defined since $-|x|>-|\sigma x|$. On the other hand, for $i>-j$ the induced suspension map sends $\sigma (Q^{-|x|})^*(x)$ to $(Q^{-|x|})^*(\sigma x)$, which is precisely the map $\mathrm{susp}^1$ in \Cref{rmk: desuspension2}. We shall see in \Cref{restricted} that the linear dual of  the cycle $x\tens x=Q^{-|x|}|x^\vee$ on the line $s=1$ of the $E^1$-page $\pi_{*}\mathrm{Bar}_1(\id,\p_{\R}, \Sigma^{-j}\F)$ of the bar spectral sequence is fact the restriction $x^{[2]}$ associated with a shifted Lie bracket.
 \end{remark}

\subsubsection{The case where $p>2$}
Now we apply the same analysis to the odd primary case. As in the case $p=2$, the $\mathrm{Ext}_{\R}(\Fp,\Fp)=\pi_*((\B(\Fp, \R, \Fp))^{\vee})$ is isomorphic as an algebra to the Koszul dual of $\R$. The Koszul generators are given by the collection $$( \beta^{\epsilon}Q^i)^*:=[ (\beta^{\epsilon}Q^i)^{\vee}]\in\mathrm{Ext}^1_{\R}(\Fp,\Fp), \epsilon\in\{0,1\}, i\in\N$$ where $( \beta^{\epsilon}Q^i)^*$ has homological bidegree $(-1,-2(p-1)i+\epsilon)$ and weight $p$. Composition is given by juxtaposition corresponding to the Yoneda product. The quadratic relations are the Koszul dual of the Adem relations (Proposition \ref{adem}), i.e. 
\begin{equation*}
   (Q^a)^*( Q^b)^*+\sum_{a+b-c>pc}(-1)^{a-c}\binom{(p-1)(b-c)-1}{a-pc-1} ( Q^{a+b-c})^*( Q^c)^*=0
\end{equation*}
for $ a \leq pb$ ,
\begin{align*}
   (\beta Q^a)^*(Q^b)^*-\sum_{a+b-c\geq pc}(-1)^{a-c}\binom{(p-1)(b-c)}{a-pc} (Q^{a+b-c})^*(\beta Q^c)^*\\+\sum_{a+b-c>pc}(-1)^{a-c}\binom{(p-1)(b-c)-1}{a-pc-1} ( \beta Q^{a+b-c})^*( Q^c)^*=0
\end{align*}
for $a \leq pb, $
\begin{equation*}
   (\beta^\epsilon Q^a)^*(\beta Q^b)^*-\sum_{a+b-c\geq pc}(-1)^{a-c}\binom{(p-1)(b-c)-1}{a-pc} (\beta^\epsilon Q^{a+b-c})^*(\beta Q^c)^*=0
\end{equation*}
for $\epsilon\in\{0,1\}$ and $ a < pb$, cf. \cite{dual}.

Analogous to the case $p=2$, we dualize the unstability condition and record the Ext degree in homological grading using a ringoid.

\begin{definition} \label{ringoidodd}
 Let $\pazocal{F}$ be the ringoid with objects $\mathbb{Z}_{\leq 0}\times \mathbb{Z}$ and morphisms freely generated over $\Fp$ under juxtaposition by the following elements: for $2i>-j$ and any $s\leq 0$ there are elements $( Q^i)^*\in\pazocal{F}\big((s,j), (s-1,j-2(p-1)i)\big)$ and  $(\beta Q^i)^*\in\pazocal{F}\big((s,j), (s-1,j-2(p-1)i+1)\big)$. We suppress the first grading for ease of notation when there is no ambiguity.

 Let $\R_u^!$ be the quotient of $\pazocal{F}$ by the ideal generated by the following quadratic relations for all $s\leq 0$:
\begin{equation*}
   (Q^a)^*( Q^b)^*=-\sum_{a+b-c>pc,\,\, 2c>-j}(-1)^{a-c}\binom{(p-1)(b-c)-1}{a-pc-1} ( Q^{a+b-c})^*( Q^c)^*,
\end{equation*}
in $\pazocal{F}\big(j, j-2(p-1)a-2(p-1)b\big)$ for all $a,b\in\mathbb{Z}$ and  satisfying $ a \leq pb$, $2b>-j$, $2a>2(p-1)b-j$,
\begin{align*}
   (\beta Q^a)^*(Q^b)^*=\sum_{a+b-c\geq pc, 2c>-j}(-1)^{a-c}\binom{(p-1)(b-c)}{a-pc} (Q^{a+b-c})^*(\beta Q^c)^*\\-\sum_{a+b-c>pc, 2c>-j}(-1)^{a-c}\binom{(p-1)(b-c)-1}{a-pc-1} ( \beta Q^{a+b-c})^*( Q^c)^*
\end{align*}
in $\pazocal{F}\big(j, j-2(p-1)a-2(p-1)b+1\big)$ for all $a,b\in\mathbb{Z}$ satisfying $ a \leq pb$, $2b>-j$, $2a>2(p-1)b-j$, and
\begin{equation*}
   (\beta^\epsilon Q^a)^*(\beta Q^b)^*=\sum_{a+b-c\geq pc, 2c>-j}(-1)^{a-c}\binom{(p-1)(b-c)-1}{a-pc} (\beta^\epsilon Q^{a+b-c})^*(\beta Q^c)^*
\end{equation*}
in $\pazocal{F}\big(j, j-2(p-1)a-2(p-1)b+\epsilon+1\big)$ for $\epsilon\in\{0,1\}$ and $a,b\in\mathbb{Z}$ satisfying $ a < pb$, $2b>-j$, $2a>2(p-1)b-j-1$.
\end{definition}
A basis for $\R_u^!\big((s,j),(s-k,-)\big)$ is given by sequences $(\beta^{\epsilon_1}Q^{i_1})^*(\beta^{\epsilon_2}Q^{i_2})^*\cdots(\beta^{\epsilon_k}Q^{i_k})^*$ where $2i_k> -j$ and $i_l> p i_{l+1}-\epsilon$ for $1\leq l<k$. The functor $\free^{\R_u^!}$ takes the free  $(\pazocal{R}')^!$-module with a homologial grading that counts the number of generators,  i.e., the element $$ (\beta^{\epsilon_1}Q^{i_1})^*(\beta^{\epsilon_2} Q^{i_2})^*\cdots ( \beta^{\epsilon_k}Q^{i_k})^*(x)\in\free^{{\pazocal{R}_u}^!}(M)$$ with $x\in M$ has homological degree $-k$. 
\begin{remark}\label{rmk: suspensionodd}
Similar to the case $p=2$ (\Cref{rmk: desuspension2}), there is an isomorphism $$\R_u^!((s,i),(s',j))\cong \R_u^!((s-r,i),(s'-r,j))$$
for any $i,j, s, s'$ and $r$ such that $s-r<0$. For any $t>0,$ there is an injection $$\mathrm{susp}^t:\R_u^!((s,i),(s',j))\hookrightarrow \R_u^!((s,i+t),(s',j+t)),$$ since more operations are defined on classes with higher homological degree.
\end{remark}
\begin{remark}\label{rmk: bottom op oddclass}
For a class $x$ in odd internal degree $2d+1$, the right hand side of the relations never involve the bottom operations on $x$, i.e. the terms of the form  $(\beta^\epsilon Q^{a+b+d})^*(\beta Q^{-d})^*(x)$, unless $a=(p-1)b-d$ so $(\beta Q^a)^*$ is the bottom operation on the odd degree class $(Q^b)^*(x)$. This is because the coefficient of such a term is nonzero only if $(p-1)(b+d)-\epsilon\geq a+pd$, i.e., $(p-1)b-d-\epsilon\geq a$. But $2a\geq 2(p-1)b-2d$ by assumption. On the other hand, if $x$ is in even internal degree $2j$, then the bottom operation $(\beta Q^{-d+(p-1)b+1})^*$ on the odd class $(Q^b)^*(x)$ is zero. 
\end{remark}
The proof for \Cref{cor: unext p=2} works verbatim for the case $p>2$.
\begin{corollary}\label{cor: unext p>2}
    The unstable $\mathrm{Ext}$ group
$$\mathrm{UnExt}^{*,*}_{\R'}(\Fp,\Sigma^{-j} \Fp)=\pi_*(\B(\id, \pazocal{A}_{\R'}, \Sigma^{-j} \Fp)^{\vee})$$ is the free $\R_u^!$-module  $\R_u^!((0,j),-)$.
\end{corollary}
  Whereas the Tor group $\mathrm{Tor}^{*,*}_{\R'}(\Fp,\Fp)\cong {({\pazocal{R}}_u^!)}^{\vee}$ is a coalgebra generated by classes in $$\mathrm{Tor}^{1,*}_{\R'}(\Fp,\Fp)=\Fp\{[ \beta^{\epsilon}Q^i]1, i\in\N, \epsilon=0,1\}.$$ 
%The unstable Tor groups are cofree comodules over the co-ringoid ${(\R_u^!)^{\vee}}$.

Now we can compute the unary operations on an odd generator using the dual bar spectral sequence for a trivial $\E^{\mathrm{nu}}_\infty\tens \Fp$-algebra $\Sigma^j \Fp$.
\begin{proposition}\label{E2unaryodd}
Let $A=\Sigma^j \Fp$. If $j$ is odd, then $\pi_{*}\Lie^{\pi}_{\Fp,\E_\infty}(A^{\vee})\cong \free^{\R_u^!}(\Sigma^{-j}\Fp)$ with respect to the total degree $s+t$. 
\end{proposition}
\begin{proof}
  If $A=\Sigma^j \Fp$ with $j$ odd, then $\pi_{*,*}\B(\id, \p_{\Fp}, \Sigma^j \Fp)$
 consists of a single $\Sigma^{j} \Fp$, since the tensor square of an odd class always vanishes due to graded commutativity. Thus the dual Grothendieck spectral sequence (\Cref{lem: Grothendieck}) collapses on the $E^2$-page, and the $E^2$-page of the dual bar spectral sequence is concentrated in a single line $s=-k$ at each weight $p^k$ for $k\in \mathbb{N}$. Hence the  dual bar spectral sequence also collapses on the second page and there are no extension problems.  
\end{proof}

When the input trivial $\E^{\mathrm{nu}}_\infty\tens \Fp$-algebra $A=\Sigma^j \Fp$ has $j$ even, the dual bar spectral sequence and the associated Grothendieck spectral sequence collapse on the $E^2$-page at weights all powers of $p$. We show this by an inductive degree comparison with the $\Fp$-basis of the $E^\infty$-page given by Brantner-Mathew in \Cref{dimension}.

\begin{lemma}\label{cor: E2unaryeven}
    Suppose that $p>2$ and $k\geq 1$. In weight $p^k$, both he dual bar spectral sequence for $A=\Sigma^j \Fp$ with $j$ even and its associated dual Grothendieck spectral sequence collapse on the $E^2$-page. Hence the weight $p^k$ part of $\pi_*(\spla(\Sigma^j \Fp))$ is isomorphic to the weight $p^k$ part of the free $\R^!_u$-module on a class $x_j$ in degree $j$.
\end{lemma}
\begin{proof}
    When $j$ is even, it follows from \Cref{dimension} that the weight $p$ unary operations on a degree $j$ class consists of one in each total degree $-2(p-1)i-\epsilon$ with $\epsilon=0,1$ and $i>0$. On the weight $p$ part of the $E^2$-page of the dual Grothendieck spectral sequence, the entries with these total degrees are exactly 1-dimensional, which is in one to one correspondence to the collection of generators $(\beta^{\epsilon}Q^i)^*(x_j)$ with $\epsilon=0,1$ and $i>-j/2$ of the weight $p$ part of the free $\R^!_u$-module on a class $x_j$ in degree $j$. 
    
    Note that any other class of weight $p$ on the $E^2$-page of the dual Grothendieck spectral sequence has to come from $\mathrm{Bar}_\bullet(\id,\p_{\Fp},\Fp\{x_j\})^\vee$ and consists of $p$-fold product of $x_j$ with product possibly coming from different iterations of $\p_{\Fp}$. But such elements, if they exist, have the same internal degree $pj$ and the smallest possible total degree $pj-p>pj-2(p-1)$, so they cannot support differentials among themselves and cannot survive to the $E^\infty$-page of the dual bar spectral sequence for degree reasons. Hence they do not exist.

    Now we induct on $k$. Suppose that the weight $p^i$ part of both spectral sequences collapse, so that the weight $p^i$ part of the $E^\infty$-page  are in bijection with weight $p^i$ part of the free $\R^!_u$-module on a class $x_j$ in degree $j$ and thus consists of iterations of weight $p$ operations.
    In weight $p^{k+1}$, the $\Fp$-basis of $\pi_*(\spla(A))$ given in \Cref{dimension} is in one to one correspondence to the collection $(\beta^{\epsilon_1}Q^{i_1})^*(\beta^{\epsilon_2}Q^{i_2})^*\cdots (\beta^{\epsilon_1}Q^{i_{k+1}})$ where $2i_{k+1}> -j$ and $i_l> p i_{l+1}-\epsilon$ for $l$. The bijection with a basis for the weight $p^{k+1}$ part of the free $\R^!_u$-module on a class $x_j$ in degree $j$ is as follows: given a monomial $(\beta^{\epsilon_1}Q^{i_1})^*(\beta^{\epsilon_2}Q^{i_2})^*\cdots(\beta^{\epsilon_k}Q^{i_k})^* (x_j)$, we send it to the sequence $(-2(p-1)i_1+\epsilon_1,\ldots, -2(p-1)i_k+\epsilon_k,0,x_j)$. By the inductive hypothesis, on the $E^2$-page of the dual Gronthendieck spectral sequence there are no weight $p^i$ classes coming from $\pi_*\mathrm{Bar}_\bullet(\id,\p_{\Fp},\Fp\{x_j\})^\vee$ for $i\leq k$. Furthermore, any potential weight $p^{k+1}$ therein would have the same internal degree and the smallest possible total degree $p^k(j-1)$, which is greater than any nonzero classes in the final answer. Hence they cannot support differentials among themselves and cannot survive to the $E^\infty$-page of the dual bar spectral sequence for degree reasons. This concludes the proof.
\end{proof}
\begin{remark}
    With some more bookkeeping, one can show with the same reasoning that the dual bar spectral sequence for $A=\Sigma^j\Fp$ and the associated dual Grothendieck spectral sequence collaspse at all weights when $j$ is even. For our purpose, we only need the degeneration at weights powers of $p$. 
\end{remark}

 Since for $p>2$, the weight $p$ part of the dual bar spectral sequence on a single generator and the associated Grothendieck spectral sequence collapse at the $E^2$-page, we can deduce a suspension isomorphism as in \Cref{rmk: additive p=2}.
 \begin{corollary}\label{cor: additive p>2}
 The weight $p$ part of the  canonical map $\Sigma\pazocal{L}_p(\Sigma^{j}\Fp)\rightarrow\pazocal{L}_p(\Sigma^{j+1}\Fp)$ induces a map of the $E^2$-pages of the dual bar spectral sequence, sending  $\sigma (\beta^\epsilon Q^i)^*(x)$ to $(\beta^\epsilon Q^i)^*(\sigma x)$ where $x$ is in internal degree $j$ and $i>-j/2$, which agrees with the suspension map $\mathrm{susp}^1$ in \Cref{rmk: suspensionodd}. 
  \end{corollary}

  \subsection{Unary operations on  the homotopy groups of spectral partition Lie algebras}
  Recall that the homotopy groups of  the free spectral partition Lie algebra on $\Sigma^{j}\Fp$ parametrizes unary operations on a degree $j$ homotopy class of a spectral partition Lie algebra and a degree $-j$ cohomology class in mod $p$ TAQ cohomology.
   Now we give a concrete construction of  unary operations
 on  the homotopy groups of any spectral partition Lie algebra $A$ using the collapse of the weight $p^k$ dual bar spectral sequence on a single generator as well as the associated Grothendieck spectral sequence for all $p$ and $k\geq 1$ (\Cref{E2unary}, \Cref{E2unaryodd} and \Cref{cor: E2unaryeven}). 
 
 \begin{construction}\label{unary}
Suppose that $\xi:\Sigma^{j} \Fp\rightarrow A$ represents a homotopy class $x\in \pi_{j}(A)$.

(1) Suppose that $p=2$. For any $r\in\mathbb{N}$ and sequence $$\alpha=(Q^{i_1})^*(Q^{i_2})^*\cdots (Q^{i_k})^*\in \R_u^!((0,2^rj-2^{r}+1),(-k,2^rj-2^{r}+1-m))$$ with $i_1+\cdots+i_k=m,$ there is a unique class  $R^{(i_1+1,i_2+1,\ldots,i_k+1)}(x^{\{2\}^r})\in \pi_{2^rj-2^r+1-m-k}(A)$ given by
 \begin{align*}
   \Sigma^{2^rj-2^r+1-m-k}\F\xrightarrow{\alpha} \free^{\R_u^!}(\Sigma^{j}\F)\hookrightarrow  \pi_*(\Lie^{\pi}_{\F,\E_{\infty}}(\Sigma^{j}\F))\xrightarrow{\xi_*} \pi_*(\Lie^{\pi}_{\F,\E_{\infty}}(A))\rightarrow\pi_*(A).
 \end{align*}

(2) Suppose that $p>2$. For any sequence $$\alpha=(\beta^{\epsilon_1}Q^{i_1})^*(\beta^{\epsilon_2}Q^{i_2})^*\cdots (\beta^{\epsilon_k}Q^{i_k})^*\in \R_u^!((0,j),(-k,j-m))$$ with $m=2(p-1)i_1+\cdots+2(p-1)i_k-\epsilon_1-\cdots-\epsilon_k,$ there is a unique class  $R^{(i_1,\ldots,i_k,1-\epsilon_1,\ldots,1-\epsilon_k)}(x)\in \pi_{j-m-k}(A)$ given by
 \begin{align*}
   \Sigma^{j-m-k}\Fp\xrightarrow{\alpha} \free^{\R_u^!}(\Sigma^{j}\Fp)\hookrightarrow \pi_*\Lie^{\pi}_{\Fp,\E_{\infty}}(\Sigma^{j}\Fp)\xrightarrow{\xi_*}  \pi_*\Lie^{\pi}_{\Fp,\E_{\infty}}(A)\rightarrow\pi_*(A).
 \end{align*}
 Translating to cohomological grading, these operations become cohomology operations on $\taq(A)$ for $A$ an $\E_\infty$-$\Fp$-algebra.
\end{construction}

We note that they are compatible with the suspension map.
\begin{proposition}\label{prop: stable}
    The operations $R^I$ and $R^{(I,\vec{\epsilon})}$ above are stable in the sense that any such operation $\alpha:\pi_m(L)\rightarrow\pi_{m-|\alpha|}(L)$  agrees with $\alpha:\pi_{m+1}(L)\rightarrow\pi_{m+1-|\alpha|}(L)$ for any spectral partition Lie algebra $L$. 
\end{proposition}
\begin{proof}
 The canonical map $\Sigma\Lie^\pi_{\Fp,\E_\infty}(\Sigma^j\Fp)\rightarrow\Lie^\pi_{\Fp,\E_\infty}(\Sigma^{j+1}\Fp)$ induces a map 
 of the weight $p$ part of the $E^2$-pages of the corresponding dual bar spectral sequences and the associated Grothendieck spectral sequences. Since both  spectral sequence collapse by \Cref{E2unary}, \Cref{E2unaryodd} and \Cref{cor: E2unaryeven}, it suffices to show that the canonical map $\Sigma \free^{\R^!_u}(\Sigma^j \Fp)\rightarrow\free^{\R^!_u}(\Sigma^{j+1}\Fp)$ sends $\sigma(Q^i)^*(x_j)$ to  $(Q^i)^*(\sigma x_j)$ for $p=2$, where $x_j$ has internal degree $j$ and $i>-j$, and $\sigma(\beta^{\epsilon}Q^i)^*(x_j)$ to  $(\beta^{\epsilon}Q^i)^*(\sigma x_j)$ for $p>2$, where $i>-j/2$. This  follows from \Cref{rmk: desuspension2}, \Cref{rmk: suspensionodd} and \Cref{cor: additive p>2}.
\end{proof}
Instead of proceeding to examine the relations among the unary operations of weight $p^k$, we take a detour to construct the structure of a shifted restricted Lie algebra on the homotopy groups of spectral partition Lie algebras. 

\section{A shifted Lie algebra structure with restriction}\label{sec: Lie}
In this section, we construct a shifted Lie bracket on the homotopy groups of spectral partition Lie algebras, or equivalent the $\Fp$-linear TAQ cohomology groups. We start by examining a homotopy fixed points spectral sequence, which reveals much of the shifted restricted Lie structure in question.

\subsection{A homotopy fixed points spectral sequence} \label{subsec: hfpss}
Since the category of $\Fp$-module spectra is equivalent to the derived category of chain complexes over $\Fp$ (\cite[7.1.1.16]{lurie}), for $A$ a bounded above $\Fp$-module spectrum, we can think of the homotopy groups of $|\B(\id, \E^{\mathrm{nu}}_\infty\tens \Fp, A^\vee)|^\vee$ as the homology of the chain complex $$C=C_*(\mathbb{D}|\B(\id, \E^{\mathrm{nu}}_\infty\tens \Fp, A^\vee)|)\cong \bigoplus_n (\Lie_{dg}^{s}(n)\tens A^{\tens n})^{h\Sigma_n}$$ by \cite[Remark 4.49]{pd}.  Here $\Lie_{dg}^{s}$ is the shifted Lie operad in the derived category of chain complexes over $\Fp$, with $\Lie_{dg}^{s}(r)$ concentrated in homological degree $1-r$. Then there is an associated homotopy fixed points spectral sequence  
\begin{equation}\label{hfpss}
        E^2_{s,t}=\bigoplus_{n} H^s\Big(\Sigma_n, \pi_t\big(\Lie_{dg}^{s}(n)\tens A^{\tens n}\big)\Big)\Rightarrow \bigoplus_{n}\pi_{t-s}\Big(\big(\Lie_{dg}^{s}(n)\tens A^{\tens n}\big)^{h\Sigma_n}\Big),
\end{equation}
which has a weight decomposition into the above summands induced by considering $A$ as a weighted spectrum concentrated in weight 1.

\begin{lemma}\label{lem: hfpsss=0}
In the homotopy fixed points spectral sequence (\ref{hfpss}) above, the line $s=0$ of the $E^2$-page is given by the free shifted restricted Lie algebra $\free^{\Lie^{s,\rho}_{\Fp}}(\pi_*(A))$ on $\pi_*(A)$.
\end{lemma}
\begin{proof}
    Note that the homology operad associated to $\Lie_{dg}^{s}$ is the shifted Lie operad $\pi_*(\Lie_{dg}^s)\cong \Lie^s$ over $\Fp$.
    The line $s=0$ of the $E^2$-page is given by $$\bigoplus_{n} H^0\Big(\Sigma_n, \pi_t\big(\Lie_{dg}^{s}(n)\tens A^{\tens n}\big)\Big)\cong \bigoplus_{n}\Lie^s(n)\tens^{\Sigma_n} \pi_*(A)^{\tens n}\cong \free^{\Lie^{s,\rho}_{\Fp}}(\pi_*(A)),$$ where the last equivalence follows from the shifted version of \cite[Theorem 1.2.5]{fresse}.
\end{proof}

\begin{lemma}\label{lem: hfpsswtp}
Let $p$ be any prime. If $A=\Sigma^j \Fp$, then the homotopy fixed spectral sequence (\ref{hfpss}) collapses on the $E^2$-page in weight $p$. If $p>2$ and $j$ is even, then it collapses in all weight.
\end{lemma}

\begin{proof}
    Since $\pi_t(\Lie^{s}(p)\tens (\Sigma^j\Fp)^{\tens p})=0$ unless $t=pj+1-p$, the weight $p$ part of the $E^2$-page of the homotopy fixed points spectral is concentrated on a single line $t=pj+1-p$. Hence weight $p$ part of the spectral sequence collapses and there are no extension problems.

    If $p>2$ and $j$ is even, then $\free^{\Lie^{s,\rho}}_{\Fp}(\Sigma^j\Fp)=\Fp\{x_j, [x_j,x_j]\}$ by \Cref{prop: basis-restrictedlie}. It follows from \Cref{lem: hfpsss=0} that the only nonzero elements are concentrated in weight $w=p^k$ or $2p^k$ as group cohomology classes of $H^*\Big(\Sigma_{w}, \pi_{w(j-1)}\big(\Lie_{dg}^{s}(w)\tens (\Sigma^j\Fp)^{\tens w}\big)\Big)$. Hence at each weight the nonzero elements again lie on a single line and the claim follows.
\end{proof}

\begin{corollary} \label{cor: hfpsswtp two gen}
    Let $p$ be any prime. Suppose that $A=\Sigma^{j_1}\Fp\oplus\Sigma^{j_2}\Fp$. The weight $p$ part of the homotopy fixed points spectral sequence (\ref{hfpss}) collapses on the $E^2$-page.
\end{corollary}
\begin{proof}
 The weight $p$ part of $\spla(A)$ has summands $\Lie_{dg}^{s}(p)\tens^{h\Sigma_p} (\Sigma^{j_i}\Fp)^{\tens p}$ for $i=1,2$. 
 The inclusions $\pi_*(\spla(\Sigma^{j_i}\Fp))\rightarrow \pi_*(\spla(A))$
detect all the higher cohomology classes on the $E^2$-page of the homotopy fixed points spectral sequence (\ref{hfpss}) for $A$, hence they are all permanent cycles. The remaining classes on the $E^2$-page are concentrated in $s=0$ and thus no further differential can happen.
\end{proof}

\subsection{A shifted Lie bracket with restriction}
Now we would like to lift the shifted restricted Lie algebra structure on the $E^2$-page of the homotopy fixed points spectral sequence (\ref{hfpss}).

 Recall that the second Goodwillie derivative of the identity functor $\partial_2(\mathrm{Id})\simeq \mathbb{S}^{-1}$ is a na\"{i}ve $\Sigma_2$-spectrum with trivial $\Sigma_2$-action. By Proposition \ref{bm535}, the weight 2 part of the free spectral partition Lie algebra on a bounded $\Fp$-module $A$ is $$\big((\partial_2(\mathrm{Id})\tens \Fp)\tens^{h\Sigma_2} A^{\tens 2}\big)\xhookrightarrow{\iota_2}\Lie^{\pi}_{\Fp,\E_{\infty}}(A)\simeq \bigoplus_n (\partial_n(\mathrm{Id})\tens \Fp)\tens^{h\Sigma_n} (A)^{\tens n}.$$ Here $\iota_n$ denotes the inclusion of the weight $n$ homogeneous piece.
 
 The binary operation $[-,-]$ representing the shifted Lie bracket is encoded by the weight two part of the structure map
 $$\xi_2:\mathbb{S}^{-1}\tens (A^{\tens 2})^{h\Sigma_2}\simeq (\partial_2(\mathrm{Id})\tens \Fp)\tens^{h\Sigma_2} A^{\tens 2}\xrightarrow{\iota_2} \Lie^{\pi}_{\Fp,\E_{\infty}}(A)\rightarrow A,$$ 
 %which descends to a  pairing $$[\ ,\ ]:\pi_m(A)\tens \pi_n(A)\rightarrow\pi_{m+n-1}(A).$$ Therefore the homotopy group of any $\Lie^{\pi}_{\Fp,\E_{\infty}}$-algebra $A$ admits a shifted Lie bracket
 explicitly given as follows.
 \begin{construction}\label{bracket}

 For $x:X=\Sigma^{j} \Fp\rightarrow A, y:Y=\Sigma^{k} \Fp\rightarrow A$ representing two homotopy classes of a spectral partition Lie algebra $A$, we have a map of  $\Lie^{\pi}_{\Fp,\E_{\infty}}$-algebras
 $$\theta:\Lie^{\pi}_{\Fp,\E_{\infty}}(X\oplus Y)\rightarrow \Lie^{\pi}_{\Fp,\E_{\infty}}(A)\rightarrow A,$$ where the second map is the $\Lie^{\pi}_{\Fp,\E_{\infty}}$-algebra structure map of $A$.  There is a binary operation $[-,-]$ on $\pi_*(A)$ represented by  the map
\begin{equation}\label{liebracket}
\mathbb{S}^{-1}\tens X\tens Y\simeq  \mathbb{S}^{-1}\tens(X\tens Y\oplus Y\tens X)^{h\Sigma_2}\hookrightarrow \mathbb{S}^{-1}\tens\big((X\oplus Y)^{\tens 2}\big)^{h\Sigma_2}\simeq\partial_2(\mathrm{Id})\tens^{h\Sigma_2}(X\oplus Y)^{\tens 2}
\end{equation}
followed by the composite $$\theta\circ\iota_2: \partial_2(\mathrm{Id})\tens^{h\Sigma_2}(X\oplus Y)^{\tens 2}\hookrightarrow\Lie^{\pi}_{\Fp,\E_{\infty}}(X\oplus Y)\rightarrow A,$$
which sends the pair of homotopy classes $x$ and $y$ to a class $[x,y]\in\pi_*(A)$.
\end{construction}
Translating to cohomological grading,  we constructed a binary operation 
$$[-,-]: \taq^m(A)\tens \taq^n(A)\rightarrow \taq^{m+n+1}(A)$$ for any $\E^{\mathrm{nu}}_{\infty}$-$\Fp$-algebra $A$.

Note that the homotopy fixed points spectral sequence 
    \begin{equation}
        E^2_{s,t}=\bigoplus_{n} H^s\Big(\Sigma_n, \pi_t\big(\Lie_{dg}^{s}(n)\tens A^{\tens n}\big)\Big)\Rightarrow \bigoplus_{n}\pi_{t-s}\Big(\big(\Lie_{dg}^{s}(n)\tens A^{\tens n}\big)^{h\Sigma_n}\Big),
    \end{equation}  
collapses in any weight $q$ that is prime to $p$ since higher cohomology groups of $\Sigma_q$ vanish and in weight 2 at $p=2$ because of \Cref{cor: hfpsswtp two gen}. Since the $s=0$-line is the free restricted Lie algebra on $\pi_*(A)$ by \Cref{lem: hfpsss=0}, the binary operation $[x,y]\in \pi_*(A)$ is represented uniquely up to a nonzero scalar by the weight 2 part of the composite
  \begin{align*}
     \Sigma^{-1}(\Sigma^j \Fp\oplus\Sigma^k \Fp)\rightarrow \free^{\Lie^{s,\rho}_{\Fp}}(\Sigma^j \Fp\oplus\Sigma^k \Fp)\hookrightarrow \pi_*(\Lie^{\pi}_{\Fp,\E_{\infty}}(\Sigma^j \Fp\oplus\Sigma^k \Fp))\xrightarrow{\theta_*}\pi_*(A).
 \end{align*}

Here we are using the fact that a shifted Lie algebra is an algebra over the quadratic shifted Lie operad generated by one element in $\Lie^s(2)$, so we can fix a choice of this generator for the shifted Lie bracket on the $E^2$-page of the homotopy fixed points spectral sequence so that $c=1$. By abuse of notation we use $[-,-]$ to denote both the binary operation on $\pi_*(A)$ and the shifted Lie bracket on the $E^2$-page.

Now we show that this binary operation is indeed a shifted Lie bracket in a general sense.

\begin{proposition}
The binary operation $[-,-]$ in Construction \ref{bracket} satisfies graded commutativity $[x,y]=(-1)^{|x||y|}[y,x]$ and the graded Jacobi identity $$(-1)^{|x||z|}[x,[y,z]]+(-1)^{|x||y|}[y,[z,x]]+(-1)^{|y||z|}[z,[x,y]]=0.$$  
\end{proposition}
We shall see in the next proposition that $[x,x]=0$ for all $x$ at $p=2$ and $[x,[x,x]]=0$ for all $x$ at $p=3$. Hence $[-,-]$ equips the homotopy groups of any spectral partition Lie algebra with a shifted Lie algebra structure as is the convention of this paper.
 \begin{proof}
Graded commutativity $[x,y]=(-1)^{|x||y|}[y,x]$ follows by construction, with the sign coming from the induced action of the transposition $(12)\in \Sigma_2$. To check  the graded Jacobi identity of the shifted bracket, we use an argument adapted from \cite{omar}. Let $A=\Sigma^{j}\Fp\oplus\Sigma^{k}\Fp\oplus \Sigma^l \Fp$. The iteration $[[-,-],-]$ of the binary operation $[-,-]$  is given by
%\begin{align*}
 %    \Sigma^{-2}(\Sigma^j \Fp\oplus\Sigma^k \Fp\oplus \Sigma^l\Fp)&\rightarrow \mathrm{wt}_3\big(\free^{\Lie^{s}_{\Fp}}\free^{\Lie^{s}_{\Fp}}(\Sigma^j \Fp\oplus\Sigma^k \Fp\oplus\Sigma^l\Fp)\big)\\   & \cong \pi_*\big((\partial_3(\mathrm{Id})\tens \Fp)\tens^{h\Sigma_3} A^{\tens 3}\big)\\     &\xrightarrow{(\iota_3)_*}\pi_*(\Lie^{\pi}_{\Fp,\E_{\infty}}(A))\xrightarrow{\theta_*}\pi_*(A), \end{align*}
the weight 3 summand
\begin{align*}
    &(\partial_2(\mathrm{Id})\tens \Fp)\tens\Big(\big((\partial_1(\mathrm{Id})\tens \Fp)\tens^{h\Sigma_1}  A \big)\tens\big((\partial_{2}(\mathrm{Id})\tens \Fp)\tens^{h\Sigma_2} A^{\tens 2}\big)\Big)\\ \xhookrightarrow{(\ref{liebracket})}&\Big[( \partial_2(\mathrm{Id})\tens \Fp)\tens^{h\Sigma_2}\Big(\big((\partial_1(\mathrm{Id})\tens \Fp)\tens^{h\Sigma_1}  A \big)\oplus\big((\partial_{2}(\mathrm{Id})\tens \Fp)\tens^{h\Sigma_2} A^{\tens 2}\big)\Big)^{\tens 2}\Big][1] \\ \rightarrow&(\partial_3(\mathrm{Id})\tens \Fp)\tens^{h\Sigma_3}A^{\tens 3}
\end{align*}
 of the monad composition $\Lie^{\pi}_{\Fp,\E_{\infty}}\circ \Lie^{\pi}_{\Fp,\E_{\infty}}\rightarrow\Lie^{\pi}_{\Fp,\E_{\infty}}$ applied to $A$. Since $\partial_2(\mathrm{Id})\simeq \mathbb{S}^{-1}$ and  $\partial_1(\mathrm{Id})\simeq \mathbb{S}$ both have trivial actions by the symmetric groups, we deduce that the source of the above structure map is equivalent to $$\partial_2(\mathrm{Id})\tens (\partial_1(\mathrm{Id})\tens \partial_2(\mathrm{Id}))\tens \Fp \tens (A^{\tens 3})^{h\Sigma_3}.$$ Denote by $\nu$ the structure map $\partial_2(\mathrm{Id})\tens (\partial_1(\mathrm{Id})\tens \partial_2(\mathrm{Id}))\rightarrow\partial_3(\mathrm{Id})$. The graded Jacobi identity
$$(-1)^{|x||z|}[x,[y,z]]+(-1)^{|x||y|}[y,[z,x]]+(-1)^{|y||z|}[z,[x,y]]=0$$ 
is then equivalent to showing that $\nu+(\sigma)_*\nu+(\sigma)^2_*\nu$ is null-homotopic, where $(\sigma)_*$ is the induced action of the cyclic permutation $(123)$. It was proved in \cite[Proposition 5.2]{omar} that $\nu+(\sigma)_*\nu+(\sigma)^2_*\nu$ is null-homotopic. 
\end{proof}

Next we investigate the interaction between the shifted Lie bracket $[-,-]$ and the unary operation in Construction \ref{unary}.
 
 \begin{proposition}\label{vanish}
 Given any $x,y\in\pi_*(A)$ and  unary operation $\alpha$ of weight greater than one, we have $[x, \alpha(y)]=0$ unless $p=2$ and $\alpha$ is an iteration of the bottom operations on $y$ or $p>2$, $y$ is in odd degree, and $\alpha$ is an iteration of bottom unary operations on $y$. In particular $[x,x]=0$ for all $x$ when $p=2$ and $[x,[x,x]]=0$ for all $x$ when $p=3$.
 \end{proposition}
\begin{proof}
We use an argument adapted from \cite[Proposition 4.3.15]{brantner}. For ease of notations we will not be using the logarithmic grading convention, and will denote the weight $k$ part by appending $(k)$ on the right. Suppose that $\alpha$ is a nonempty sequence of operations with weight $w$ divisible by $p$,  and $x:\Sigma^j \Fp\rightarrow A, y:\Sigma^k \Fp\rightarrow A$ representing two homotopy classes of a spectral partition Lie algebra $A$.  The operation $[x,\alpha(y)]$ is encoded by the weight $w+1$ part of the structure map  $\Lie^{\pi}_{\Fp,\E_{\infty}}\circ \Lie^{\pi}_{\Fp,\E_{\infty}}\rightarrow\Lie^{\pi}_{\Fp,\E_{\infty}}$, i.e.,
 \begin{align}\label{w+1}
 \begin{split}
      \Sigma^{j+k+|\alpha|-1}\Fp
      &\rightarrow d_2\tens\Big(\big(d_1\tens^{h\Sigma_1}  \Sigma^j \Fp\big)\tens\big(d_w\tens^{h\Sigma_w} (\Sigma^k \Fp)^{\tens w}\big)\Big)\\
      &\xhookrightarrow{(\ref{liebracket})} \Big[d_2\tens^{h\Sigma_2}\Big(\big(d_1\tens^{h\Sigma_1}  \Sigma^j \Fp \big)\oplus\big(d_w\tens^{h\Sigma_w} (\Sigma^k \Fp)^{\tens w}\big)\Big)^{\tens 2}\Big](w+1)\\
      &\rightarrow d_{w+1}\tens^{h(\Sigma_1\times\Sigma_{w})} (\Sigma^{j}\Fp\otimes\Sigma^{kw}\Fp)\\
     &\hookrightarrow\Lie^{\pi}_{\Fp,\E_{\infty}}(\Sigma^{j}\Fp\oplus\Sigma^{k}\Fp)\rightarrow \Lie^{\pi}_{\Fp,\E_{\infty}}(A)\rightarrow A,
  \end{split}
 \end{align}
 where we write $d_n$ for $\partial_n(\mathrm{Id})\tens \Fp$ for ease of notation.
Note that the action obtained by restriction to $\Sigma_1\times \Sigma_w\subset \Sigma_{w+1}$ on $\partial_{w+1}(\mathrm{Id})$ is freely induced from the action of the trivial subgroup on $\mathbb{S}^{-w}$ (cf. \cite[proof of 4.3.15]{brantner}). For any finite group $G$, we have $(\mathrm{Ind}^{G}_{\{e\}}(X)\tens Y)^{hG}\simeq X \tens Y$ for all $G$-spectra $Y$ by the Wirthm\"{u}ller isomorphism. Hence 
$$(\partial_{w+1}(\mathrm{Id})\tens \Fp)\tens^{h(\Sigma_1\times\Sigma_{w})} \Sigma^{j+kw}\Fp\simeq\Sigma^{j+(k-1)w}\Fp.$$
In particular, the $\Fp$-module of weight $w+1$ operations on spectral partition Lie algebras  coming from the bracket of one weight one operation and one weight $w$ operation is one-dimensional with a generator given by  $[[\cdots[[x,y],y]\cdots],y]$, where we take the bracket with $y$ exactly $w$ times. In the free case this class is nonzero, since the associated homotopy fixed points spectral sequence collapses at weight $w+1$ and this bracket is a nonzero element on the line $s=0$.

If $\alpha$ represents an iteration of the self-bracket, then $[x,\alpha(y)]=0$ by the Jacobi identity when $p\neq 3$. When $p=3$ and $|x|=k$ is even, $[x,[x,x]]$ is in degree $3k-2$. The weight 3 part of the homotopy fixed points spectral sequence on a single generator, which collapses on the $E^2$-page by \Cref{lem: hfpsswtp}, is 0 in total degree $3k-2$, and in particular the element representing $[x,[x,x]]$ is 0.

Suppose that $\alpha$ is not an iteration of the self-bracket, and not an iteration of the bottom operations on $x$ when $p=2$, or an iteration of the bottom operation $R^{(-|x|+1)/2}$ on odd $x$ when $p>2$. Then a comparison of total degrees  shows that $[x,\alpha(y)]$ has strictly smaller total degree than the generator $\gamma=[[\cdots[[x,y],y]\cdots],y]$ of weight $w+1$ operations  coming from the bracket of one weight one operation and one weight $w$ operation. Therefore it has to be zero.

Finally, when $p=2$, there is only one unary operation of degree $-|x|+1$ on any class $x$ up to a scalar, and in the homotopy fixed points spectral sequence this class is given by the restriction. Hence we deduce that $[x,x]=0$ since taking self-bracket is an additive operation while the restriction is not.
\end{proof}

On the other hand. we expect to see a shifted restricted Lie algebra structure  with the restriction $x^{\{p\}}$ on with  given by the bottom operation $R^{(-|x|+1)/2}(x)$ on any odd degree class $x$ when $p>2$ and the bottom operation $R^{-|x|+1}(x)$ on any class $x$, as is hinted by the $s=0$ line of the $E^2$-page of the homotopy fixed points spectral sequence in \Cref{lem: hfpsss=0}. This was also noted by Basterra and Mandell, cf. \cite[Example 1.8.8]{lawson}. Now we lift the structure of a restriction to the homotopy groups of spectral partition Lie algebras.

\begin{lemma}\label{lem: hfpssE^2}
If $p>2$ and $A=\Sigma^j\Fp$ with $j$ odd, then the restriction $x^{\{p\}}$ on the degree $j$ class $x$ representing $\Sigma^j\Fp$ on the $E^2$-page of the homotopy fixed points spectral sequence (\ref{hfpss}) survives to the $E^\infty$-page and is detected by $\lambda_j R^{(-j+1)/2}(x)\in\pi_{pj+1-p}\spla(A)$ for a nonzero constant $\lambda_j\in\Fp$.

If  $A=\Sigma^j\F$ for any $j$,  then the restriction $x^{\{2\}}$ on the degree $j$ class $x$ representing $\Sigma^j\F$ on the $E^2$-page of the homotopy fixed points spectral sequence (\ref{hfpss}) survives to the $E^\infty$-page and is detected by $R^{-j+1}(x)\in\pi_{2j-1}\spla(A)$.
\end{lemma}
\begin{proof}
If $p>2$ and $j$ is odd, then the $s=0$ line of the weight $p$ part of the $E^2$-page 
$$H^0\big(\Sigma_p; \Lie^{s}(p)\tens (\Sigma^{j}\Fp)^{\tens p})\big)\cong \Lie^{s}(p)\tens^{\Sigma_p} (\Sigma^{j}\Fp)^{\tens p}$$ contains an element that serves as the restriction $x^{\{p\}}$  in $\free^{\Lie^{s,\rho}_{\Fp}}(\Sigma^{j}\Fp)$. Note that $x^{\{p\}}$ is in topological degree $pj+1-p$ and the only element in  the weight $p$ part of $ \pi_{pj+1-p}(\Lie^{\pi}_{\Fp,\E_\infty}(\Sigma^{j}\Fp))$ comes from the bottom operation $(\beta Q^{(-j+1)/2})^*(x)$ up to a unit $\lambda$ on the $E^2$-page of the dual bar spectral sequence (cf. \Cref{E2unaryodd}). Hence the image of $x^{\{p\}}$ in $ \pi_{pj+1-p}(\Lie^{\pi}_{\Fp,\E_\infty}(\Sigma^{j}\Fp))$ under the degeneration of the weight $p$ part of the spectral sequence  is given by $\lambda R^{(-j+1)/2}(x)$.  This unit $\lambda$ is fixed for any class of a given odd degree $j$  by naturality of the restriction map on a class in degree $j$ and the bottom Dyer-Lashof operation on an odd class in degree $j$ as the $p$-fold Massey product on $x$.

If $j$ is even, then $\pi_{pj+1-p}(\Lie^{\pi}_{\Fp,\E_\infty}(\Sigma^{j}\Fp))=0$ (cf. \Cref{E2unaryodd}) so  the class $x$ representing $\Sigma^j\Fp$ does not admit a restriction since, which is as expected for a shifted restricted Lie algebra.

The proof for $p=2$ follows from the exact same dimension counting and degree comparison, and there is no ambiguity for the unit $\lambda$ over $\F$.
\end{proof}

\begin{theorem}\label{restricted}
 The homotopy groups of a spectral partition Lie algebra over $\Fp$ admit a restriction map $x\mapsto x^{\{p\}}$ that coincides with the bottom operation $\lambda_{|x|}R^{(-j+1)/2}$ on any class $x$ in odd degree $j$ when $p>2$, and the bottom operation $R^{-|x|+1}(x)$ on any class $x$ when $p=2$. 
\end{theorem}
\begin{proof}
We spill out the details for the case $p>2$, noting that the case $p=2$ is the same but simpler.

First we show that this assignment is indeed a restriction, in the sense that $[y,\lambda_{|x|}R^{(-|x|+1)/2}(x)]$ is the $p$-fold bracket $[[\cdots[[y,x],x]\cdots],x]$ for an odd class $x$ in the homotopy groups of a $\spla$-algebra.  Take $A=\Sigma^j\Fp\oplus \Sigma^k \Fp$, with $j$ odd, and consider the homotopy fixed points spectral sequence (\ref{hfpss}). The $(p+1)$-th summand of the spectral sequence collapses on the $E^2$-page, since the group cohomology of $\Sigma_{n}$ with coefficients in $\Mod_{\Fp}$ is concentrated in degree 0 when $n$ is coprime to $p$. It follows from \Cref{lem: hfpssE^2} that the restriction $x^{\{p\}}$ on the generator $x$ of $\Sigma^j\Fp$ on the weight $p$ part of the $E^2$-page of (\ref{hfpss}) survives to the element $\lambda_j R^{(-j+1)/2}(x)$. Furthermore, the identity $[y,x^{\{p\}}]=[[\cdots[[y,x],x]\cdots],x]$ on the $E^2$-page survives to an identity in
$\pi_{pj+k-p}(\Lie^{\pi}_{\Fp,\E_\infty}(\Sigma^{j}\Fp\oplus\Sigma^{k}\Fp))$. 

On the other hand, suppose that $j_1$ and $j_2$ are both odd and let $x_1,x_2$ represent the generator of $\Sigma^{j_1}\Fp$ and $\Sigma^{j_2}\Fp$ on the $E^2$-page of the homotopy fixed points spectral sequence (\ref{hfpss}) for $A=\Sigma^{j_1}\Fp\oplus\Sigma^{j_2}\Fp$. The weight $p$ part of $E^2$-page has an $\Fp$-basis given by one generator of each nonzero cohomology group of $\Sigma_p$ with coeffecients in $\Lie^s(p)\tens (\Sigma^{j_i}\Fp)^{\tens p}$ for $i=1,2$, and elements of  a basis $B$ for the weight $p$ part of the free shifted Lie algebra on generators $x_1$ and $x_2$. On the other hand, the weight $p$ part of $E^\infty$-page has an $\Fp$-basis given by the weight $p$ unary operations on $x_1$ and $x_2$ respectively, and elements of  $B$ by \Cref{cor: hfpsswtp two gen}. By \Cref{lem: hfpssE^2}, $x_i^{\{p\}}=\lambda_{j_i}R^{(-j_i+1)/2}(x_i)$ for $i=1,2$. Then the identity $(x_1+x_2)^{\{p\}}=x_1^{\{p\}}+x_2^{\{p\}}+\sum_{i=1}^{p-1} \frac{s_i}{i}(x,y)$ on the line $s=0$ of the $E^2$-page survives to on identity $(x_1+x_2)^{\{p\}}=\lambda_j R^{(-j_1+1)/2}(x_1)+\lambda_k R^{(-j_2+1)/2}(x_2)+\sum_{i=1}^{p-1} \frac{s_i}{i}(x_1,x_2)$ on the $E^\infty$-page via the degeneration of the weight $p$ part of the spectral sequence, where $s_i$ is the coefficient of $t^{i-1}$ in the formal expression $\mathrm {ad} (tx+y)^{p-1}(x)$. %In particular, if $j_1=j_2=j$ then $$ \lambda_j R^{(-j+1)/2}(x_1+x_2)=\lambda_j R^{(-j+1)/2}(x_1)+\lambda_j R^{(-j+1)/2}(x_2)+\sum_{i=1}^{p-1} \frac{s_i}{i}(x_1,x_2).$$

Finally, we want to show that the collection $x^{\{p\}}:=\lambda_j R^{(-|x|+1)/2}(x)$ for $|x|=j$ odd and $\lambda_j$ a unit depending only on $j$ defines a restriction map $(-)^{\{p\}}$ for the shifted Lie bracket on the homotopy groups of spectral partition Lie algebras by extending to linear sums of classes $x,y$ with $|x|=j\neq|y|=k$ odd via  $(x+y)^{\{p\}}:=\lambda_j R^{(-j+1)/2}(x)+\lambda_k R^{(-k+1)/2}(y)+\sum_{i=1}^{p-1} \frac{s_i}{i}(x,y)$.
The $p$th iteration of $[-,x]$ on a class $y$ is encoded
 by a summand in the weight $p+1$ part of the iterated monad composition $$(\Lie^{\pi}_{\Fp,\E_{\infty}})^{\circ p}\rightarrow (\Lie^{\pi}_{\Fp,\E_{\infty}})^{\circ p-1}\rightarrow\cdots\rightarrow \Lie^{\pi}_{\Fp,\E_{\infty}}$$ applied to $\Sigma^j\Fp\oplus \Sigma^k\Fp$. Explicitly, this summand is the $(p-1)$-th iteration of $\partial_2(\mathrm{Id})\tens (\partial_1(\mathrm{Id})\tens(-))$ on $\partial_2(\mathrm{Id})$. Note that the last step  $\Lie^{\pi}_{\Fp,\E_{\infty}}\circ \Lie^{\pi}_{\Fp,\E_{\infty}}\rightarrow\Lie^{\pi}_{\Fp,\E_{\infty}}$  of the above chain of compositions is
 \begin{align*}
       &(\partial_2(\mathrm{Id})\tens \Fp)\tens\Big(\big((\partial_1(\mathrm{Id})\tens \Fp)\tens^{h\Sigma_1}  \Sigma^j \Fp\big)\tens\big((\partial_{p}(\mathrm{Id})\tens \Fp)\tens^{h\Sigma_p} (\Sigma^k \Fp)^{\tens p}\big)\Big)\\
      \rightarrow &(\partial_{p+1}(\mathrm{Id})\tens \Fp)\tens^{h(\Sigma_1\times\Sigma_{p})} (\Sigma^{j}\Fp\otimes\Sigma^{kp}\Fp)
 \end{align*}
 as in (\ref{w+1}), which we showed to be one-dimensional in Proposition \ref{vanish} with $[[\cdots[[y,x],x]\cdots],x]$ where $x$ appears $p$ times a  generator. The monad composition  induces a map of homotopy fixed points spectral sequences, both of which collapse on the $E^2$-page at weight $p+1$ with no extension problems. Hence we get a map that is the weight $p+1$ part of 
 $$\bigoplus_n \Big(\Lie^{s}(n)\tens \big(\pi_*(\bigoplus_m\Lie^{s}(m)\tens^{h\Sigma_m}(\Sigma^j\Fp\oplus\Sigma^k\Fp)^{\tens m})\big)^{\tens n}\Big)^{\Sigma_n}\rightarrow\bigoplus_n \Big( \Lie^{s}(n)\tens (\Sigma^j\Fp\oplus\Sigma^k\Fp)^{\tens n}\Big)^{\Sigma_n}$$
 along the line $s=0$ on the $E^2$-pages. This map has as summand \begin{align*}
     &\Lie^{s}(2)\tens^{\Sigma_2}\Big(\big(\Lie^{s}(1)\tens(\Sigma^j\Fp\oplus\Sigma^k\Fp)^{\tens 1}\big)^{\Sigma_{1}}\oplus\big(\Lie^{s}(p)\tens(\Sigma^j\Fp\oplus\Sigma^k\Fp)^{\tens p}\big)^{\Sigma_{p}}\Big)^{\tens 2}
     \\&\rightarrow\big(\Lie^{s}(p+1)\tens(\Sigma^j\Fp\oplus\Sigma^k\Fp)^{\tens p+1}\big)^{\Sigma_{p+1}},
 \end{align*} with further summand 
 \begin{align*}
     \phi:\Lie^{s}(2)\tens\Big(\Sigma^j\Fp\tens\big(\Lie^{s}(p)\tens(\Sigma^k\Fp)^{\tens p}\big)^{\Sigma_{p}}\Big)\rightarrow\big(\Lie^{s}(p+1)\tens(\Sigma^j\Fp\oplus\Sigma^k\Fp)^{\tens p+1}\big)^{\Sigma_{p+1}}.
 \end{align*}
The map $\phi$ agrees with \Cref{bracket} on the $E^\infty$-page, i.e., it is the evaluation of the free  Lie bracket. The image of $[y, R^{(-|x|+1)/2}(x)]$ under $\phi$ is $[y, x^{\{p\}}]=[[\cdots[[y,x],x]\cdots],x]$ up to a unit $\lambda_{|x|}$. Hence on the $E^\infty$-page we have $[y, \lambda_{|x|} R^{(-|x|+1)/2}(x) ]=[[\cdots[[y,x],x]\cdots],x]$  as desired.

To identify the iteration $x^{\{p\}^k}$ of the restriction map on a class $x$ in odd degree $j$, we represent the class $y=R^{(-j+1)/2}(x)=x^{\{p\}}$ by a map $y:\Sigma^{pj-p+1}\Fp\rightarrow\spla(\Sigma^j\Fp)$. Since $pj-p+1$ is odd, the class $R^{-p(j-1)/2}(y)$ is the restriction $\frac{1}{\lambda_{pj-p+1}}y^{\{p\}}$ on y in $\pi_*\spla(\Sigma^{pj-p+1})$ by \Cref{lem: hfpssE^2}. Its image under the induced composition of homotopy groups of $$\spla(\Sigma^{pj-p+1})\xrightarrow{\spla(y)}\spla\circ\spla(\Sigma^j\Fp)\rightarrow\spla(\Sigma^j\Fp),$$
where the last map is the monadic composition, is given by $R^{-p(j-1)/2}R^{(-j+1)/2}(x)$ by \Cref{lem: composition Lp and P}. Hence $\lambda_{pj-p+1}\lambda_j^p R^{-p(j-1)/2}R^{(-j+1)/2}(x)=x^{\{p\}^2}$. Now induction along the $k$ implies that $x^{\{p\}^k}$ is the $k$th iteration of the bottom operation on an odd class up to a unit.
\end{proof}

\begin{remark}
    Note that the restriction map cannot detected on the $E^2$-page of the dual bar spectral sequence when $p>2$, since the $p$-fold bracket  $[y,x^{\{p\}}]$ with  $x$ lives in filtration $p$ of the bar filtration, whereas $[y,R^{(-|x|+1)/2}(x)]$ lives in filtration 2; nor can it be detected on the $E^2$-page of the associated dual Grothendieck spectral sequence when $p=2$, since the associated filtration linearizes the unary operations.

    In particular, we see that for $p=2$, the bottom operation $R^{-j+1}(x)=x^{\{2\}}$ on a class $x$ of the homotopy groups of a spectral partition Lie algebra $L$ is not additive. Similarly, for $p>2$ the bottom operation $R^{(-j+1)/2}(x)=x^{\{p\}}$ on a class $x$ in odd degree $j$ of the homotopy groups of a spectral partition Lie algebra $L$ is not additive. 
\end{remark}

Combining \Cref{vanish} and \Cref{restricted}, we have the following theorem. 
  \begin{theorem}\label{compatible}
The binary operation $[- , -]$ constructed above makes the homotopy groups of any spectral partition Lie algebra $A$  a shifted restricted Lie algebra, with Lie bracket $$[-, -]:\pi_j(A)\tens \pi_k(A)\rightarrow\pi_{j+k-1}(A)$$
over $\Fp$. If $p=2$, for all $j$ and $x\in\pi_j(A)$ the restriction $x^{\{2\}}$ is represented by the bottom operation $R^{-j+1}(x)$. The restriction map on a sum of classes $x,y$ in degrees $j\neq k$ is given by $$(x+y)^{\{2\}}=R^{-j+1}(x)+ R^{-k+1}(y)+ [x,y].$$ If $p>2$,  for all odd $j$ and $x\in\pi_j(A)$ the restriction $x^{\{p\}}$ is the bottom operation $R^{(-j+1)/2}(x)$ up to a unit $\lambda_j$. The restriction map on a sum of odd classes $x,y$ in degrees $j\neq k$ are given by $$(x+y)^{\{p\}}=\lambda_j R^{(-j+1)/2}(x)+\lambda_k R^{(-k+1)/2}(y)+\sum_{i=1}^{p-1} \frac{s_i}{i}(x,y),$$ where $s_i$ is the coefficient of $t^{i-1}$ in the formal expression $\mathrm {ad} (tx+y)^{p-1}(x)$.

The bracket is compatible with the unary operations in Theorem \ref{unarymod2}  and \ref{unaryodd} in the  sense that $[x, \alpha (y)]=0$ for $x,y\in\pi_*(A)$ and any unary operations $\alpha$ of weight greater than one that is not an iteration of the restriction map. 
Equivalently, for any $\E^{\mathrm{nu}}_{\infty}$-$\Fp$-algebra $A$,  there is a shifted Lie bracket with restriction
$$[-,-]: \taq^j(A)\tens \taq^k(A)\rightarrow \taq^{j+k+1}(A)$$ satisfying the above conditions.
 \end{theorem}
 \begin{remark}\label{comparebrackets}
  The interaction between the unary operations and the shifted Lie bracket on the homotopy groups of spectral partition Lie algebras differ from that on the homology of spectral Lie algebras. It was shown in \cite{omar,kjaer} that on the mod $p$ homology groups of spectral Lie algebras, the bracket $[x,\alpha(y)]$ always vanishes if $\alpha$ is a unary operation of weight greater than one. In comparison, on the homotopy groups of spectral partition Lie algebras the bracket $[x,\alpha(y)]$ does not necessarily vanish when $\alpha$ is an iteration of the restriction map, which is a non-additive unary operation. 
  
  This phenomenon also shows up in the Lubin-Tate theory of spectral Lie algebras, as was observed in \cite[Proposition 4.3.16]{brantner} that the non-additive unary operation $\theta$ interacts nontrivially with the bracket. For instance, when $p=2$, the bottom non-vanishing operation $\bar{Q}^{|x|}$ on a mod 2 homology class $x$ of a free spectral Lie algebra is identified with the nonzero self-bracket $[x,x]$ by \cite[Lemma 6.4]{omar}. Hence $[y, \bar{Q}^{|x|}(x)]=0$ by the Jacobi identity for all $x,y$. In comparison, the bottom operation $R^{-|x|+1}$ on a  class in the homotopy group of a free spectral partition Lie algebra over $\F$ represents the restriction on $x$, so $[y,R^{-|x|+1}(x)]=[[y,x],x]$ is nonzero. Whereas self-brackets always vanish in shifted restricted Lie algebras over $\F$, cf. \cite[Remark 1.2.9]{fresse}.
 \end{remark}

\section{The algebraic structure of the homotopy groups of spectral partition Lie algebras} \label{section5}
In this section, we deduce all natural operations on the homotopy groups of spectral partition Lie algebras and mod $p$ TAQ (co)homology. First we show that composition product of unary operations of weight $p^k$ on the homotopy groups of spectral partition Lie algebras agrees, up to a shearing, with the Yoneda product on the $E^2$-page of the dual bar spectral sequence. This makes use a general result of Brantner (\cite[Theorem 3.5.1 and 4.3.2]{brantner}). Then we combine the weight $p^k$ unary operations with the shifted restricted Lie structure and use a dimension comparision with Brantner-Mathew (\cite{bm}) to show that we have identified all the algebraic structure on the homotopy groups of spectral partition Lie algebras.

\subsection{Relations among unary operations of weight $p^k$}

A convenient way to encode the structure of unary operations of weight $p^k$ coming from Ext groups with composition given by a shearing of the Yoneda product is via a power ring, as was introduced in \cite{brantner} to encode additive unary operations on the Lubin-Tate theory of spectral Lie algebras. In this sense our power ring is a twisted version of ringoids. 
\begin{definition}
 A \textit{power ring} is a collection  $$P=\{P^j_k[w]\}_{(j,k,w)\in\mathbb{Z}^2\times\mathbb{Z}_{\geq 0})}$$ of abelian groups with elements $\iota_{i}\in P^i_i[0]$ for all $i$, equipped with associative and unital composition maps
 $P^i_j[v]\times P^j_k[w] \rightarrow P^i_k[v+w]$.

 For a given prime $p$, an element in $P^j_k[w]$ represents a unary operation of weight $p^w$. A \textit{module} over the power ring $P$ is a (weighted) $\Fp$-module $M=\bigoplus_{(j,w)\in\mathbb{Z}\times\mathbb{Z}_{\geq 0}} M_j(w)$ with structure maps $P^i_j[v]\times M_j(w) \rightarrow M_i(p^vw)$ that are compatible with the composition maps in $P$.
\end{definition}
\begin{remark}\label{rmk: powerring}
    Note that our convention differs from \cite[Definition 3.17]{brantner} and \cite[Definition 4.5]{bhk} in that we switch to a logarithmic grading convention for the weight grading. Furthermore, we relax the condition of bilinearity of the composition map to accommodate the fact that the restriction map (represented by the bottom operation by \Cref{restricted}) is nonadditive, which we show in \Cref{cor: betabottomadditive} is the only nonadditive unary operation of weight $p$. Note that the homotopy groups of spectral partition Lie algebras are still modules over the power ring of unary operations despite the nonadditivity of the restriction map, since in the Adem relations, the restriction map either doesn't appear in the relation or appear on both sides of the relation (cf. \Cref{rmk: bottom op oddclass}).
    The underlying reason is that when $p>2$, on the associated graded of the bar filtration, i.e. the $E^\infty$-page of the dual bar spectral sequence, the unary operations $R^i$ and $\beta^\epsilon R^i$ are all additive and assemble into a power ring in the sense of Brantner. Similarly, when $p=2$, on the associated graded of the filtration associated with the Grothendieck spectral sequence, the unary operations $R^i$ are all additive. In other words, there is a nontrivial extension of algebraic structure on the $E^\infty$-page of the dual bar spectral sequence. 
\end{remark}

 \begin{definition}\label{powerring}
The collection $\{\pazocal{P}^j_k[w]:=\R_u^!((0,k),(-w,j+w)), w>0\}$, along with $\pazocal{P}^i_i[0]:=\Fp\{\iota_i\}$ for all $i$ and $\pazocal{P}^j_i[0]=\emptyset$ for $i\neq j$, defines a power ring $\pazocal{P}$, with composition product given by the sheared Yoneda product
\begin{equation}\label{diagram: composition in P}
    \begin{tikzcd}
     \pazocal{P}^i_j[v]\times\pazocal{P}^j_k[w]\ar[ddd, dashed]\ar[r,"\cong"] &\R_u^!((0,j),(-v,i+v))\times\R_u^!((0,k),(-w,j+w))\ar[d,"\mathrm{susp}^w\times \id",hook]\\
        &\R_u^!((0,j+w),(-v,i+v+w))\times\R_u^!((0,k),(-w,j+w))\ar[d,"\cong"]\\
         &\R_u^!((-w,j+w),(-v-w,i+v+w))\times\R_u^!((0,k),(-w,j+w))\ar[d,"\mathrm{juxtaposition}"]\\
        \pazocal{P}^i_k[v+w]\ar[r, "\cong"] &\R_u^!((0,k),(-v-w,i+v+w)),
    \end{tikzcd}
\end{equation}
 for $v,w>0$, as well as isomorphisms $\pazocal{P}^j_i[w]\otimes\pazocal{P}^i_i[0]\xrightarrow{\cong}\pazocal{P}^j_i[w]$ and $\pazocal{P}^j_j[0]\otimes\pazocal{P}^j_i[w]\xrightarrow{\cong}\pazocal{P}^j_i[w]$ exhibiting $\iota_i$ as a two-sided unit.
  \end{definition}
  
The first map is an injection on the left factor because  operations are stable under suspension and here $w\geq 0$, cf. \Cref{prop: stable}. The last map is the composition in the ringoid $\R_u^!$, i.e., juxtaposition corresponding to the Yoneda product on Ext groups.

Explicitly, when $p=2$, the $\F$-module $\pazocal{P}^k_j[w]$ consists of operations $R^{i_1,\ldots, i_w}$ such that $j-i_1-\ldots i_w=k$ and  $i_l-1>i_{l+1}+\cdots+i_w -j -(w-l)$ for all $1\leq l\leq w$, subject to the relations in $\R_u^!((0,j),(-w,k+w))$. The composition product sends $R^{(i_1,\ldots, i_v)}\circ R^{(j_1,\ldots, j_w)}$ to the juxtaposition $R^{(i_1,\ldots, i_v,j_1,\ldots, j_w)}.$ The weight $2$ operations are given by the collection $R^{i}\in\pazocal{P}^{j-i}_j[1]$ for any  $i>-j$. 
  \begin{theorem}\label{unarymod2}
   The homotopy groups of a spectral partition Lie algebra over $\F$, or the reduced TAQ cohomology of an $\E_\infty$-$\F$-algebra form a left module over the power ring $\pazocal{P}$ of unary operations. The relations among the weight $2$ operations are given by the Adem relations  $$R^a R^b=\sum_{a+b-c\geq 2c,\,\, c> -j}\binom{b-c-1}{a-2c} R^{a+b-c}R^c$$
in $\pazocal{P}^{j-a-b}_j[2]$ for all $a,b\in\mathbb{Z}$ satisfying $b-j\geq a< 2b$ and $b> -j$. 

A basis for unary operations on a degree $j$ class $x$ is given by the collection of all monomials $R^{i_1}R^{i_2}\cdots R^{i_l}$ such that $i_l>-j$ and $i_m\geq 2i_{m+1}$ for $1\leq m<l$.
  \end{theorem}
  
When $p>2$, the weight $p$ unary operations are given by the collection of elements $\beta^\epsilon R^{i}:=R^{(i,\epsilon)}\in\pazocal{P}^{j-2(p-1)i-\epsilon}_j$  for $\epsilon=0,1$ and any  $2i>-j$.

  \begin{theorem}\label{unaryodd}
    The homotopy groups of a spectral partition Lie algebra $A$ over $\Fp$, or the reduced TAQ cohomology of any $\E_\infty$-$\Fp$-algebra, form a module over the power ring $\pazocal{P}$ of unary operations. The relations among the weight $p$ operations are given by the  Adem relations
    \begin{equation*}
   \beta R^a \beta R^b=\sum_{a+b-c> pc, 2c>-j}(-1)^{a-c+1}\binom{(p-1)(b-c)-1}{a-pc-1} \beta R^{a+b-c} \beta R^c 
\end{equation*}
in $\pazocal{P}_j^{j-2(p-1)a-2(p-1)b-2}[2]$ for all $a,b\in\mathbb{Z}$ satisfying $a \leq pb$, $2b>-j$, and $2a>2(p-1)b-j,$
\begin{align*}
   R^a \beta R^b=&\sum_{a+b-c\geq pc,2c>-j}(-1)^{a-c}\binom{(p-1)(b-c)}{a-pc} \beta P^{a+b-c} R^c\\&-\sum_{a+b-c> pc, 2c>-j}(-1)^{a-c}\binom{(p-1)(b-c)-1}{a-pc-1}  R^{a+b-c}\beta R^c
\end{align*}
 in $\pazocal{P}_j^{j-2(p-1)a-2(p-1)b-1}[2]$ for all $a,b\in\mathbb{Z}$ satisfying  $a \leq pb$, $2b>-j$, and $2a>2(p-1)b+1-j$,
\begin{equation*}
   \beta^\epsilon R^a R^b=\sum_{a+b-c\geq pc, 2c>-j}(-1)^{a-c}\binom{(p-1)(b-c)-1}{a-pc} \beta^\epsilon R^{a+b-c}R^c
\end{equation*}
in $\pazocal{P}^{j-2(p-1)a-2(p-1)b-\epsilon}_j[2]$ for all $a,b\in\mathbb{Z}$ satisfying   $ a < pb$, $2b>-j$, $2a>2(p-1)b-j$, and $\epsilon\in\{0,1\}$.

A basis for unary operations on a degree $j$ class with $j$ odd is given by the collection of all monomials $\beta^{\epsilon_1}R^{i_1}\beta^{\epsilon_2}R^{i_2}\cdots\beta^{\epsilon_l}R^{i_l}$ such that $2i_l>-j$ and $i_m\geq pi_{m+1}+\epsilon_{m+1}$ for $1\leq m<l$. If $j$ is even, a basis is given by the collection $\beta^{\epsilon_1}R^{i_1}\beta^{\epsilon_2}R^{i_2}\cdots\beta^{\epsilon_l}R^{i_l}B^{\epsilon}$ such that $2i_l>-(1+\epsilon)j-\epsilon$  and $i_m\geq pi_{m+1}+\epsilon_{m+1}$ for $1\leq m<l$. Here $B$ stands for the self-bracket on an even degree class.
  \end{theorem}
In other words, \Cref{unarymod2} and \Cref{unaryodd} identify the monadic structure map of the algebraic approximation $\pazocal{L}_p$ of $\spla$ (\Cref{notn: L_p approx}) on elements that come from the power ring $\pazocal{P}$.  We start by showing that the composition product in $\pazocal{P}$ agrees with that of the unary operations on the $E^2$-page of the dual bar spectral sequence.
\begin{lemma}\label{lem: compositionE^2}
The composition product $\cup$ defined by the top arrow of the commutative diagram
     \begin{center}
     \begin{tikzcd}
      \free^{\R_u^!}\circ\free^{\R_u^!}(\Sigma^k\Fp)\ar[r]\ar[d,"\cong"] &\free^{\R_u^!}(\Sigma^k\Fp)\ar[d,"\cong"]\\\pazocal\pi_*(\B(\id, \pazocal{A}_{\R'},\pi_*(\B(\id, \pazocal{A}_{\R'}, \Sigma^{-k} \Fp)^{\vee})^{\vee})\ar[r] &\pi_*(\B(\id, \pazocal{A}_{\R'}, \Sigma^{-k} \Fp)^{\vee})
     \end{tikzcd}
\end{center}
is given by the sheared Yoneda product along the bottom, which agrees with the composition product (\ref{diagram: composition in P}) in $\pazocal{P}$ given in \Cref{powerring}. Namely, given elements $b\in \R'((0,j),(-v,i+v))\cong \mathrm{UnExt}^{-v,i+v}_{\R'}(\Fp,\Sigma^j\Fp)$ and $a\in \R'((0,k),(-w,j+w))\cong \mathrm{UnExt}^{-w,j+w}_{\R'}(\Fp,\Sigma^k\Fp),$ their composition $b\cup a$ is given by $$b\circ a |x_k\mapsto b|a|x_k\in\R_u^!((0,k),(-v-w,i+v+w))\cong\mathrm{UnExt}^{-v-w,i+v+w}_{\R'}(\Fp,\Sigma^k\Fp).$$
\end{lemma}
\begin{proof}
Recall from \Cref{cor: unext p=2} and \Cref{cor: unext p>2} that  $\free^{\R_u^!}(\Sigma^k\Fp)$ is  isomorphic to  $$\mathrm{UnExt}^{*,*}_{\R'}(\Fp,\Sigma^k\Fp)\cong \pi_*(\B(\id, \pazocal{A}_{\R'}, \Sigma^{-k} \Fp)^{\vee}).$$ 
We will make use of a general result of Brantner that follows from \cite[Theorem 3.5.1 and 4.3.2]{brantner}: Suppose that  $\mathbf{T}$ is an augmented additive monad on $\Mod_{\Fp}$ associated to the free (unstable) module functor over an algebra $R$. Then the composition map of the monad $|\B(\id,\mathbf{T}, -)|^{\vee}$, obtained by dualizing the comonad structure map $$|\B(\id,\mathbf{T}, -)|\xleftarrow{\simeq}|\B(\id,\mathbf{T}, |\B(\mathbf{T},\mathbf{T}, -)|)|\rightarrow|\B(\id,\mathbf{T}, |\B(\id,\mathbf{T}, -)|)|$$ of the comonad $|\B(\id,\mathbf{T}, -)|$,  is compatible with the Yoneda product on the (unstable) Ext groups over $R$ up to a shearing of the Ext groups, cf. \cite[proof of Lemma 4.3.3. and p. 135]{brantner}). 

Here we specialize to the situation where $\pazocal{A}_{\R'}$ is an additive monad associated with the free functor that takes the unstable module over the Koszul algebra $\R$. It follows that the top map is a sheared Yoneda product on (unstable) Ext groups. Explicitly,
given $b\in \R'((0,j),(-v,i+v))\cong \mathrm{UnExt}^{-v,i+v}_{\R'}(\Fp,\Sigma^j\Fp)$ and $a\in \R'((0,k),(-w,j+w))\cong \mathrm{UnExt}^{-w,j+w}_{\R'}(\Fp,\Sigma^k\Fp),$ the image along the top map is  $$b\circ a |x_k\mapsto b|a|x_k\in\R_u^!((0,k),(-v-w,i+v+w))\cong\mathrm{UnExt}^{-v-w,i+v+w}_{\R'}(\Fp,\Sigma^k\Fp)$$ via the composite
\begin{center}
    \begin{tikzcd}
      \R_u^!((0,j),(-v,i+v))\times\R_u^!((0,k),(-w,j+w))\ar[d,"\mathrm{susp}^w\times \id",hook]\\
        \R_u^!((0,j+w),(-v,i+v+w))\times\R_u^!((0,k),(-w,j+w))\ar[d,"\cong"]\\
         \R_u^!((-w,j+w),(-v-w,i+v+w))\times\R_u^!((0,k),(-w,j+w))\ar[d]\\
         \R_u^!((0,k),(-v-w,i+v+w)).
    \end{tikzcd}
\end{center}
The first map is an injection on the left factor because  operations are stable under suspension and here $w\geq 0$, cf. \Cref{prop: stable}. The last map is the composition in the ringoid $\R_u^!$, i.e. juxtaposition corresponding to the Yoneda product on unstable Ext groups. This is exactly the composition product in $\pazocal{P}$ defined in (\ref{diagram: composition in P}).
\end{proof}

Next we lift the composition product above to the image of the power ring $\pazocal{P}$ in the algebraic approximation monad $\pazocal{L}_p$.

\begin{lemma}\label{lem: composition Lp and P}
    The monad structure map of the algebraic approximation $\pazocal{L}_p$ of $\spla$ (cf. \Cref{notn: L_p approx}), restricted to unary operations coming from the power ring $\pazocal{P}$, is given by the composition product $\cup$ in \Cref{lem: compositionE^2}.
\end{lemma}
 
 \begin{proof}
 By construction, the set $\pazocal{P}^i_j[w]$ is isomorphic to the image  of $\R_u^!((0,j),(-w,w+i))$ in $\pi_i\spla(\Sigma^j \Fp)[w]$  via the collapse of the weight $p^w$ part of the dual bar spectral sequences on a single generator and the associated Grothendieck spectral sequences (\Cref{E2unary}, \Cref{E2unaryodd} and \Cref{cor: E2unaryeven}).
We need to verify that  compositions of  unary operations  on the homotopy groups of spectral partition Lie algebras is reflected by the composition product of the power ring $\pazocal{P}$.

For ease of notations, we will use $\mathbf{L}$ to denote the monad $\Lie^{\pi}_{\Fp,\E_{\infty}}$ throughout this proof. The unary operations  on the homotopy groups of algebras over $\mathbf{L}$, other than the restriction map when $p=2$ and the self-brackets on even classes when $p>2$, are concentrated in weights $p^n$ for $n\geq 1$ by construction. When $A$ is bounded above, they live in the homotopy groups of the summands $$\mathbf{L}[n](A):=(\partial_{p^n}(\mathrm{Id})\tens \Fp)\tens^{h\Sigma_{p^n}} (A)^{\tens p^n}\xhookrightarrow{\iota_{p^n}} \mathbf{L}(A)$$ by Proposition \ref{bm535}.   The composition $\beta\circ \alpha$ of two unary operations $$\alpha\in \pazocal{P}^j_k[w]\subseteq \pi_j\mathbf{L}[w](\Sigma^k \Fp)\cong \pazocal{L}_p(\Sigma^k \Fp)[w],\,\,\,\,\beta\in\pazocal{P}^i_j[v]\subseteq \pi_i\mathbf{L}[v](\Sigma^{j} \Fp)\cong\pazocal{L}_p(\Sigma^j \Fp)[v],$$ considered as maps $\alpha: \Sigma^j \Fp\rightarrow \pazocal{L}_p(\Sigma^k \Fp)[w]$ and $\beta:\Sigma^i \Fp\rightarrow\pazocal{L}_p(\Sigma^{j} \Fp)[v]$, is given by
  \begin{equation}\label{comp}
      \Sigma^i \Fp\xrightarrow{\beta}\pazocal{L}_p(\Sigma^j \Fp)[v]\xrightarrow{\pazocal{L}_p(\alpha)[v]}\pazocal{L}_p[v]\circ\pazocal{L}_p(\Sigma^k \Fp)[w]\rightarrow\pazocal{L}_p(\Sigma^k \Fp)[v+w].
  \end{equation}
 The last map is induced by the  weight $p^{v+w}$ summand 
 $$(\partial_{p^v}(\mathrm{Id})\tens \Fp)\tens^{h\Sigma_{p^v}}\big((\partial_{p^w}(\mathrm{Id})\tens \Fp)\tens^{h\Sigma_{p^w}} A^{\tens p^w})\big)^{\tens p^v}\rightarrow (\partial_{p^{(v+w)}}(\mathrm{Id})\tens \Fp)\tens^{h\Sigma_{p^{v+w}}} A^{\tens p^{v+w}}$$of the structure map of the monad $\mathbf{L}\circ\mathbf{L}\rightarrow\mathbf{L}$ on any bounded above object $A$.
   
   Let $a\in \R^!((0,k),(-w,j+w)),b\in \R^!((0,j),(-v,i+v))$ be the unique preimages under the collapse of the weight $p^w$ and weight $p^v$ parts of the dual bar spectral sequences converging respectively to $\pi_*\mathbf{L}(\Sigma^k \Fp)\cong \pazocal{L}_p(\Sigma^{k}\Fp)$ and $\pi_*\mathbf{L}(\Sigma^j \Fp)\cong \pazocal{L}_p(\Sigma^{j}\Fp)$, as well the collapses of the associated dual Grothendieck spectral sequence in both weights.
   
    Since $\pazocal{L}_p(\Sigma^{j}\Fp)$ is bounded above and of finite type, we can run the dual bar spectral sequence and the associated dual Grothendieck spectral sequence for the $\Fp$-module $A=\mathbf{L}(\Sigma^k \Fp)$ converging to $\pi_*\mathbf{L}\circ\mathbf{L}(\Sigma^k \Fp)\cong \pazocal{L}_p\circ \pazocal{L}_p(\Sigma^{j}\Fp)$, with summand $\pi_*\mathbf{L}(\Sigma^j\Fp)\cong\pazocal{L}_p(\Fp\{\alpha(x_k)\})$ picking out the free spectral partition Lie algebra on the generator represented by $\alpha (x_k)\in \pazocal{L}_p(\Fp\{x_k\})$ sitting in weight $p^w$. The weight $p^{v+w}$ part of both spectral sequences for this summand collapse on the $E^2$-page. Therefore, the monadic structure map $\pazocal{L}_p\circ\pazocal{L}_p(\Sigma^k \Fp)\rightarrow\pazocal{L}_p(\Sigma^k \Fp)$  induces a map of the $E^2$-pages of the associated dual Grothendieck spectral sequences.
In particular, the map (\ref{comp}) lives inside the summand $\free^{\R_u^!}\circ \free^{\R_u^!}(\Sigma^k\Fp)\rightarrow\free^{\R_u^!}(\Sigma^k\Fp)$, which  is
   given explicitly by $b|x_j\mapsto b\circ (a|x_k)\mapsto b| a|x_k$ by \Cref{lem: compositionE^2}.  Here we use $|$ to denote juxtaposition in $\R_u^!$ and $x_k$ the generator for $\Sigma^k\Fp$.
   Passing to the $E^\infty$-pages, we obtain a commutative diagram 
 \begin{equation}\label{diag: comp}
      \begin{tikzcd}
\pazocal{P}^i_j[v]\times \pazocal{P}^j_k[w]\ar[r,"*"]\ar[d,"\cong"]& \pazocal{P}^i_k[v+w]\ar[d,"\cong"]\\
      \R_u^!((0,j),(-v,i+v))\times \R_u^!((0,k),(-w,j+w))\ar[r,"\cup"]\ar[d,hook]& \R_u^!((0,k),(-v-w,i+v+w))\ar[d,hook]\\\pazocal{L}_p(\Sigma^{j} \Fp)[v]\times\pazocal{L}_p(\Sigma^{k} \Fp)[w]\ar[r,"\circ"]&\pazocal{L}_p(\Sigma^{k} \Fp)[v+w]
     \end{tikzcd}
 \end{equation}
as desired. The top square commutes by \Cref{lem: compositionE^2} and the bottom square by the degeneration of the weight $p^{v+w}$ part of the dual bar spectral sequences and the dual Grothendieck spectral sequences on the $E^2$-pages.
\end{proof}
\begin{lemma} \label{cor: betabottomadditive}
\begin{enumerate}
    \item If a unary operation in $\pazocal{P}^j_k[1]$ is in the image of the suspension map $\mathrm{susp}^1$ (\Cref{prop: stable}), then it is an additive operation. Therefore $R^i\in \pazocal{P}^j_*[1]$ is additive if $i-1>-j$ when $p=2$ and $\beta^\epsilon R^i\in \pazocal{P}_j^*[1]$ is additive if $i-1>-(j+1)/2$. 
    \item For $p>2$ and $j$ odd, the operation $\beta R^{(-j+1)/2}\in\pazocal{P}^{j-(p-1)(j+1)-1}_j[1]$ is additive.
    \end{enumerate}
\end{lemma}
 \begin{proof}
     (1). The first statement follows from taking the Spanier Whitehead dual of the $\Fp$-version of \cite[Lemma 4.2.21]{brantner}. Combining with \Cref{prop: stable}, we see that if $p>2$ and $x$ is a class in even degree or if $p=2$, then any weight $p$ operation defined on $x$ is in the image of $\mathrm{susp}^1$. Whereas if $p>2$ and $x$ has odd degree, then any weight $p$ operation that is not the bottom operation $\beta^\epsilon R^{(-j+1)/2}$   is in the image of $\mathrm{susp}^1$.

     (2). Let $x,y,z$ be the degree $j$ generators in $\pi_*(\spla(\Sigma^j\Fp\oplus\Sigma^j\Fp\oplus\Sigma^j\Fp))$. Suppose that $\beta R^{(-j+1)/2}(x+y)=\beta R^{(-j+1)/2}(x)+\beta R^{(-j+1)/2}(y)+B(x,y)$. Then $B(x,y)$ has to be a linear combination of Lie brackets of even degree $p(j-1)$  and not a unary operation acting on $x$ or $y$ for degree reasons (in particular it does not live in the degree of the bottom operation on degree $j$). Now consider the element $[z,\beta R^{(-j+1)/2}(x+y)]$, which vanishes by \Cref{vanish}. Hence $[z,B(x,y)]=0$ in the free shifted Lie algebra on letters $x,y,z$, which implies that $B(x,y)$ has to consists of self-brackets. But the degree of any self-bracket is odd, so $B(x,y)=0$.
\end{proof}

Now we assemble the three lemmas above to finish the proof of \Cref{unarymod2} and \Cref{unaryodd}.
\begin{proof}[Proof of Theorem 6.4 and 6.5]
It remains to verify the Adem relations and the bases. By \Cref{lem: compositionE^2}, the composition product $\pazocal{P}^{j-a-b}_{j-a}[1]\times \pazocal{P}^{j-a}_j[1]\rightarrow\pazocal{P}^{j-a-b}_j[2]$ sends $(R^{b},R^{a})$ to $R^{(b,a)}=(Q^{b-1})^*(Q^{a-1})^*$ for $p=2$, and $(\beta^{\epsilon_1}R^{b'-\epsilon_1},\beta^{\epsilon_2}R^{a'-\epsilon_2})$ to $$R^{(b'-\epsilon_1,a'-\epsilon_2, \epsilon_1,\epsilon_2)}=(\beta^{1-\epsilon_1}Q^{b'-1})^*(\beta^{1-\epsilon_2}Q^{a'-1})^* $$ for $p>2$ where $2(p-1)a'=a, 2(p-1)b'=b$. The two classes in the composition are defined whenever $a>-j+1$ and $b>a-j$ for $p=2$, and $2a'>-j$ and $2b'>a+\epsilon_2-j$ for $p>2$. Hence all the Adem relations hold.

When $p=2$, a basis for additive unary operations on a degree $j$ class is given by all monomials $$(Q^{i_1})^*(Q^{i_2})^*\cdots(Q^{i_s})^*\in\R_u^!((0,j),(s,m))$$ such that $i_s>-j$ and $i_l>2i_{l+1}$ for all $1\leq l<s$ by \Cref{E2unary}. Any such monomial is the image of the (well-defined) iterated composition $R^{i_1+1}R^{i_2+1}\cdots R^{i_s+1}$ in $\pazocal{P}^{j-m-s}_j[s]$.  Hence every additive unary operation $R^{(i_1,\ldots, i_s)}$ can be written as a linear combination of compositions $R^{j_1}\ldots R^{j_s}$ of operations in $\pazocal{P}^*_*[1]$. The case $p>2$ follows verbatim by using the degeneration in \Cref{E2unaryodd},  \Cref{lem: hfpsswtp} and \Cref{cor: E2unaryeven}.
 \end{proof}

 \begin{corollary}\label{algebra}
 When $j$  gets arbitrarily large, the algebra of unary operations on a degree $j$ homotopy class of a spectral partition Lie algebra, or equivalently a degree $-j$ class in mod $p$ TAQ cohomology, is the Koszul dual algebra of the mod $p$ Dyer-Lashof algebra.
 \end{corollary}
\begin{proof}
    By \Cref{prop: stable}, the suspension map $$\mathrm{susp}^t:\pazocal{P}^i_j[v]\cong\pazocal{R}^!_u((0,j),(-v,i+v))\rightarrow\pazocal{P}^{i+t}_{j+t}[v]\cong\pazocal{R}^!_u((0,j+t),(-v,i+v+t))$$ is injective for all $t>0$. Thus taking the colimit along $\mathrm{susp}^1$ for fixed $i-j$ and $v$, we obtain the universal unary operations of weight $v$ and degree $i-j$ as elements in $\tilde{\pazocal{P}}^{i-j}[v]$ of the power ring $\pazocal P$. Furthermore, the composition product of $\pazocal{P}$ (cf. \Cref{diag: comp}) is compatible with taking the colimit along $\mathrm{susp}^1$ in each variable, where the suspension map on the along bottom is induced by $\Sigma\pazocal{L}_p(\Sigma^j\Fp)\rightarrow\pazocal{L}_p(\Sigma^{j+1}\Fp)$. To identify the algebra generated by $\tilde{\pazocal{P}}^{i-j}[v]$ with the composition above with the Koszul dual of the mod $p$ Dyer-Lashof algebra, we use the identification of  $$\mathrm{UnExt}^{*,*}_{\R'}(\Fp,\Sigma^{j} \Fp)=\pi_*(\B(\id, \pazocal{A}_{\R'}, \Sigma^{-j} \Fp)^{\vee}),$$ as the free $\R^!_u$-module on a class in degree $(0,j)$ (\Cref{cor: unext p=2} and \Cref{cor: unext p>2}). In the colimit, the instability conditions on the class $\underset{j\rightarrow\infty}{\mathrm{colim}}\Sigma^{-j}\Fp$ as the trivial unstable module over $\R'$ becomes vacuous, so the unstable Ext group is isomorphic to the usual Ext group over the Dyer-Lashof algebra $\R$.
\end{proof}
 For $p>2$  one needs be careful about the precise duality. The Dyer-Lashof operation $Q^i$ is sent to $\beta R^i=\beta P^i$ in cohomological degree $2(p-1)i+1$ and $\beta Q^i$ to $R^i=P^i$ in cohomological degree $2(p-1)i$ since the Bockstein homomorphism increases cohomological degree by one.

\subsection{Generation}
Finally we put all the structures together to obtain the optimal target category for the homotopy group of spectral partition Lie algebras, or equivalently the reduced mod $p$ TAQ cohomology, which records the entire algebraic structure.

\begin{definition}\label{sLieP}
 A $\pazocal{P}$-$\Lie^{s,\rho}$-\textit{algebra} $L$ is a module over the power ring $\pazocal{P}$ (\Cref{powerring}), together with a shifted Lie bracket and a restriction $(-)^{\{p\}}$, that satisfies the following conditions:
 
 (1) If $p=2$, for all $j$ and $x\in\pi_j(A)$ the restriction $x^{\{2\}}$ is given by the bottom operation $R^{-j+1}(x)$. The restriction map on a sum of classes $x$ and  $y$ is given by $$(x+y)^{\{2\}}=x^{\{2\}}+ y^{\{2\}}+ [x,y].$$
 If $p>2$,  for all odd $j$ and $x\in\pi_j(A)$ the restriction $x^{\{p\}}$ is up to a unit $\lambda_j$ the bottom operation $R^{(-j+1)/2}(x)$. The restriction map on the sum of classes $x,y$ in degrees $j\neq k$ is given by $$(x+y)^{\{p\}}=\lambda_j R^{(-j+1)/2}(x)+\lambda_k R^{(-k+1)/2}(y)+\sum_{i=1}^{p-1} s_i(x,y),$$ where $s_i$ is the coefficient of $t^{i-1}$ in the formal expression $\mathrm {ad} (tx+y)^{p-1}(x)$;
 
 (2) The bracket $[y, \alpha (x)]$ vanishes for any $x,y\in L$ and  $\alpha$ a unary operation of weight greater than one, unless $\alpha$ is an iteration of the restriction map.
 
 Denote by $\Lie^{s,\rho}_{\pazocal{P}}$ the category of $\pazocal{P}$-$\Lie^{s,\rho}$-algebras.
\end{definition}

 The free $\pazocal{P}$-$\Lie^{s,\rho}$-algebra functor $\free^{\Lie^{s,\rho}_{\pazocal{P}}}$ on a $\Fp$-module $M$ can be computed as follows:  first we take the free shifted restricted Lie algebra over $\Fp$, then take the free $\pazocal{P}$-module on $\free^{\Lie^{s,\rho}_{\Fp}}(M)$. Then we identify the restriction  $x\mapsto x^{\{p\}}$ with the bottom operation $R^{-|x|+1}(x)$ when $p=2$ and the bottom operation $R^{(-|x|+1)/2}(x)$ up to a unit $\lambda_{|x|}$ for any odd degree $x$ when $p>2$. Finally we extend the shifted  Lie bracket and the restriction map to the quotient of $\free^{\pazocal{P}}\free^{\Lie^{s,\rho}_{\Fp}}(M)$ by the above identification,  subject to the conditions in Definition \ref{generation}. 

 By \Cref{unarymod2}, \Cref{unaryodd} and \Cref{compatible}, the homotopy group every spectral partition Lie algebra, or the reduced TAQ cohomology of any $\E_\infty$-$\Fp$-algebra, has the structure of an $\Lie^{s,\rho}_{\pazocal{P}}$-algebra. 
\begin{theorem}\label{generation}
The canonical map of $\pazocal{P}$-$\Lie^{s,\rho}$-algebras 
$$\alpha: \free^{\Lie^{s,\rho}_{\pazocal{P}}}(\pi_*(A))\rightarrow \pi_*(\Lie^{\pi}_{\Fp,\E_{\infty}}(A))$$ is an isomorphism when $A$ is any direct sum of shifts of $\Fp$'s. 
\end{theorem}
\begin{proof}
    For $A$ an infinite direct sum, write $A$ as the filtered colimit of $A_{[-n,n]}$, where $A_{[-n,n]}$ is the truncation of $A$ above degree $-n$ and below degree $n$. Now we use the fact that free functors are left adjoints and preserve colimits, while $\spla$ preserves filtered colimits by \Cref{def: spla}.
    Hence it suffices to show that $\alpha$ is an isomorphism in the cases where $A=\Sigma^{-j_1}\Fp\oplus\cdots\oplus \Sigma^{-j_k}\Fp$ is a finite sum of shifts of $\Fp$'s, and this will be achieved by finding a bijection between a basis for $\free^{\Lie^{s,\rho}_{\pazocal{P}}}(\pi_*(A))$ and the basis of $\pi_*(\Lie^{\pi}_{\Fp,\E_{\infty}}(A))$ given by Brantner-Mathew in \Cref{dimension}. 

    For $p=2$, a basis for the $\F$-vector space $\free^{\Lie^{s,\rho}_{\pazocal{P}}}(\pi_*(A))$ is given by the monomials 
$R^{i_1}R^{i_2}\cdots R^{i_l}u^{\{2\}^r}$ such that $i_l>-2^rj+2^r$ and $i_m\geq 2i_{m+1}$ for $1\leq m<l$, where and $u$ ranges over Lyndon words on letters $x_1, \ldots , x_k$ with $\mathrm{deg}(x_m)=j_m$. 
    Recall from \Cref{dimension} that the $\F$-vector space $\pazocal{L}_2(A)$ has a basis indexed by
sequences $(l_1,\ldots, l_k, u)$. 
The integers $l_1,\ldots, l_k$ satisfy $l_j < 2l_{j+1}$ for all $1 \leq j < k$ and $l_k \leq \mathrm{deg}(u)$.
The total degree of $(l_1,\ldots, l_k,u)$ is $ \mathrm{deg}(u) + (l_1-1) + \ldots + (l_k-1)$. The bijection between this basis and a basis of $\free^{\Lie^{s,\rho}_{\pazocal{P}}}(\pi_*(A))$ is given by sending a monomial 
$R^{i_1}\cdots R^{i_k}(u^{\{2\}^r})$ to the sequence $(-i_1+1, \ldots, -i_k+1, 2^{r-1}\mathrm{deg}(u)-2^{r-1}+1,\ldots,  2\mathrm{deg}(u)-1,\mathrm{deg}(u) ,u)$. 

For $p>2$, a basis for the $\Fp$-module $\free^{\Lie^{s,\rho}_{\pazocal{P}}}(\pi_*(A))$ is given by $\beta^{\epsilon_1}R^{i_1}\beta^{\epsilon_2}R^{i_2}\cdots\beta^{\epsilon_k}R^{i_k} (w)$ satisfying $2i_k> -\mathrm{deg}(w)$ and $i_l> p i_{l+1}-\epsilon$ for $1\leq l<k$ by \Cref{unaryodd}. Here $w$ is either a Lyndon word $u$ on letters $x_1,\ldots x_l$ with $\mathrm{deg}(x_i)=j_i$, or a self bracket on some Lyndon word $u$ with $\mathrm{deg}(u)$ even, in which case $\mathrm{deg}(w)=2\mathrm{deg}(u)-1$.   On the other hand, by \Cref{dimension}, the $\Fp$-vector space $\pazocal{L}_p(A)$ has a basis indexed by
sequences $(l_1,\ldots, l_k, e,u)$. Here $u$ is also a Lyndon word on letters letters $x_1,\ldots x_k$. We have $e \in \{0, \iota\}$, where
$\iota= 1$ if $p$ is odd and $\mathrm{deg}(u)$ is even; otherwise, $\iota=0$.
The integers $l_1,\ldots, l_k$ satisfy:
\begin{enumerate}
    \item Each $l_j$ is congruent to 0 or 1 modulo $2(p-1)$.
    \item For all $1 \leq j < k$, we have $l_j < pl_{j+1}$.
    \item We have $l_k \leq (p- 1)(1+ e) \mathrm{deg}(u)- \iota$. Note that equality is never achieved for parity reasons when $\mathrm{deg}(u)$ is odd and divisibility reasons (by $(p-1)$) when $\mathrm{deg}(u)$ is even.
\end{enumerate}
The degree of $(l_1,\ldots, l_k, e,u)$ is $((1 + e) \mathrm{deg}(u)- e) + l_1 + \ldots + l_k - k$.

The bijection is given as follows: given a monomial $\beta^{\epsilon_1}R^{i_1}\beta^{\epsilon_2}R^{i_2}\cdots\beta^{\epsilon_k}R^{i_k} (u)$, if $w=u$ is a Lyndon word, then we send it to the sequence $(-2(p-1)i_1+1-\epsilon_1,\ldots, -2(p-1)i_k+1-\epsilon_k,0,u)$; if $w=[u,u]$ with $\mathrm{deg}(u)$ even, then we send it to the sequence $(-2(p-1)i_1+1-\epsilon_1,\ldots, -2(p-1)i_k+1-\epsilon_k,1,u)$ in the same total degree. 
\end{proof}

Hence we have identified the algebraic approximation $\pazocal{L}_p$ of $\spla$ (cf. \Cref{notn: L_p approx}) as the monad associated to the free functor $\free^{\Lie^{s,\rho}_{\pazocal{P}}}$, as well as the target category for the homotopy groups of spectral partition Lie algebras and the  reduced  TAQ cohomology of any $\E_\infty$-$\Fp$-algebra.

\section{Operations on mod $p$ TAQ cohomology}
In the last section, we record a computation of all natural operations on  the mod $p$ TAQ cohomology $\mathrm{TAQ}_{\mathbb{S}}^*(-;\Fp)$ of $\E_\infty$-$\mathbb{S}$-algebras (cf. \Cref{def: FpTAQ}), as well as determining their relations. The results in this section are largely inspired by conversations with Tyler Lawson.

Recall from \cite[1.8]{lawson} that the mod $p$ TAQ homology of an $\E_\infty$-$\mathbb{S}$-algebra $R$ can be computed by
$$\mathrm{TAQ}^\mathbb{S}_*(R; \Fp)\simeq \pi_*(|\B(\Fp\tens \id, \E_\infty,R)|).$$ For any $\E^{\mathrm{nu}}_\infty$-$\mathbb{S}$-algebra $A$, the reduced mod $p$ TAQ cohomology of an $\E_\infty$-$\mathbb{S}$-algebra $\mathbb{S}\oplus A$ is the same as the mod $p$ TAQ cohomology groups  
$$\taq^{n}_\mathbb{S}(A;\Fp):= [\Sigma^{-n}|\B(\id, \E^{\mathrm{nu}}_\infty,A)|,  \Fp]_{\mathrm{Sp}}.$$ When $A$ is of finite type, i.e. the underlying $\Fp$-module of $A$ has finite-dimensional homotopy groups in each degree,  $\taq^{*}_\mathbb{S}(A;\Fp)\simeq \pi_{-*}(|\B(\Fp\tens\id, \E^{\mathrm{nu}}_\infty,A)|^\vee).$ Since all operations vanish on the unit of the mod $p$ TAQ cohomology except for multiplication by units, we will again compute operations on $\mathrm{TAQ}_\mathbb{S}^*(-;\Fp)$ by throwing away the base point and computing the dual bar spectral sequence on representing objects.

\begin{corollary}\label{slinear}
 Unary operations on a degree $j$ cohomology class in the reduced mod $p$ $\mathbb{S}$-linear TAQ cohomology $\taq_\mathbb{S}^*(A;\Fp)$ of $\E^{\mathrm{nu}}_\infty$-$\mathbb{S}$-algebras $A$ are parametrized by the free ${\pazocal{P}}$-$\Lie^{s,\rho}$-algebra $\free^{\Lie^{s,\rho}_{\pazocal{P}}}(\Sigma^{-j}\pazocal{A})$, where $\pazocal{A}$ is the mod p Steenrod algebra with homological grading. 
 
In general, for any tuple $(i_1,\ldots i_k)$, the $k$-ary cohomology operations $$\prod^k_{i=1}\mathrm{TAQ}_{\mathbb{S}}^{i_l}(- ;\Fp)\rightarrow\mathrm{TAQ}_{\mathbb{S}}^{m}(- ;\Fp)$$ away from the unit are parametrized by the homological degree $-m$ part of  $$\free^{\Lie^{s,\rho}_{\pazocal{P}}}(\Sigma^{-i_1}\pazocal{A}\oplus\cdots\oplus \Sigma^{-i_k}\pazocal{A}).$$
\end{corollary}

\begin{proof}
The representing objects for the  mod $p$ TAQ cohomology functor $\mathrm{TAQ}_{\mathbb{S}}^*(-;\Fp)$ are the trivial square-zero extensions $\mathbb{S}\oplus\Sigma^{i_1}\Fp\oplus\cdots \Sigma^{i_n}\Fp$ \cite[section 1.8]{lawson}. To compute the $k$-ary operations, we plug in the trivial algebras $\mathbb{S}\oplus\Sigma^{i_1}\Fp\oplus\cdots \Sigma^{i_n}\Fp$. It follows from the definitions that we have a base change formula
$$\mathrm{TAQ}_\mathbb{S}(R;\Fp)\tens\Fp\simeq \mathrm{TAQ}_{\Fp}(R\tens_{\mathbb{S}} \Fp).$$ It then follows from
 \Cref{generation} that $k$-ary operations on $\taq_{\Fp}^*(-)$ are parametrized by $$\pi_*\spla(\Sigma^{i_1}\Fp\tens\Fp\oplus\cdots\oplus \Sigma^{i_n}\Fp\tens \Fp) \cong \free^{\Lie^{s,\rho}_{\pazocal{P}}}(\Sigma^{-i_1}\pazocal{A}\oplus\cdots\Sigma^{-i_n}\pazocal{A}).$$ In particular, unary operations on a degree $j$ cohomology class in the reduced mod $p$ TAQ cohomology are parametrized by
 \begin{align*}
     \taq_{\Fp}^*(\Sigma^j \Fp\tens \Fp)\cong \pi_*(\spla(\Sigma^{-j} \Fp\tens \Fp)) \cong \free^{\Lie^{s,\rho}_{\pazocal{P}}}(\Sigma^{-j}\pazocal{A}).
 \end{align*}
\end{proof}

Finally we deduce the relations between the Steenrod operations and the unary $\Fp$-linear TAQ cohomology operations.

\begin{proposition}\label{slinearrelations}
The Steenrod operations commute with the bracket via the usual Cartan formula
$$Sq^a[x,y]=\sum_i[Sq^i(x),Sq^{a-i}(y)],\,\,\mathrm{for\ }p=2,$$
$$P^a[x,y]=\sum_i[P^i(x),P^{a-i}(y)],\,\,\,\,\beta P^a[x,y]=\sum_i([\beta P^i(x),P^{a-i}(y)]+[ P^i(x),\beta P^{a-i}(y)])$$ for $p>2$. They commute with unary $\Fp$-linear TAQ cohomology operations $R^i$ via the Nishida relations on mod $p$ cohomology, i.e.,
 $$Sq^a R^{-|x|+1} (x) = \sum \binom{j-c}{a-2c} R^{a+j+1-c} Sq^c  (x) + \sum_{l<k, l+k=a} [Sq^l (x) , Sq^k( x) ],$$
 $$ Sq^a R^b (x) = \sum \binom{b-1-c}{a-2c} R^{a+b-c} Sq^c  (x),\,\,\,\, b>-|x|+1$$
 for $p=2$. For $p>2$ we have  
 \begin{align*}
     P^n \beta R^j(x)=&(-1)^{n-i}\sum_i\binom{(j-i)(p-1)}{n-pi}\beta R^{n+j-i}P^i(x)\\
     &+(-1)^{n-i}\sum_i\binom{(j-i)(p-1)-1}{n-pi-1} R^{n+j-i}\beta P^i(x),
 \end{align*}
$$P^n R^j(x)=(-1)^{n-i}\sum_i\binom{(j-i)(p-1)-1}{n-pi}R^{n+j-i}P^i(x)$$
 for all $2j>-|x|+1$, as well as
\begin{align*}
    P^n R^{j}(x)=&(-1)^{n-i}\sum_i\binom{(j-i)(p-1)-1}{n-pi}R^{n+j-i}P^i(x)\\
    &+\frac{1}{\lambda_{|x|}}\sum_{I, \sigma\in\Sigma_p, \sigma(1)=1}[[\cdots[[P^{i_{\sigma(1)}}(x),P^{i_{\sigma(2)}}(x)],P^{i_{\sigma(3)}}(x)]\cdots], P^{i_{\sigma(p)}}(x)]
\end{align*}
when the degree of $x$ is odd and $2j=-|x|+1$, where the bracket term sums over all nondecreasing sequences $I=(0\leq i_1\leq i_2\leq\ldots\leq i_p)$ with $i_1+i_2+\cdots+i_p=n$ 
for $p>2$.
\end{proposition}
Recall that $\lambda_{|x|}$ is the unit by which bottom operation on an odd degree class $x$ differs from the restriction $x^{\{p\}}$ on $x$, cf. \Cref{restricted}. 
\begin{remark}
 Note that the commuting relations between the Steenrod operations and the TAQ cohomology operations $R^i$ coincide with the Adem relations for Steenrod algebras, thereby reinforcing the heuristics that the operations $R^i$ are extended Steenrod operations.
\end{remark}

\begin{proof}[Proof of Proposition 6.2]
We will derive the commuting relations from the $E^2$-page of the dual bar spectral sequence on a single cohomology class in degree $j$.
Since free functors preserve colimits, it follows from the limiting case of Proposition \ref{Einfty} and \ref{Einftyodd} that the dual bar spectral sequence converging to
$$\taq^{-*}_{\Fp}(\taq_{\Fp}(\Sigma^j\Fp\tens\Fp)\tens\Fp)\cong \pi_*(\spla(\spla(\Sigma^{-j}\Fp\tens\Fp))\tens\Fp)$$
collapses on the $E^2$-page with
$$E^\infty \cong E^2 \cong \free^{\Lie^{s,\rho}_{\pazocal{P}}}(\free^{\Lie^{s,\rho}_{\pazocal{P}}}(\Sigma^{-j}\pazocal{A})\tens\Fp).$$  Note that the Steenrod operations are the linear dual of the Dyer-Lashof operations \cite{Einfinity}, which satisfy the Cartan formula with respect to the polynomial product (\Cref{def: polyR}). Since the generator of the shifted Lie bracket on the $E^2$-page is represented by the linear dual of the polynomial product on the $E^1$-page, it commutes with the Steenrod operations via the Cartan formula. On the other hand, the operations $R^i= (Q^{i-1})^*$ come from the linear dual of the Dyer-Lashof operations $Q^{i-1}$, so the Steenrod operations commute with $R^i$ via the Nishida relations on cohomology. When $p=2$ the relations are worked out explicitly, for example, by Miller in
 \cite{nishida}. The Nishida relations for applying a Steenrod operation to the bottom operation on $x$ involves an extra bracket term because the bottom operation is the restriction on $x$.
 
 For $p>2$,  the Nishida relations on cohomology can be read off from Theorem 3 and its corollary in Nishida's original paper \cite{nishidaodd}:
$$P^n \beta R^j=(-1)^{n-i}\sum_i\binom{(j-i)(p-1)}{n-pi}\beta R^{n+j-i}P^i+(-1)^{n-i}\sum_i\binom{(j-i)(p-1)-1}{n-pi-1} R^{n+j-i}\beta P^i,$$
$$P^n R^j=(-1)^{n-i}\sum_i\binom{(j-i)(p-1)-1}{n-pi}R^{n+j-i}P^i.$$
Analogous to the case $p=2$, when $x$ is a class in odd degree, the Nishida relations for the Steenrod action on the bottom class $P^n R^{(-|x|+1)/2}(x)$ involve Lie bracket terms in addition to monomials of unary operations, since $\lambda_{|x|}R^{(-|x|+1)/2}(x)$ is the restriction on $x$. 

In order to determine the extra bracket terms, we need an explicit expression for the restriction map on an odd class. In the setting of unshifted graded $\Fp$-modules, this is worked out by Fresse in \cite[Remark 1.2.8]{fresse}. Note that there is an embedding of the Lie operad into the associative operad Assoc. Furthermore, there is an identity
\begin{equation}\label{ass}
    \sum_{\sigma\in\Sigma_p}X_{\sigma(1)}\cdots X_{\sigma(p)}=\sum_{\sigma\in\Sigma_p, \sigma(1)=1}\langle\langle\cdots\langle\langle X_{\sigma(1)},X_{\sigma(2)}\rangle,X_{\sigma(3)}
\rangle\cdots\rangle, X_{\sigma(p)}\rangle
\end{equation}
in the associative operad, where $\langle x, y\rangle=xy-yx$ is the commutator. For $x\in V$ in even degree, the $p$th power on $x$ is given by $$\sum_{\sigma\in\Sigma_p}X_{\sigma(1)}\cdots X_{\sigma(p)}\tens x^{\tens p}\in (\mathrm{Assoc}(p)\tens V^{\tens p})^{\Sigma_p}\cong(\mathrm{Assoc}(p)\tens V^{\tens p})_{\Sigma_p}.$$ Using the identity (\ref{ass}), we can pull back the $p$th power on $x$ along the embedding $$(\Lie(p)\tens V^{\tens p})^{\Sigma_p}\hookrightarrow(\mathrm{Assoc}(p)\tens V^{\tens p})^{\Sigma_p}.$$  The resulting element is the restriction on $x$ in the free restricted Lie algebra on $V$, i.e.,
\begin{equation}\label{formulaforrestriction}
    x^{\{p\}}=\Big(\sum_{\sigma\in\Sigma_p, \sigma(1)=1}[[\cdots[[X_{\sigma(1)},X_{\sigma(2)}],X_{\sigma(3)}]\cdots], X_{\sigma(p)}]\Big)\tens x^{\tens p}\in (\Lie(p)\tens V^{\tens p})^{\Sigma_p}.
\end{equation}

Since we are working with shifted graded $\Fp$-modules, the commutator in the shifted graded associative algebra is $\langle x, y\rangle=xy-(-1)^{(|y|-1)(|x|-1)}yx$. If $x,y$ are both in odd degrees, then $\langle x, y\rangle=xy-yx$. Hence the identity (\ref{ass}) pulls back to the restriction map (\ref{formulaforrestriction}) on an odd class $x$ in the free shifted graded restricted Lie algebra over $\Fp$. Now we apply the Steenrod operation $P^n$ to the $p$th power on $x$ and use the Cartan formula. Note that the Steenrod operations $P^a$ raises degree by an even number, so none of the signs are altered. Pulling back to the free shifted restricted Lie algebra, we deduce that the bracket terms in the Nishida relation for $P^n R^{(-|x|+1)/2}(x)$  consists of terms $$\frac{1}{\lambda_{|x|}}\sum_{\sigma\in\Sigma_p, \sigma(1)=1}[[\cdots[[P^{i_{\sigma(1)}}(x),P^{i_{\sigma(2)}}(x)],P^{i_{\sigma(3)}}(x)]\cdots], P^{i_{\sigma(p)}}(x)]$$
for all nondecreasing sequences $0\leq i_1\leq i_2\leq\ldots\leq i_p$ with $i_1+i_2+\cdots+i_p=n$.
\end{proof}


\begin{thebibliography}{50}
\bibitem[AC20]{omar}{Antolín Camarena, O. (2020). The mod 2 homology of free spectral Lie algebras. \textit{Transactions of the American Mathematical Society, 373}(9), 6301-6319.}



\bibitem[AB21]{young}{Arone, G. Z., \& Brantner, D. L. B. (2021). The action of Young subgroups on the partition complex. \textit{Publications mathématiques de l'IHÉS, 133}(1), 47-156.}

\bibitem[ADL13]{ADL}{Arone, G. Z., Dwyer, W. G., \& Lesh, K. (2013). Bredon homology of partition complexes. \textit{arXiv preprint arXiv:1306.0056}.}

\bibitem[AM99]{am}{Arone, G. \&  Mahowald, M. (1999). The Goodwillie tower of the identity functor and the unstable periodic homotopy of spheres. \textit{Inventiones mathematicae, 135}(3), 743-788.}

\bibitem[Aro06]{arone}{Arone, G. (2006). A note on the homology of {$\Sigma_n$}, the Schwartz genus, and solving polynomial equations. \textit{Contemporary Mathematics}, 399, 1.}

\bibitem[Bas99]{taq}{Basterra, M. (1999). André–Quillen cohomology of commutative S-algebras. \textit{Journal of Pure and Applied Algebra, 144}(2), 111-143.}

\bibitem[Beh12]{behrens}{Behrens, M. (2012). \textit{The Goodwillie tower and the EHP sequence} (Vol. 218, No. 1026). American Mathematical Soc..}

\bibitem[BR20]{behrensrezk}{Behrens, M., \& Rezk, C. (2020). The {B}ousfield-{K}uhn functor and topological {A}ndr\'{e}-{Q}uillen cohomology. \textit{Inventiones mathematicae, 220}(3), 949-1022.}

\bibitem[Boa99]{boardman}{Boardman, J. M. (1999). Conditionally convergent spectral sequences. \textit{Contemporary Mathematics, 239}, 49-84.}

\bibitem[Bou68]{bousfield}{Bousfield, A. K. (1968). \textit{Operations on derived functors of non-additive functors.} Brandeis Univ.. Available at \url{https://math.mit.edu/~hrm/manuscripts/bousfield-operations.pdf}.}

\bibitem[BC70]{unstable}{Bousfield, A. K., \& Curtis, E. B. (1970). A spectral sequence for the homotopy of nice spaces. \textit{Transactions of the American Mathematical Society, 151}(2), 457-479.}

\bibitem[BKS05]{restrictedbasis}{Bryant, R. M., Kovács, L. G., \& Stöhr, R. (2005). Subalgebras of free restricted Lie algebras. \textit{Bulletin of the Australian Mathematical Society, 72}(1), 147-156.}



\bibitem[BMMS86]{bmms}{Bruner, R. R., May, J. P., McClure, J. E., \& Steinberger, M. (2006). \textit{H ring spectra and their applications} (Vol. 1176). Springer.}

\bibitem[BHK19]{bhk}{Brantner, L., Hahn, \& J.,  Knudsen, B. (2019). The Lubin-Tate Theory of Configuration Spaces: I. \textit{arXiv preprint arXiv:1908.11321}.}

\bibitem[Br17]{brantner}{Brantner, D. L. B. (2017). \textit{The Lubin-Tate theory of spectral Lie algebras} (Doctoral dissertation).}

\bibitem[BCN21]{pd}{Brantner, L., Campos, R., \& Nuiten, J. (2021). PD Operads and Explicit Partition Lie Algebras. \textit{arXiv preprint arXiv:2104.03870}.}

\bibitem[BM19]{bm}{Brantner, L., \& Mathew, A. (2019). Deformation theory and partition Lie algebras. To appear in \textit{Acta Mathematica}.}

\bibitem[Car58]{cartan}{ Cartan, H. (1958). \textit{Séminaire Henri Cartan: ann. 7 1954/1955; Algèbres d'Eilenberg-Maclane et homotopie.} Secrétariat mathématique.}

\bibitem[Chi05]{ching}{Ching, M. (2005). Bar constructions for topological operads and the Goodwillie derivatives of the identity. \textit{Geometry \& Topology, 9}(2), 833-934.}

\bibitem[CH19]{chingharper}{Ching, M., \& Harper, J. E. (2019). Derived Koszul duality and TQ-homology completion of structured ring spectra. \textit{Advances in Mathematics, 341}, 118-187.
}


\bibitem[CLM76]{adem}{Cohen, F. R., Lada, T. J., \& May, P. J. (2007). \textit{The homology of iterated loop spaces} (Vol. 533). Springer.}

\bibitem[DL62]{dyerlashof}{Dyer, E., \& Lashof, R. K. (1962). Homology of iterated loop spaces. \textit{American Journal of Mathematics, 84}(1), 35-88.}

\bibitem[Dri]{drinfeld}{Drinfeld, V. (2014). A letter from Kharkov to Moscow. \textit{EMS Surveys in Mathematical Sciences, 1}(2), 241-248.}

\bibitem[Dwy80]{dwyer}{Dwyer, W. G. (1980). Homotopy operations for simplicial commutative algebras. \textit{Transactions of the American Mathematical Society, 260}(2), 421-435.}

\bibitem[FG12]{fg}{Francis, J., \& Gaitsgory, D. (2012). Chiral Koszul duality. \textit{Selecta Mathematica, 18}(1), 27-87.}

\bibitem[Fre00]{fresse}{Fresse, B. (2000). On the homotopy of simplicial algebras over an operad. \textit{Transactions of the American Mathematical Society, 352}(9), 4113-4141.}

\bibitem[Fre04]{fresse2}{Fresse, B. (2004). Koszul duality of operads and homology of partition posets. In Homotopy theory: relations with algebraic geometry, group cohomology and algebraic K-theory, \textit{Contemp. Math. 346}, Amer. Math. Soc. pp. 112-215.}

\bibitem[GK94]{gk}{Ginzburg, V., \& Kapranov, M. (1994). Koszul duality for operads. \textit{Duke mathematical journal, 76}(1), 203-272.}

\bibitem[Goe90]{goerss}{Goerss, P. G. (1990). On the André-Quillen cohomology of commutative $\F$-algebras. \textit{Astérisque}, (1).}



\bibitem[Jac41]{jacobson}{Jacobson, N. (1941). Restricted Lie algebras of characteristic p.  \textit{Transactions of the American Mathematical Society, 50}(1), 15-25.}

\bibitem[Joh95]{johnson}{Johnson, B. (1995). The derivatives of homotopy theory. \textit{Transactions of the American Mathematical Society, 347}(4), 1295-1321.}

\bibitem[JN14]{jn}{Johnson, N. \&  Noel, J. (2014). Lifting homotopy T-algebra maps to strict maps. \textit{Advances in Mathematics, 264}, 593-645.}



\bibitem[Kja18]{kjaer}{Kjaer, J. J. (2018). On the odd primary homology of free algebras over the spectral Lie operad. \textit{Journal of Homotopy and Related Structures, 13}(3), 581-597.}

\bibitem[Kon23]{konovalov}{Konovalov, N. (2023). Algebraic Goodwillie spectral sequence. \textit{arXiv preprint arXiv:2303.06240.}}

\bibitem[KL83]{dual}{Kraines, D., \& Lada, T. (1983). The cohomology of the Dyer-Lashof algebra. In \textit{Proceedings of the Northwestern Homotopy Theory Conference} (pp. 145-152).}

\bibitem[Kri93]{kriz}{Kriz,Igor. (1993). Towers of $E_\infty$ ring spectra with an application to $BP$. Preprint.}


\bibitem[KA56]{kudoaraki}{Kudo, T., \& Araki, S. (1956). Topology of {$H_n$}-spaces and {$H$}-squaring operations. \textit{Memoirs of the Faculty of Science, Kyushu University. Series A, Mathematics, 10}(2), 85-120.}

\bibitem[Kuh85]{transfer}{Kuhn, N. J. (1985). Chevalley group theory and the transfer in the homology of symmetric groups. \textit{Topology, 24}(3), 247-264.}

\bibitem[Kuh01]{kuhncofib}{Kuhn, N. J. (2001). New relationships among loopspaces, symmetric products, and Eilenberg MacLane spaces. In \textit{Cohomological Methods in Homotopy Theory: Barcelona Conference on Algebraic Topology, Bellaterra, Spain, June 4–10, 1998} (pp. 185-216). Birkhäuser Basel.}

\bibitem[Law20]{lawson}{Lawson, T. (2020). $E_n$-spectra and Dyer-Lashof operations. In \textit{ Handbook of Homotopy Theory} (pp. 793-849). Chapman and Hall/CRC.}

\bibitem[Lur11]{lurie11}{Luire, J. (2011). Derived algebraic geometry X: Formal moduli problems, preprint available from the authors website.}

\bibitem[May72]{Einfinity}{May, J. P. (1972). \textit{The geometry of iterated loop spaces.} Lecture Notes in Math., Vol. 271,
SpringerVerlag,
BerlinNew
York, 1972.}

\bibitem[Lur17]{lurie}{Lurie, J. (2017). Higher algebra, preprint available from the authors website.}

%\bibitem[May72]{Einfinity}{May, J. P. (1972). \textit{The geometry of iterated loop spaces.} Lecture Notes in Math., Vol. 271, SpringerVerlag, Berlin NewYork, 1972.}

\bibitem[McC01]{mccleary}{McCleary, J. (2001). \textit{A user's guide to spectral sequences} (No. 58). Cambridge University Press.}

\bibitem[Mil16]{nishida}{Miller, H.R. Nishida relations and Singer construction. \url{https://klein.mit.edu/~hrm/papers/nishida3.pdf}.}


\bibitem[Nak57]{nak57}{Nakaoka, M. (1957). Cohomology mod p of the p-fold symmetric products of spheres. \textit{Journal of the Mathematical Society of Japan, 9}(4), 417-427.}

\bibitem[Nak58]{nak58}{Nakaoka, M. (1958). Cohomology mod $ p $ of symmetric products of spheres. \textit{Journal of the Institute of Polytechnics, Osaka City University. Series A: Mathematics, 9}(1), 1-18.}

\bibitem[Nis68]{nishidaodd}{Nishida, G. (1968). Cohomology operations in iterated loop spaces. \textit{Proceedings of the Japan Academy, 44}(3), 104-109.
}


\bibitem[Pri70]{priddy}{Priddy, S. B. (1970). Koszul resolutions. \textit{Transactions of the American Mathematical Society, 152}(1), 39-60.}

\bibitem[Pri73]{priddy2}{Priddy, S. (1973). Mod-p right derived functor algebras of the symmetric algebra functor. \textit{Journal of Pure and Applied Algebra, 3}(4), 337-356.}

\bibitem[Pri10]{pridham}{Pridham, J. P. (2010). Unifying derived deformation theories. \textit{Advances in Mathematics, 224}(3), 772-826.}

\bibitem[Rez12]{rezk}{Rezk, C. (2012). Rings of power operations for Morava {$E$}-theories are Koszul. \textit{arXiv preprint arXiv:1204.4831.}}

\bibitem[Sal98]{salvatore}{Paolo Salvatore, \textit{Configuration operads, minimal models and rational curves.}, Ph.D. thesis, University
of Oxford, 1998.}

\bibitem[Tak99]{takayasu}{Shinichiro Takayasu, On stable summands of Thom spectra of {$B(\mathbb{Z}/2)^n$} associated to Steinberg modules,
\textit{J. Math. Kyoto Univ. 39} (1999), no. 2, 377–398. }

\bibitem[Zh21]{me}{Zhang, A. Y. (2021). Quillen homology of spectral Lie algebras, with application to mod $p$ homology of labeled configuration spaces. To appear in \textit{Algebraic and Geometric Topology.}}

\end{thebibliography}
\end{document}